\documentclass[12pt,a4paper]{amsart}

\usepackage[colorlinks,pagebackref,hypertexnames=false]{hyperref}

\usepackage[a4paper,left=1.8cm,right=1.8cm,top=2.5cm,bottom=2.5cm]{geometry}

\usepackage{lmodern}
\usepackage{bbm}
\usepackage{float}
\usepackage{xurl}
\usepackage{ragged2e}

\usepackage{enumitem}
\setlist[itemize]{noitemsep, topsep=0pt, partopsep=0pt}
\setlist[enumerate]{noitemsep, topsep=0pt, partopsep=0pt}
\setlist[description]{noitemsep, topsep=0pt, partopsep=0pt}

\usepackage{etoolbox}
\AtBeginEnvironment{thebibliography}{%
  \setlength{\itemsep}{0pt}%
  \setlength{\parskip}{0pt}%
  \setlength{\topsep}{0pt}%
  \setlength{\partopsep}{0pt}%
}

\newlist{defenum}{enumerate}{3}
\setlist[defenum]{label=(\alph*),leftmargin=*,align=left}
\newlist{defsubenum}{enumerate}{1}
\setlist[defsubenum]{label=(\roman*),leftmargin=*,align=left}

\usepackage{amsmath,amssymb,amsthm}
\usepackage{graphicx}

\usepackage[nameinlink,noabbrev]{cleveref}
\usepackage{autonum}

\usepackage{amsmath,amssymb,amsthm} 
\usepackage{graphicx}
\usepackage{hyperref}                

\makeatletter
\renewcommand\normalsize{%
    \@setfontsize\normalsize{11.7}{14pt plus .3pt minus .3pt}%
    \abovedisplayskip 10\p@ \@plus4\p@ \@minus4\p@
    \abovedisplayshortskip 6\p@ \@plus2\p@
    \belowdisplayshortskip 6\p@ \@plus2\p@
    \belowdisplayskip \abovedisplayskip}
\renewcommand\small{%
    \@setfontsize\small{9.5}{12\p@ plus .2\p@ minus .2\p@}%
    \abovedisplayskip 8.5\p@ \@plus4\p@ \@minus1\p@
    \belowdisplayskip \abovedisplayskip
    \abovedisplayshortskip \abovedisplayskip
    \belowdisplayshortskip \abovedisplayskip}
\renewcommand\footnotesize{%
    \@setfontsize\footnotesize{8.5}{9.25\p@ plus .1pt minus .1pt}
    \abovedisplayskip 6\p@ \@plus4\p@ \@minus1\p@
    \belowdisplayskip \abovedisplayskip
    \abovedisplayshortskip \abovedisplayskip
    \belowdisplayshortskip \abovedisplayskip}
\ifdefined\pdfpagewidth
\else
\fi
\calclayout
\makeatother

\usepackage{mathtools,amssymb}
\usepackage[cal=rsfso,scaled=1.05]{mathalfa}
\usepackage[english]{babel}
\usepackage[T1]{fontenc}
\usepackage{tikz-cd}
\usepackage{bm}
\usepackage[colorlinks,pagebackref,hypertexnames=false]{hyperref}
\usepackage[nameinlink,noabbrev]{cleveref}
\usepackage{autonum}
\usepackage[alphabetic,backrefs]{amsrefs}

\usepackage{array,booktabs,tabularx,multirow}

\usepackage{enumitem}
\setlist{nosep}
\setlist[enumerate]{label=(\roman*)}

\hypersetup{
    colorlinks=true,
    linkcolor=red,
    citecolor=blue,
    filecolor=blue,
    urlcolor=blue
}

  {}              

\newtheorem{theorem}{Theorem}[section]
\newtheorem{Mtheorem}{Theorem}

\newtheorem{Ctheorem}[Mtheorem]{(Theorem)}

\newtheorem{corollary}[theorem]{Corollary}
\newtheorem{Mcorollary}[Mtheorem]{Corollary}

\newtheorem{lemma}[theorem]{Lemma}
\newtheorem{conjecture}[theorem]{Conjecture}
\newtheorem{proposition}[theorem]{Proposition}
\theoremstyle{definition}
\newtheorem{definition}[theorem]{Definition}
\newtheorem{example}[theorem]{Example}
\theoremstyle{remark}
\newtheorem{remark}[theorem]{Remark}

\crefname{theorem}{Theorem}{Theorems}
\crefname{Mtheorem}{Main Theorem}{Main Theorems}
\crefname{lemma}{Lemma}{Lemmata}
\crefname{corollary}{Corollary}{Corollaries}
\crefname{proposition}{Proposition}{Propositions}
\crefname{definition}{Definition}{Definitions}

\newcommand{\Q}{\mathbb{Q}}
\newcommand{\R}{\mathbb{R}}
\newcommand{\Z}{\mathbb{Z}}
\newcommand{\C}{\mathbb{C}}

\newcommand{\PP}{\mathbb{P}}

\newcommand{\rank}{\mathrm{rk}}

\newcommand{\Perf}{\mathrm{Perf}}

\newcommand{\rH}{\mathrm{H}}

\newcolumntype{L}{>{$}l<{$}}
\newcolumntype{C}{>{$}c<{$}}
\newcolumntype{R}{>{$}r<{$}}

\newcommand{\lp}{\textup{\texttt{+}}}
\newcommand{\lm}{\textup{\texttt{-}}}
\begin{document}


\title{Atoms meet symbols}


\author[Cavenaghi]{Leonardo F. Cavenaghi}
\address{Institute of Mathematics and Informatics, Bulgarian Academy of Sciences
 \& College of Arts and Sciences, Department of Mathematics, University of Miami, Ungar Bldg, 1365 Memorial Dr 515, Coral Gables, FL 33146, USA}
\email{leonardofcavenaghi@gmail.com}


\author[Katzarkov]{Ludmil Katzarkov}
\address{College of Arts and Sciences, Department of Mathematics, University of Miami, \& Institute of the Mathematical Sciences of the Americas (IMSA). Ungar Bldg, 1365 Memorial Dr 515, Coral Gables, FL 33146. \& International Center for Mathematical Sciences (ICMS), Sofia-Bulgaria.} 
\email{lkatzarkov@gmail.com}

\author[Kontsevich]{Maxim Kontsevich}
\address{Institut des Hautes Études Scientifiques,
35 route de Chartres, F - 91440 Bures-sur-Yvette}
\email{maxim@ihes.fr}

\setcounter{tocdepth}{1}

\begin{abstract}
This paper introduces a novel framework for constructing invariants in $G$-equivariant birational geometry by unifying two recent approaches: the \emph{theory of atoms} developed by Katzarkov, Kontsevich, Pantev, and Yu, and the theory of \emph{modular symbols} due to Kontsevich, Tschinkel, and Pestun. 


We initiate the theory of Chen-Ruan atoms. Assuming the blowup formula for the quantum Chen-Ruan cohomology, we outline how to extend the theory of atoms to global quotient orbifolds and present some striking applications.

In addition, we develop a separate class of purely geometric invariants for $\mathbb{Z}/2$- and $\mathbb{Z}/3$-actions on surfaces and threefolds.

We provide many examples of non-$G$-linearizable $G$-actions on projective varieties treated with these new techniques.
\end{abstract}

\maketitle

\section{Introduction}

Let \( X \) be a smooth projective variety of dimension \( n \) over a field \( k \) of characteristic zero. Its function field, denoted \( k(X) \), consists of rational functions defined on \( X \). The variety $X$ is called \emph{rational} if it is birationally equivalent to \( \mathbb{P}^n \), meaning that \( k(X) \) is isomorphic to the purely transcendental field \( k(x_1, \ldots, x_n) \). 

Given two smooth projective varieties \( X \) and \( Y \) of dimension \( n \) over a field \( k \) with characteristic zero, we say that they are \emph{birationally equivalent} if there exists an isomorphism between their function fields, \( k(X) \cong k(Y) \) identical on $k$. Experts have devoted significant efforts to the challenging task of constructing explicit birational morphisms between smooth projective varieties or of producing new birational invariants (cf. \cites{kresch2026invariantsequivariantbirationalgeometry,bira-book-2020, birkar2017birationalgeometryalgebraicvarieties, birkar2012lecturesbirationalgeometry, bira-book-2023, Voisin_2002, Voisin2014, Voisin2016StableBI, Voisin2019}). In this paper, we will always assume $k=\C$. The whole story can be generalized to non-algebraically closed fields of characteristic zero using a version of the motivic Galois group due to Y.Andr\'{e}, see \cite{katzarkovpantevyu} for more details.

\  

Assume that a finite group \( G < \mathrm{Aut}(X) \) acts generically freely on a smooth projective variety \( X \), where \( \mathrm{Aut}(X) \) denotes the group of automorphisms of \( X \) -- herein we say that $X$ has a generically free regular $G$-action and say that such an $X$ is a $G$-variety. Two $G$-varieties \( X \) and \( Y \) of the same dimension are said to be \( G \)-birationally equivalent if there exists a \( G \)-equivariant birational map between them.\footnote{It makes little (if any) sense to say that a variety \( X \) is \( G \)-rational. This is because some projective spaces \( \mathbb{P}^n \) can have two non-\( G \)-equivariantly birational \( G \)-actions where \( G < \mathrm{Aut}(\mathbb{P}^n)\simeq \mathrm{PGL}(n+1) \), see \cite{Reichstein2002}.} Unsurprisingly, $G$-equivariant birational geometry and birational geometry over non-algebraically closed fields are closely related, see the Appendix \ref{ap:specialization}.

Some commonly used \( G \)-birational invariants include the existence of fixed points upon restriction to abelian subgroups \( H < G \); group cohomology \cites{Bogomolov2013, Prokhorov}; and the \( G \)-equivariant Minimal Model Program. However, significant developments on the subject began to emerge after the introduction of certain \emph{modular symbols}. The following construction is due to Kontsevich, Pestun, and Tschinkel (cf. \cite{kontsevich2021equivariant}).

Assume that the finite group \(G\) is \textit{abelian}, and let its dual group be denoted by \(G^{\vee} = \mathrm{Hom}(G, \mathbb{C}^{\times})\). For \(n \geq 2\), we consider the \(\mathbb{Z}\)-module \(\mathcal{B}_n(G)\), which is generated by symbols \([a_1, \ldots, a_n]\) for \(\{a_i \in G^{\vee}\}_{i=1}^n\) satisfying \( \sum_{i=1}^n \mathbb{Z} a_i = G^{\vee}\), subject to the following relations:
\begin{itemize}
    \item[(O)] \([a_1, \ldots, a_n] = [a_{\sigma(1)}, \ldots, a_{\sigma(n)}]\) for any permutation \(\sigma \in {\mathcal S}_n\).
    \item[(B)] for all \(2 \leq k \leq n\), all \(a_1, \ldots, a_k \in G^{\vee}\), and all \(b_1, \ldots, b_{n-k} \in G^{\vee}\) such that
    \[
    \sum_{i=1}^k \mathbb{Z} a_i + \sum_{j=1}^{n-k} \mathbb{Z} b_j = G^{\vee},
    \]
    we have:
    \[
    [a_1, \ldots, a_k, b_1, \ldots, b_{n-k}] = \sum_{1 \leq i \leq k, a_i \neq a_{i'}\,\forall i'<i} [a_1 - a_i, \ldots, a_i, \ldots, a_k - a_i, b_1, \ldots, b_{n-k}].
    \]
\end{itemize}

For a given smooth irreducible projective variety \(X\) of dimension \(n \geq 2\) over \(\mathbb{C}\) that has a regular generically free action of a finite abelian group \(G\), we associate an element of \(\mathcal{B}_n(G)\) in the following way: Let
\[
X^G = \bigsqcup_{\alpha \in A} F_{\alpha}
\]
denote the decomposition of the $G$-fixed point locus of $X$ into a disjoint union of closed smooth irreducible subvarieties of \(X\). For each irreducible component \(F_{\alpha}\), we fix a point \(x_{\alpha} \in F_{\alpha}\) and consider the action of \(G\) on the tangent space \(T_{x_{\alpha}}X\). This yields a decomposition into eigenspaces corresponding to some characters \(a_{1,\alpha}, \ldots, a_{n,\alpha}\):
\[
T_{x_{\alpha}} X = \bigoplus_{i=1}^{n} E_{i,\alpha},
\]
with the indices defined up to permutation. For each \(\alpha\), we can thus associate a symbol:
\[
[a_{1,\alpha}, \ldots, a_{n,\alpha}] \in \mathcal{B}_n(G).
\]
Finally, we define 
\[
\beta(X) := \sum_{\alpha} [a_{1,\alpha}, \ldots, a_{n,\alpha}].
\]
One of the main results in \cite{kontsevich2021equivariant} is that \(\beta(X)\) is a \(G\)-birational invariant of \(X\). Following the Weak Factorization Theorem (Theorem \ref{thm:weak}), this is obtained by showing that \(\beta(X)\) remains invariant under \(G\)-equivariant blowups: i.e., $\beta(X)$ remains invariant under blowups in closed smooth \(G\)-invariant subvarieties\footnote{An alternative name is \textit{\(G\)-stable} subvariety.} \(Z \subset X\) of codimension \(r \geq 2\).

Through the lens of this new invariant and its generalizations, many pairs of \( G \)-varieties are non-\( G \)-equivariantly birational, as demonstrated in a series of works \cites{Tschinkel2024, kresch2019birational, kresch2022burnside, kresch2023birational}. A central idea in these works is the creation of new \( \mathbb{Z} \)-modules whose generators are constrained by blowup relations. For a quick recall and more details, refer to Section \ref{sec:atomes-meet-symbols}.

\

Recall that for a smooth projective variety \(X\), its genus-zero Gromov-Witten invariants define the quantum product \(\star_s\) on the Betti cohomology \(\mathrm{H}:=\mathrm{H}^{\ast}(X, \mathbb{Q})\), deformed by a parameter \(s \in\mathrm{H}^{\text{even}}(X,\mathbb Q)\). The ``theory of atoms'', developed in \cite{katzarkovpantevyu}, offers a new approach to constructing birational invariants using this quantum multiplication. See \cites{guere2026irrationalitycubicfourfolds, benedetti2026quantumcohomologyirrationalitygushelmukai,ElaginSchneiderShinder} for additional developments.

The central invariants, called \emph{Hodge atoms}, are pieces obtained from an intrinsic decomposition of the \emph{A-model F-bundle} associated with \(X\) (Section \ref{sec:A-model}). This F-bundle is a non-archimedean version of a variation of non-commutative Hodge structures \cite{KatzarkovKontsevichPantev2008}. The F-bundle combines enumerative data from Gromov-Witten theory with the classical polarized \(\mathbb{Z}/2\)-weighted \(\mathbb{Q}\)-Hodge structure on the Betti cohomology \(\mathrm{H}\) of \(X\).

The extraction of atoms relies on a spectral decomposition of the A-model F-bundle (for more details, refer to Section \ref{sec:global-hodge-atoms}). At a general point within the locus of Hodge classes inside the F-bundle's base, the generalized eigenspaces for the quantum multiplication $\star$ by \(\boldsymbol{\kappa}:=\mathsf{Eu}\star\) (where $\mathsf{Eu}$ is the \emph{Euler vector field} (Equation \eqref{eq:Euler})) are compatible with the Hodge structure on \(\mathrm{H}\). Concretely, each generalized eigenspace decomposes canonically into \(\overline{ \mathbb{Q}}\)-linear representations of the Mumford-Tate group \(\mathsf{Hod}\) (the group for \(\mathbb{Z}/2\)-graded polarizable pure Hodge structures), termed \emph{local Hodge atoms}. These local Hodge atoms can be tracked through numerical invariants (see Equation \eqref{invariants-box}). Understanding how the collection of atoms behaves under birational transformations, particularly blowups (following the work of Iritani \cite{iritani2023quantumcohomologyblowups}), provides a powerful obstruction to rationality (see, e.g., Theorem \ref{thm:Katzarkov-Kontsevich-Pantev-Yu} due to Katzarkov, Kontsevich, Pantev, and Yu).

\ 

In the first part of this paper (Sections \ref{sec:theory-of-atoms} and \ref{sec:atomes-meet-symbols}), we extend the theory of atoms to the \(G\)-equivariant setting and unify it with the theory of modular symbols. This unification is motivated by the observation that these two perspectives---the ``atomic'' and the ``symbolic''---are complementary facets of a deeper structure. While the theory of \(G\)-equivariant atoms provides powerful global invariants derived from the interplay between Gromov-Witten theory and the Mumford-Tate group action on the total cohomology of a variety \(X\), the theory of modular symbols captures the local geometry of the action of a finite group $G$, by analyzing by analyzing characters of \(G\)-representations on the normal bundles to the fixed locus \(X^G\). Since a \(G\)-equivariant blowup simultaneously alters both the global atomic content and the local fixed-locus geometry, a complete invariant must track both transformations.

We construct a unified algebraic object: a combinatorial \(\mathbb{Z}\)-module, denoted \(\mathcal{B}_3(\mathbb{Z}/p)^{\text{comb}}\), whose generators are a formal union of atoms and symbols. The relations defining this module are those given by \(G\)-equivariant blowups. From this module, we derive two new \(\mathbb{Z}\)-valued \(\mathbb{Z}/p\)-birational invariants.

\

Recall that any smooth hypersurface $X(1,1,1,1)$ of multi-degree $(1,1,1,1)$ in
\((\mathbb{P}^1)^4\) is rational. The projection to the first three factors identifies it with the blowup of  \((\mathbb{P}^1)^3\) in an elliptic curve. We prove:
\begin{Mtheorem}[=Theorem \ref{thm:X(1,1,1,1)-withZ_2}]\label{Mthm:X(1,1,1,1)-withZ_2}
    Assume that \( X(1,1,1,1) \) is invariant under the \(\mathbb Z/2\)-action induced by swapping the first and second \(\mathbb{P}^1\) factors and the third and fourth \(\mathbb{P}^1\) factors. Then, there is no \(\mathbb Z/2\)-linearizable action on \(\mathbb{P}^3\) making it \(\mathbb Z/2\)-equivariantly birational to \( X(1,1,1,1) \).
\end{Mtheorem}

Interestingly enough, the following dichotomy holds:
\begin{Mcorollary}[=Corollary \ref{thm:X(1,1,1,1)}]\label{cor:general}
    For any subgroup $G<\mathcal S_4$-action on $X(1,1,1,1)$ induced from a transitive action in $\{1, 2, 3, 4\}$ either
\begin{itemize}
\item[$(a)$] the $G$-action on the set $\{1,2,3,4\}$ has a fixed point, and in this case $X(1,1,1,1)$ is $G$-birationally equivalent to $\mathbb P^3$ with a $G$-action
\item[$(b)$] the $G$-action on the set $\{1,2,3,4\}$ has no fixed points, and in this case $X(1,1,1,1)$ is not $G$-birationally equivalent to $\mathbb P^3$ with any $G$-action.
\end{itemize}
\end{Mcorollary}

Combining Theorem \ref{Mthm:X(1,1,1,1)-withZ_2} with the work of Kuznetsov and Prokhorov \cite{Kuznetsov_Prokhorov_2024} we obtain:
    
\begin{Mcorollary}[=Corollary \ref{cor:Kuznetsov}]\label{cor:Kuznetsov-Intro}
        Let $X$ be a singular cubic threefold with four ordinary double points. Assume that $\mathcal S_4$ acts on $X$ exchanging these ordinary double points. Then $X$ is not birationally equivalent to $\mathbb P^3$ with \emph{any} $\mathcal S_4$-linearizable action.
\end{Mcorollary}

Y. Tschinkel pointed out to us that Theorem \ref{Mthm:X(1,1,1,1)-withZ_2} and Corollaries \ref{cor:general}, \ref{cor:Kuznetsov-Intro} answer affirmatively the conjecture presented on p. 3 in \cite[Case $s=4,~d=p=0$]{cheltsov2024equivariant}. Namely,
\begin{Mtheorem}
    Let $X$ be a singular cubic threefold with four ordinary double points and a $G$-action. Assume that $\mathrm{rk}\mathsf{Cl}(X)=1$ where $\mathsf{Cl}(X)$ stands to the divisor class group of $X$. Then the $G$-action is non-linearizable if the $G$-action is transitive on the double points.
\end{Mtheorem}

In another direction, we investigate ``stable $G$-equivariant birationality'' regarding products with $\mathbb P^1$. Meaning, if one knows that a $G$-variety $X$ is non-$G$-equivariantly birational to some projective space with any linearizable action, we could ask whether $X\times \mathbb P^1$ is $ G$-equivariantly birational to some projective space with a linearizable $G$-action. We have (compare with Proposition 2.1 in \cite[p. 229]{Tschinkel-Yang})
\begin{Mtheorem}[=Example \ref{ex:cubic-surface}]
     Let $X$ be a smooth cubic surface with a finite group $G$ action by regular automorphisms. Assume that $
    \mathrm{rank}~(\mathrm{Pic}(X))^{G}=1$. Then $X$ is not $G$-equivariantly birational to $\mathbb P^2$ with any $G$-linearizable action. Moreover, $X\times \mathbb P^1$ with the trivial $G$-action on the second factor is not $G$-equivariantly birational to $\mathbb P^3$ with any $G$-linearizable action.
\end{Mtheorem}

Theorem \ref{thm:X(1,1,1,1)} says that $X(1,1,1,1)$ with the $\mathbb Z/2$-action induced from permuting first and second and third and fourth factors of $(\mathbb P^1)^4$ is not $\mathbb Z/2$-equivariantly birational to $\mathbb P^3$ with any $\mathbb Z/2$-linearizable action. We show (compare with \cite[Proposition 5.3]{cheltsov2024equivariant}):
\begin{Mtheorem}[=Example \ref{ex:multi-degree-times}]
      Assume that \( X(1,1,1,1) \) is invariant under the  $\mathbb Z/2$-action induced from permuting first and second and third and fourth factors of $(\mathbb P^1)^4$. Consider the $\mathbb Z/2$-variety $X(1,1,1,1)\times \mathbb P^1$ with trivial $\mathbb Z/2$-action on the $\mathbb P^1$-factor. Then $X(1,1,1,1)\times \mathbb P^1$ is not $\mathbb Z/2$-birationally equivalent to $\mathbb P^4$ with \emph{any} $\mathbb Z/2$-linearizable action.
\end{Mtheorem}

The above results are all proven only via equivariant atom theory (i.e., by directly understanding the invariants in Equation \eqref{invariants3-box} in Section \ref{sec:equivariant-atoms}). The combination ``atoms meet symbols'' is manifested through the following:

\begin{Mtheorem}[=Theorem \ref{thm:invariant}]\label{ithm:atoms-meet-symbols}
    Let \(X\) be an irreducible smooth projective threefold equipped with a regular generically free \(\mathbb{Z}/p\)-action, where \(p\) is prime. Fix an integer \(\mathsf{g} \geq 2\), and denote by \(X^{\mathbb{Z}/p}\) the fixed-point locus of the action.  Then the following quantity defines a \(\mathbb{Z}/p\)-birational invariant:
\begin{equation}
\begin{split}
  &-\;\#\bigl\{
    \text{1-dimensional components of } X^{\mathbb{Z}/p} 
    \text{ of genus } \mathsf{g} 
  \bigr\} \\
  &\quad -\; 2\,\#\bigl\{
    \text{components of $X^{\mathbb{Z}/p}$ birational to the product of a curve of genus } \mathsf{g} 
    \text{ and } \mathbb{P}^1 
  \bigr\} \\
  &\quad +\; \#\Bigl\{ 
    \mathbb{Z}/p\text{-equivariant atoms } \boldsymbol\alpha 
    \;\Bigm|\; P_{\boldsymbol\alpha}(t) = \mathsf{g} t^{-1} + 2 + \mathsf{g} t, \\
  &\hspace{5em} 
    \text{and } \mathbb{Z}/p \text{ acts trivially on the corresponding eigenspace}
  \Bigr\}.
\end{split}
\end{equation}

\end{Mtheorem}
In Theorem \ref{ithm:atoms-meet-symbols}, the Hodge polynomial $P_{\boldsymbol\alpha}(t)\in \mathbb Z[t,t^{-1}]$ is an invariant of a Hodge atom, see the formula \eqref{invariants-box} in Section \ref{sec:atoms-and-birational-geometry}. A ``classical''\footnote{i.e., without using atoms} analog of this theorem appears in Proposition 5.2 in \cite{kresch2025intermediatejacobiansburnsideinvariants}, which appeared after the first version of our paper was made public.

We have a finer invariant, depending on the isogeny class of Jacobians of genus $\mathsf  g$ curves.
\begin{Mtheorem}[=Theorem \ref{thm:invariant-fine}]\label{ithm:atoms-meet-symbols2}
In the assumptions of the previous theorem, let us assume
    that we fix a curve $C_0$ of genus \(\mathsf g\geq 2\). Then the following quantity defines a \(\mathbb{Z}/p\)-birational invariant:
\begin{equation}
\begin{split}
  &-\;\#\bigl\{
    \text{components of } X^{\mathbb{Z}/p} \cong C  
  \bigr\} 
   -\; 2\,\#\bigl\{
    \text{components of $X^{\mathbb{Z}/p}$ birational to } C \times \mathbb{P}^1  
  \bigr\} \\
  &\quad +\; \#\Bigl\{
    \mathbb{Z}/p\text{-equivariant atoms } \boldsymbol\alpha 
    \;\Bigm|\; 
    \mathbb{Z}/p \text{ acts trivially on the corresponding eigenspace,} \\
  &\qquad\quad 
    \text{and the isomorphism class of the } \bar{\mathbb{Q}}\text{-linear } 
    \mathsf{Hod}\text{-representation associated to } \boldsymbol\alpha \\
  &\qquad\quad 
    \text{is equivalent to that on } \mathrm{H}^{\ast}(C_0)
  \Bigr\}.
\end{split}
\end{equation}
Here, $C$ denotes a curve such that $\operatorname{Jac}(C)$ is isogenous to $\operatorname{Jac}(C_0)$.
\end{Mtheorem}

In Theorem \ref{ithm:atoms-meet-symbols2}, $\mathsf{Hod}$ denotes the pro-reductive group over $\mathbb Q$ closely related to the usual Mumford-Tate group, see Section  \ref{sec:Hodge-Z_p}.

Using this theorem, we can construct new examples of threefolds with cyclic group actions which are not equivariantly birational to $\mathbb P^3$ with any linear action. These examples are sharp, in the sense that the ``atoms meet symbols'' framework obstructs equivariant birational equivalences for them, while none of the known purely symbolic invariants (a la Kontsevich-Pestun-Tschinkel) obstruct equivariant birationality in these cases.
\begin{Mtheorem}[=Example \ref{ex:fixed-higher-genus-curve}]\label{Mtheorem:examples-higher-genus}
For every monic polynomial \(P \in \mathbb{C}[x]\) of degree \(d = 3k,~k>1\) with only simple roots, there exists a smooth compactification \(\widetilde{X}\) of the affine \(\mathbb{Z}/3\)-variety
\[
X: x_1x_2x_3 = P(x_4),
\]
equipped with a generically free regular \(\mathbb{Z}/3\)-action given by cyclically permuting the variables \(x_1, x_2, x_3\). The fixed locus of this action is a curve of genus \(\mathsf{g} \geq 2\), and \(\widetilde{X}\) is not \(\mathbb{Z}/3\)-birationally equivalent to \(\mathbb{P}^3\) with any linear \(\mathbb{Z}/3\)-action.
\end{Mtheorem}
Theorem \ref{Mtheorem:examples-higher-genus} extends the case $d=3$, which is treated in Proposition 6.5 in \cite{cheltsov2025equivariantgeometrysingularcubic}.

A family of other examples can be straightforwardly derived from this technology (Theorem \ref{Mtheorem:product=with-line} below). From classical methods, they can be recovered via the results in \cite{Bogomolov2013}. See also the discussion in Section 6 in \cite{kresch2025intermediatejacobiansburnsideinvariants}.
\begin{Mtheorem}[=Example \ref{ex:product=with-line}]\label{Mtheorem:product=with-line}
     Let $p$ be a prime number. Let $S$ be a smooth rational surface endowed with a nontrivial action of $\mathbb{Z}/p$ such that its fixed locus $S^{\mathbb{Z}/p}$ contains a curve of genus $\mathsf{g} \ge 2$. Let $X = S \times \mathbb{P}^1$ be endowed with the diagonal $\mathbb{Z}/p$-action, for any given $\mathbb{Z}/p$-action on $\mathbb{P}^1$. Then the variety $X$ is not $\mathbb{Z}/p$-birationally equivalent to the projective space $\mathbb{P}^3$ endowed with a linearizable $\mathbb{Z}/p$-action. 
\end{Mtheorem}

Initially, we found it challenging to apply atom theory to recover Theorems \ref{Mtheorem:examples-higher-genus} and \ref{Mtheorem:product=with-line}. However, this can be more or less straightforwardly obtained from a very nontrivial combination of atom theory with classical monodromy-type results:

\begin{Mtheorem}[=Theorem \ref{thm:classical-monodromy}]\label{Mtheorem:monodromy}
    Let $(X_t)_{t\in B}$ be an algebraic family of compact algebraic 3-dimensional varieties endowed with a generically free action of a finite group $G$, over a connected base $B$. Assume that 
\begin{enumerate}
    \item the action of $\mathsf{Hod}$ on $\mathrm H^\ast(X_t,\Q)$ is trivial for all $t\in B$, in other words, all rational cohomology classes of $X_t$ are Hodge classes,
    \item  the local system $\mathrm H^\ast(X_t,\Q)^G$ on $B$ is trivial,
    \item there exists an irreducible representation $\rho$ of the local system $\mathrm H^\ast(X_t,\overline \Q)$ of dimension strictly larger than the cardinality $|G|$ of $G$.
\end{enumerate}
    Then, for any $t\in B$, the $G$-variety $X_t$ is not $G$-equivariantly birationally equivalent to $\mathbb P^3$ endowed with a generically free linearizable $G$-action.
\end{Mtheorem}

Theorem \ref{Mtheorem:monodromy} is just the first of many possible deep results coming from the combination of this A-side B-side\footnote{Although not advanced in this manner, atom theory is developed from insights from mirror symmetry, see, e.g., Section \ref{sec:ideal-legal}} Hodge theory. Other results shall appear elsewhere \cite{CKK:atom_theory}.

\ 

In Section \ref{sec:CR}, we develop the theory of atoms enhanced with information from Chen-Ruan cohomology  \cite{Chen2004}. Namely, each $G$-variety $X$ defines an orbifold $[X/G]$. We can use the atoms coming from the cohomology of the \emph{twisted sectors}, which in the case of $G=\mathbb Z/p$ coincide with the fixed points set for the $G$-action, to produce finer invariants of $G$-equivariant atoms. We assume the blowup formula for the Chen-Ruan quantum cohomology is valid; see Conjectures \ref{conj:1}-\ref{conj:3}. Applications include a new treatment of all of the already considered examples, as well as some totally new unknown results, which we briefly mention in what follows. 

\begin{Ctheorem}[=Corollary \ref{cor:rough}]\label{MTheorem:all-dimensions}
      For a $d$-dimensional smooth projective variety $X$ endowed with $G$-action, if 
$$\mathrm{Coeff}_{t^{d-2}}P_{X^g}> \mathrm{Coeff}_{t^{d-2}}P_{X}$$ 
 then $X$ is not $G$-equivariantly birational to a projective space with a $G$-linearizable action, where for each variety $Y$ the polynomial $P_Y(t)$ is the Hodge polynomial
 \[P_Y:=\sum_{p,q\ge 0}h^{p,q}(Y)t^{p-q}\]
 and for each $g\in G$, $X^g$ stands to the fixed locus under the action of $g$.
\end{Ctheorem}

Theorem \ref{MTheorem:all-dimensions} allows us to present the first use of atom theory to obstruct $G$-equivariant birationality with projective spaces with linearizable actions without dimensional constraints. A concrete application is:

\begin{Ctheorem}[=Example \ref{ex:family-higher}]
    For any $d\ge 3$, let $X$ to be the $d$-dimensional affine variety
    $$x_1 x_2 x_3=P(x_4,\dots,x_{d+1})$$
   where $P(x_3,\dots, ,x_{d+1})$ is a  polynomial of degree $3k$ where $2k\ge d-1$, defining a smooth hypersurface in the projective space $\mathbb P^{d-2}$. Let $\Z/3$ act on $X$ by cyclic permutations of $x_1,x_2,x_3$.

Then,
\begin{enumerate}
 \item there is a smooth compactification $\widetilde X$ of $X$ for which the $\mathbb Z/3$ action extends
\item the fixed locus is a smooth compactification of the affine variety given by the equation
 $$x^3=P(x_4,\dots,x_{d+1})$$
\item $\widetilde X$ with this $\mathbb Z/3$-action is not birationally equivalent to $\mathbb P^d$ with any $\mathbb Z/3$-linearizable action.
\end{enumerate}
\end{Ctheorem}

A result about smooth complete intersection of quadrics with involutions is also obtained (Theorem \ref{thm:complete-intersection} below). Notice that a criterion for the linearizability of $G$-actions on three-dimensional complete intersections of two quadrics in $\mathbb P^5$ appears in Theorem 24 in \cite{Hassett2022}. See also the discussion below the statement.

\begin{Ctheorem}[=Example \ref{ex:complete-intersection-odd}]\label{thm:complete-intersection}
    There exists a $\mathbb Z/2$-action on the three-dimensional smooth complete intersections of two quadrics in $\mathbb P^5$ making it non $\mathbb Z/2$-birationaly equivalent to $\mathbb P^3$ with any $\mathbb Z/2$-linearizable action.
\end{Ctheorem}

\ 

It is also noteworthy that the theory should be extrapolated to answer questions regarding birational types of Deligne-Mumford stacks, as in \cites{kresch2019birational, kresch2023birational}. An eventual combination of these ``Chen-Ruan atoms'' and other types of symbol groups should also be achievable. Aside from that, it will be interesting to see if Ruan’s crepant resolution
conjecture \cite{ bryan2007crepantresolutionconjecture} could be proved with this technology -- recall, for instance, that a classical result of Batyrev \cite[Corollary 6.29]{Batyrev:1997hj} about Hodge numbers of birational Calabi-Yau manifolds was reproved in \cite{katzarkovpantevyu} using atom theory instead of, e.g., motivic integration.

\ 

The last part of this paper (Section \ref{sec:geometric}) is disconnected from the atom theory. Noticing that the modular symbols invariant only keep track of birational types of connected components $F_\alpha$ of the locus of fixed points, it is natural to try to exploit further data, like actual isomorphism classes of $F_\alpha$, together with the various weight components of the normal bundle. We obtain the following results, which also serve as a showcase of the strength of atom theory, as they do not produce obstructions in many reasonable cases (for further details, see the discussion at the beginning of Section \ref{sec:geometric}). More concretely, the invariant in Theorem \ref{ithm:exotic-invariant} does not recover Theorem \ref{Mtheorem:examples-higher-genus}.

\begin{Mtheorem}[=Theorem \ref{thm:pure-geometry}]\label{ithm:exotic-invariant}
Let $X$ be a smooth projective threefold with a generically free $\mathbb Z/3$‐action.  Then the following quantity is invariant under \(\mathbb{Z}/3\)-equivariant blowups:
\begin{equation}
  J = 
    \sum_{[\lp\lp\lm],\; [\lp\lm\lm]} 1
    + 
    \sum_{(C,[\lp\lp]),\; (C,[\lm\lm])} \left(1 - \mathsf{g} + d\right)
    +
    \sum_{(C,[\lp\lm])} \left(2 - 2\mathsf{g} + d\right)
    +
    \sum_{S} \left(3 - 3\mathsf{g} - K_S\cdot \mathcal N\right)
\end{equation}
where the summation runs over fixed point set components $X^{\mathbb Z/3}$. Here:
\begin{itemize}
  \item 
   $C$ is a curve with genus \(\mathsf{g}=\mathsf{g}(C) \geq 0\).
  \item 
    \(d \in \mathbb{Z}\) is the degree of \(\wedge^{2}(\mathcal{N}_XC)\), where \(\mathcal{N}_XC\) is the normal bundle of $C$ in $X$.
  \item $S$ is birational to a ruled surface whose base of the ruling is a curve with genus 
    \(\mathsf{g} \geq 0\) (slightly abusing the notation, we write $\mathsf g=\mathsf g(S)$).
  \item 
    \(K_{S}\) is the canonical bundle of the surface \(S\).
  \item 
    \(\mathcal{N}=\mathcal N_X S\) is the normal bundle of \(S\) in $X$.
\end{itemize}
The signs decorations are the non-zero characters for \(\mathbb Z/3 = (\mathbb Z/3)^{\vee}\), i.e.,
\[
  \lp,\,\lm \;=\; 1,\,2 \pmod{3}.
\]
so that
\begin{itemize}
    \item $[\lp\lp\lm], [\lp\lm\lm]$ are the normal weights for the $\mathbb Z/3$-normal bundle representation at fixed points
    \item for symbols like $(C,[\lp\lp]), (C,[\lp\lm]), (C,[\lm\lm])$ the terms $[\lp\lp], [\lm\lm], [\lp\lm]$ stand to the normal weights for the $\mathbb Z/3$-normal bundle representation at the fixed curve $C$.
\end{itemize}\end{Mtheorem}

Although we do not present this explicitly here, Theorem \ref{ithm:exotic-invariant} is obtained by constructing a new $\mathbb Z$-module in the spirit of the construction of Kontsevich, Pestun, and Tschinkel, but keeping track not only of the characters but of every discrete parameter related to the geometry of the fixed point set on smooth projective threefolds with generically free regular $\mathbb Z/3$-actions. We prove it directly by verifying the invariance of $J$ under all the possible blowups in smooth invariant centers. 

In Section \ref{sec:dimension-3-2}, we provide the following invariant for $\mathbb Z/2$-generically free actions on smooth projective threefolds. It can be seen as a complementary result to Theorem \ref{ithm:exotic-invariant} in the sense that it is derived through the same reasoning.

\begin{Mtheorem}[=Theorem \ref{thm:pure-geometry-2d}]
       Let $X$ be a smooth projective threefold with a generically free $\mathbb Z/2$‐action.  Then the following quantity is invariant under \(\mathbb{Z}/2\)-equivariant blowups:
\begin{equation}
  K = 
    \#\text{fixed points}
    + 
    \sum_{C} \left(2 - 2\mathsf{g} + d\right)
    +
    \sum_{S} \left(4 - 4\mathsf{g} - K_S\cdot \mathcal N\right)
\end{equation}
where the summation runs over fixed point set components $X^{\mathbb Z/2}$. Here:
\begin{itemize}
  \item 
   $C$ is a curve with genus \(\mathsf{g}=\mathsf{g}(C) \geq 0\).
  \item 
    \(d \in \mathbb{Z}\) is the degree of \(\wedge^{2}(\mathcal{N}_XC)\), where \(\mathcal{N}_XC\) is the normal bundle of $C$ in $X$.
  \item $S$ is birational to a ruled surface whose base of the ruling is a curve with genus 
    \(\mathsf{g} \geq 0\) 
  \item 
    \(K_{S}\) is the canonical bundle of the surface \(S\).
  \item 
    \(\mathcal{N}\) is the normal bundle of \(S\) in $X$.
\end{itemize}
\end{Mtheorem}

\ 

In Section \ref{sec:dimension-2}, we present an essentially geometric/topological invariant in dimension two. It arises from the combination of the same detailed geometry of fixed point sets for generically free regular $\mathbb Z/p$-actions on surfaces, while also keeping track of atoms. The surprise comes because it has no atomic information at all.

\begin{Mtheorem}[=Theorem \ref{thm:degrees}]
Let \( X \) be a smooth projective surface with a regular generically free \( \mathbb{Z}/2 \)-action. Then the following quantity is a \( \mathbb{Z}/2 \)-birational invariant:
\[
I:=\chi(X^{\mathbb{Z}/2}) + \sum_{\alpha:\dim F_\alpha=1} \deg\big( \mathcal{N}_X F_\alpha \big),
\]
where the sum runs over all 1-dimensional irreducible components \( F_\alpha \subset X^{\mathbb{Z}/2} \) of dimension $1$, and \( \mathcal{N}_X F_\alpha \) denotes the normal bundle of curve \( F_\alpha \) in \( X \).
\end{Mtheorem}

\ 

\subsection*{Acknowledgments}
L.~F.~Cavenaghi and L.~Katzarkov are supported by the Simons Foundation, grant SFI-MPS-T-Institutes-00007697, and the Ministry of Education and Science of the Republic of Bulgaria, grant DO1-239/10.12.2024. 

When this work was started, L. F. Cavenaghi was receiving support from The São Paulo Research Foundation FAPESP, grants 2022/09603-9 and 2023/14316-1.

L.~Katzarkov is supported by the Simons Investigators Award (no. 003136), the Simons Collaboration on Homological Mirror Symmetry (award no. 003093), and by the NSF FRG grant DMS-2245099.

The authors thank Lino Grama for his help with the paper. He played an essential role in the paper's beginning.

We are deeply indebted to Yuri Tschinkel and Ivan Cheltsov for their constant attention and encouragement to the paper.

The authors thank Tony Pantev for critical reading of this manuscript which ensured better clarity on the exposition.

In addition, Ludmil Katzarkov and Maxim Kontsevich are grateful to Jeffrey Fuqua for his continuous financial support of mathematical research throughout the years, and for his support and leadership in establishing the excellent creative environment at IMSA, University of Miami.

\

\

\section{The Theory of Atoms}
\label{sec:theory-of-atoms}

The \emph{theory of atoms}, recently introduced by Katzarkov, Kontsevich, Pantev, and Yu~\cite{katzarkovpantevyu}, provides new birational invariants of smooth projective varieties. The theory emerges from a synthesis of classical Hodge theory with Gromov-Witten theory, utilizing the former to refine and decompose the latter. While its conceptual roots lie in insights from Homological Mirror Symmetry (HMS,~\cite{Kontsevich-HMS}), it has been systematically developed without direct recourse to mirror symmetry methods. Instead, it leverages motivic structures inherent in both the algebraic and symplectic geometry of a single projective variety.

The result is a framework for constructing invariants, here termed the \emph{atom formalism}\footnote{It is important to stress that the ``theory of atoms'' refers to a framework for constructing invariants, rather than to a specific set of results or conjectures.}, which manifests in part as a \emph{non-archimedean $\mathbbm{k}$-analytic} refinement of quantum cohomology. In classical settings, quantum products~\cite{Kontsevich1994} are defined by formal power series or series valued in a Novikov field, giving rise to A-model Frobenius manifolds or variations of A-model noncommutative Hodge structures~\cites{KatzarkovKontsevichPantev2008, CSabbah}. These structures may be formal or complex-analytic, depending on notoriously subtle convergence properties.

The theory of atoms circumvents these analytic challenges through a change of setting. Instead of working over $\C$, the theory operates in the category of non-archimedean $\mathbbm{k}$-analytic supermanifolds over the algebraically closed valued field
\[
\mathbbm{k} := \bigcup_{n \in \mathbb{N}} \overline{\mathbb{Q}}(\!(\mathbf{y}^{1/n})\!),
\]
the field of Puiseux series in an auxiliary variable $\mathbf{y}$. This shift is not merely a technical convenience; it is the foundation that enables the very definition of atoms. The canonical spectral decomposition (Section \ref{sec:local-hodge-atoms}) at the heart of the theory is ill-defined over $\C$ due to analytic obstructions such as the Stokes phenomenon~\cite{KatzarkovKontsevichPantev2008}.

\ 

In what follows, we will recall the basic structure of Gromov-Witten theory for a smooth projective variety $X/\C$ and explore its interaction with the symmetries of Hodge theory. We begin by reviewing the construction of genus-zero Gromov-Witten invariants and their organization via the Gromov-Witten potential. We recall the construction of a non-archimedean analytic F-bundle---called the \emph{maximal A-model F-bundle}---which encodes the deformation data of quantum cohomology. In Section \ref{sec:Hodge-Z_p}, we describe the action of the Mumford-Tate group on this F-bundle, leading to the definition of \emph{Hodge atoms} in Section~\ref{sec:local-hodge-atoms}. This is meant to be a short recap for the sake of self-containment from part of the material in \cite{katzarkovpantevyu}. 

The remaining subsections present many new contributions: in Sections \ref{sec:equivariant-atoms}-\ref{sec:filtration} one introduces $G$-equivariant Hodge atoms, in Section \ref{sec:ideal-legal} we explain the connection between atom theory and categories, in Section \ref{sec:examples-atoms} we present some applications computing explicit examples, and in Section \ref{sec:monodrony} we merge atom theory with classical monodromy.

\ 

\subsection{Non-archimedean A-model F-bundles}
\label{sec:A-model}

Fix $X/\C$ a smooth projective variety and set $\mathrm{H}:= \mathrm{H}^\ast(X,\Q)$, the Betti cohomology of $X$ with rational coefficients. Let $\mathsf{CH}_1(X)$ be the group of 1-cycles on $X$. The group of 1-cycles modulo homological equivalence is denoted
\[
\mathsf{N}_1(X,\mathbb Z):=\mathsf{CH}_1(X)/\!\sim_{\mathrm{hom}}.
\]

For every $n\ge 0$ and $\beta\in\mathsf{N}_1(X,\mathbb Z)$, the expected (virtual) dimension of the moduli stack $\overline{\mathcal M}_{0,n}(X,\beta)$ of stable genus-zero $n$-marked maps of class $\beta$ is
\[
\boldsymbol{\delta}(n,\beta)
= n + (\dim_{\mathbb C} X - 3) + \langle c_1(T_X),\beta\rangle,
\]
where $\langle-,-\rangle$ denotes the natural pairing between cohomology and curve classes. The stack $\overline{\mathcal M}_{0,n}(X,\beta)$ parameterizes data $(C,p_1,\dots,p_n,\varphi)$ where $C$ is a connected nodal curve of genus $0$, the $p_i$ are distinct smooth marked points on $C$, and $\varphi: C\to X$ is a morphism with $\varphi_*[C]=\beta$, satisfying a stability condition (each contracted irreducible component must carry at least three special points, whether marked or nodal).

The stack $\overline{\mathcal M}_{0,n}(X,\beta)$ is a proper Deligne-Mumford stack and carries a virtual fundamental class \cites{behrend-fantechi1997,li-tian1998}
\[
[\overline{\mathcal M}_{0,n}(X,\beta)]_{\mathrm{vir}} \in \mathsf{CH}_{\boldsymbol{\delta}(n,\beta)}(\overline{\mathcal M}_{0,n}(X,\beta)).
\]
which further maps to a virtual fundamental class modulo homological equivalence
\[
[\overline{\mathcal M}_{0,n}(X,\beta)]_{\mathrm{vir}}^{\mathrm{hom}} \in \mathsf{CH}_{\boldsymbol{\delta}(n,\beta)}^{\mathrm{hom}} (\overline{\mathcal M}_{0,n}(X,\beta)).
\]

The evaluation morphism
\[
\mathrm{ev}:\overline{\mathcal M}_{0,n}(X,\beta)\longrightarrow X^{\times n},\qquad
(C,p_1,\dots,p_n,\varphi)\longmapsto (\varphi(p_1),\dots,\varphi(p_n)),
\]
allows the definition of the Gromov-Witten cycle class by pushforward of the virtual class:
\[
I_{n,\beta}(X)\;:=\;\mathrm{ev}_*\big([\overline{\mathcal M}_{0,n}(X,\beta)]_{\mathrm{vir}}^{\mathrm{hom}} \big)
\;\in\; \mathsf{CH}^{\mathrm{hom}}_{\boldsymbol{\delta}(n,\beta)}(X^{\times n})\otimes \mathbb Q^.
\]

Let $\mathsf{NE}(X,\mathbb Z)\subset\mathsf{N}_1(X,\mathbb Z)$ stand to the monoid of effective curve classes. For $\beta\in \mathsf{NE}(X,\mathbb Z)$ one can pair $I_{n,\beta}(X)$ with arbitrary cohomology classes $\gamma_1,\dots,\gamma_n\in\mathrm{H}$, yielding the Gromov-Witten correlators:
\[
\langle \gamma_1\cdots\gamma_n\rangle_\beta
:=\int_{I_{n,\beta}(X)} \gamma_1\boxtimes\cdots\boxtimes\gamma_n
=\big(I_{n,\beta}(X),\; \mathrm{pr}_1^{\ast}\gamma_1\otimes\cdots\otimes\mathrm{pr}_n^{\ast}\gamma_n\big)
\in\mathbb Q.
\]
where $(\cdot,\cdot)$ means the Poincar\'e pairing in $X^{\times n}$ and $\mathrm{pr}_i:X^{\times n}\rightarrow X$ are the projections of the i-th factor.

\ 

Consider the monoid algebra $\Q[\mathsf{NE}(X,\mathbb Z)]$ where, for each $\beta\in\mathsf{NE}(X,\mathbb Z)$, one denotes by $q^\beta$ be the corresponding basis element. The formal Novikov ring $\Q[[q]]$ is the completion of $\Q[\mathsf{NE}(X,\mathbb Z)]$ at the ideal generated by $\{q^\beta:\beta\neq 0\}$. Choose a homogeneous $\Q$-basis $\{T_0,\dots,T_{r}\}$ of $\mathrm{H}$ with $T_0=1\in\mathrm{H}^0(X,\Q)$. The dual linear coordinates on $\mathrm{H}$ are denoted $t_0,\dots,t_r$. The variables $t_i$ are treated as super-variables with parity $\deg(T_i)\bmod 2$. Consequently, the algebra of $t$-functions is the completed symmetric algebra in the even variables tensored with the exterior algebra in the odd variables. 

\begin{definition}
The genus-zero Gromov-Witten potential of $X$ is the formal series with $\mathbb Q[[q]]$-coefficients in the set of super variables $t_i$
\[
\Phi(q;t)
:=\sum_{\beta\in\mathsf{NE}(X)}\sum_{n\ge 0}\frac{q^\beta}{n!}\sum_{i_1,\dots,i_n}
\langle T_{i_1}\cdots T_{i_n}\rangle_\beta \; t_{i_1}\cdots t_{i_n}
\;\in\; \mathbb Q[[q]][[t_0,\dots,t_r]].
\]
\end{definition}
This ring of formal power series, $\Q[[q]][[t_0,\dots,t_r]]$, is the coordinatized form of the \emph{big Novikov ring} of $X$, denoted $\mathsf{Nov}_X$.

We now describe the domain on which this formal potential acquires genuine non-archimedean analyticity. We work over the non-archimedean field $\mathbbm{k}$ of Puiseux series, equipped with its standard valuation $\mathrm{val}.$ 

\ 

Let $\mathsf{NS}(X)_{\mathrm{tf}}$ be the torsion-free part of the N\'eron-Severi group of $X$ with rank $\rho:=\rank\mathsf{NS}(X)_{\mathrm{tf}}$ and choose a $\Z$-basis $\{e_1,\dots,e_\rho\}$ of $\mathsf{NS}(X)_{\mathrm{tf}}$. Form the algebraic torus $\mathcal{T}_{\mathbbm{k}}:=\mathsf{NS}(X)_{\mathrm{tf}}\otimes_{\Z}\mathbb{G}_{m,\mathbbm{k}} \cong (\mathbb{G}_{m,\mathbbm{k}})^\rho$, and denote its non-archimedean analytic realization (in the sense of Berkovich) by $\mathcal{T}_{\mathbbm{k}}^{\mathrm{an}}$. The canonical valuation map is
\[
\operatorname{val}:\mathcal{T}_{\mathbbm{k}}^{\mathrm{an}}\longrightarrow \mathsf{NS}(X)_{\R}
:=\mathsf{NS}(X)_{\mathrm{tf}}\otimes_{\Z}\R,
\qquad
(z_1,\dots,z_\rho)\longmapsto \sum_{i=1}^\rho \operatorname{val}(z_i)\, e_i.
\]
Let $\mathsf{Amp}(X)\subset\mathsf{NS}(X)_{\R}$ be the ample (K\"ahler) cone and set $B_{X,q}:=\mathrm{val}^{-1}\big(\mathsf{Amp}(X)\big)\subset\mathcal{T}_{\mathbbm{k}}^{\mathrm{an}}$. 

For the cohomology directions, we define analytic domains $\mathcal B^{\mathrm{ev}}_{X,t}$ and $B^{\mathrm{ev}}_{X,t}$. The former, $\mathcal B^{\mathrm{ev}}_{X,t}$, is the product of an affine line (coordinate $t_0$) with the open unit polydisk $\{ t_i \in \mathbbm{k}: |t_i| < 1 \}$ in the even-degree cohomology coordinates (excluding $t_0$). The latter, $B^{\mathrm{ev}}_{X,t}$, is defined analogously but with the degree-2 coordinates restricted to the transcendental subspace $\mathrm{H}^2_{\mathrm{trans}}(X,\Q)\otimes\mathbbm{k}$. Let $B_X^{\mathrm{odd}}$ be the purely odd super-analytic variety associated with $\mathrm{H}^{\mathrm{odd}}(X,\Q)\otimes\mathbbm{k}$.

We define the analytic parameter spaces:
\[
\mathcal B_X:= B_{X,q}\times\mathcal B^{\mathrm{ev}}_{X,t}\times B_X^{\mathrm{odd}},
\qquad
B_X:= B_{X,q}\times B^{\mathrm{ev}}_{X,t}\times B_X^{\mathrm{odd}}.
\]
By construction, $\mathcal B_X$ is an open subvariety in the super-analytic variety $\mathcal{T}_{\mathbbm{k}}^{\mathrm{an}}\times (\mathrm{H}\otimes_{\Q}\mathbbm{k})^{\mathrm{an}}$. The space $B_X$ is the analogue where degree-2 directions are restricted to the transcendental part. The inclusion $\mathrm{H}^2_{\mathrm{trans}}\hookrightarrow \mathrm{H}^2$ induces a closed analytic embedding $B_X\hookrightarrow\mathcal B_X$. The space $\mathcal B_X$ is called the \emph{overmaximal} base, while $B_X$ is the \emph{maximal} base. Following \cite[Lemma 3.29]{katzarkovpantevyu}, $\Phi$ defines a $\mathbbm k$-valued analytic function on $\mathcal B_X$.

Let $\mathbb D$ be the germ at $0$ of an analytic disk with coordinate $u$. We define the trivial (super)analytic vector bundle $\mathcal H := \mathrm{H}\otimes_{\Q}\mathcal O^{\mathrm{an}}_{\mathcal B_X\times\mathbb D}$. The Poincar\'e pairing on $\mathrm{H}$ induces a constant, non-degenerate bilinear form on $\mathcal H$. The third derivatives of the convergent potential $\Phi$ define the analytic quantum product $\star$ on the fibers of $\mathcal H$. By the WDVV identities, this product is associative and (graded) commutative, \cite{Kontsevich1994}.

There is a flat meromorphic connection $\nabla$ on $\mathcal H$ over $\mathcal B_X\times\mathbb D$ with poles at most at $u=0$. To describe it in coordinates, let $L_1, \dots, L_m \in \mathrm{Pic}(X)$ be ample line bundles such that their first Chern classes form a basis for the $\mathbb Q$-N\'eron-Severi space of $X$. Write $\omega_i := \mathsf{cl}(c_1(L_i))$ for the image of $c_1(L_i)$ under the cycle class map and consider the open simplicial cone
$$
\sigma := \left\{ \sum_{i=1}^m a_i \omega_i \;\middle|\; a_i \in \mathbb{R}_{>0} \right\} \subset \mathbb{R}\omega_1 \oplus \dots \oplus \mathbb{R}\omega_m.
$$
Let $B_{\sigma,q} \subset \mathsf{NS}(X, \mathbbm k)$ be the preimage of $\sigma$ by the valuation map, and let $\mathcal B_{\sigma} := B_{\sigma,q} \times \mathcal B^\mathrm{ev}_{X,t} \times  B^\mathrm{odd}_{X}$.
Note that by definition $\mathcal B_{\sigma} \subset \mathcal B_X$ is an open analytic subvariety and $(\mathcal B_{\sigma},(\{q_j\},\{t_i\}))$ work as a coordinate chart where we can define
\begin{align}
\nabla_{u\partial_u} &= u\partial_u + u^{-1}\big(\mathsf{Eu}\star(\cdot)\big) + \tfrac12\big(\mathsf{Deg}-\dim_{\mathbb C} X\cdot\mathrm{id}\big),\\
\nabla_{q_j\partial_{q_j}} &= q_j\partial_{q_j} + u^{-1}\big(\omega_j\star(\cdot)\big),\\
\nabla_{\partial_{t_i}} &= \partial_{t_i} + u^{-1}\big(T_i\star(\cdot)\big).
\end{align}
Here, $\mathsf{Deg}:=\oplus_{a=0}^{2\dim_{\C}X}a\cdot \mathrm{id}_{\rH^a(X)}$ is the degree operator, and $\mathsf{Eu}$ is the Euler vector field expressed as 
\begin{equation}\label{eq:Euler}
\mathsf{Eu}_b \;=\; c_1(T_X) \;+\; \left(\frac{\mathsf{Deg} - 2\cdot\mathrm{id}}{2}\right)(b)
\;\in\; T_b\mathcal B_X \;\subset\; \mathcal H_b
\end{equation}
for every $b\in \mathcal B_{\sigma}$.

The pair $(\mathcal H, \nabla)$ over the overmaximal base $\mathcal{B}_X$ is called the \emph{overmaximal A-model F-bundle}. Its restriction to the maximal subvariety $B_X\subset \mathcal{B}_X$ defines the following main concept:

\begin{definition}\label{def:Amodel-revised}
The bundle $(\mathcal{H}, \nabla) / B_X$ is called the \emph{non-archimedean maximal A-model F-bundle} associated to $X$. The base $B_X$ is called an \emph{F-manifold}. The triple $(B_X, \mathsf{Eu}, \star)$ is called the \emph{A-model triple}.
\end{definition}

Next, we describe how the \emph{noncommutative Mumford-Tate group} $\mathsf{Hod}$ acts on the F-manifold $B_X$.

\ 

\subsection{Action of the Mumford-Tate group}
\label{sec:Hodge-Z_p}

Recall that a (pure) $\mathbb Q$-Hodge structure of weight $w\in \mathbb Z$ is a $\mathbb Q$-vector space $V$ endowed with the descoposition of its complexification $V\otimes_{\mathbb Q}\mathbb C=\oplus_{p\in \mathbb Z}V^{p,w-p}$ such that $V^{p,w-p}=\overline{V^{w-p,p}}$.
The \emph{Hodge-Tate structure of weight \(2n\)}, denoted \(\mathbb{Q}^{\mathrm{H}}(-n)\), is the 1-dimensional space $V\simeq \mathbb Q$ with Hodge decomposition $V\otimes\mathbb C=V^{n,n}$.
The symmetric twists \((n,n)\) (\emph{Tate twists}) are \(V \otimes \mathbb{Q}^{\mathrm{H}}(-n)\) for any \(\mathbb{Q}\)-Hodge structure \(\mathrm{H}\) of weight \(w \in \mathbb{Z}\). We denote them as \(V(-n)\).
When $w$ is even, the elements of $V\cap V^{w/2,w/2}$ are called \emph{Hodge classes}.

A Hodge structure $(V,(V^{p,w-p})_{p\in\Z})$ is called \emph{polarizable} if there exists a non-degenerate bilinear pairing
$(\,,\,):V\otimes V\to \mathbb Q^{\mathrm H}(-w)$ inducing a positive-definite Hermitian form $v\mapsto(-1)^p (v,\overline v)$ on each summand $V^{p,w-p}$. The category $\mathcal{HS}$ of finite direct sums of pure Hodge structures of various weights is a semisimple tensor category endowed with a faithful functor to the category of vector spaces over $\mathbb Q$. By Tannakian formalism, the category $\mathcal{HS}$ is identified with the category of representations of a pro-reductive group over $\mathbb Q$, which we denote by $\mathsf{MT}$.
For any concrete object $\mathbf{V}=(V,(V^{p,q})_{p,q\in\Z})\in \mathcal{HS}$, the image of $\mathsf{MT}$ in $\mathrm{Aut}(V)$ is the reductive algebraic subgroup defined over $\Q$ whose $\overline{\Q}$-points are automorphisms of $V\otimes \overline\Q$ fixing all Hodge classes of weight $(0,0)$ in all Hodge structures $\mathbf{V}^{m_1}\otimes (\mathbf{V}^*)^{m_2}$ for all $m_1,m_2\ge 0$.

The pro-reductive group $\mathsf{MT}$ over $\Q$ is endowed with two homomorphisms
$$\mathbb G_{m,\Q}\stackrel{f_1}{\hookrightarrow}\mathsf{MT}\stackrel{f_2}{\twoheadrightarrow} \mathbb G_{m,\Q}\,.$$
The monomorphism $f_1$ is characterized by the property that $f_1(\mathbb G_{m,\Q})$
acts by character $w\in \Z=(\mathbb G_m)^\vee$ in each polarizable pure Hodge structure of weight $w$. The epimorphism $f_2$ is the action of $\mathsf{MT}$ on $\Q(-1)$. The composition $f_2\circ f_1$ is $\lambda\mapsto \lambda^2$.
Let us denote by $\mathsf{Hod}$ the kernel of
$f_2$. This is again a pro-reductive group over $\Q$ endowed with a \emph{central} element $\epsilon=f_1(-1)$ of order $2$. The group $\mathsf{Hod}$ acts on $\Z/2$-folded polarizable pure Hodge structure, the element $\epsilon$ acts by $(-1)^{\mathrm{wt}}$. This group is natural from the point of view of noncommutative geometry; see \cite{KatzarkovKontsevichPantev2008}.

\ 

For a smooth projective complex variety $X/\C$, the $B$-model noncommutative Hodge structure associated with the $\Z$-graded category $\mathrm{Perf}(X)$ is essentially a \emph{super-representation} of $\mathsf{Hod}$, where $\epsilon$ acts as $(-1)^{\mathrm{parity}}$. The invariant space
\[
\mathrm{H}^{\ast}(X,\Q)^{\mathsf{Hod}}
\]
coincides with the space of Hodge classes
\[
\bigoplus_{p}\bigl(\mathrm{H}^{2p}(X,\Q)\cap \mathrm{H}^{p,p}\bigr).
\]
Since the genus-zero Gromov-Witten invariants are defined by algebraic cycles, the group $\mathsf{Hod}$ acts on $B_X$: the action on the cone part is trivial, while on the complement of the image of $\mathsf{NS}(X)$ in $\mathrm{H}^{\ast}_B(X,\mathbbm{k})$ it is given by the natural $\mathsf{Hod}$-action after extension of scalars. We denote the fixed locus by $B_X^{\mathsf{Hod}}$.

\ 

\subsection{The Spectral Decomposition and Local Hodge Atoms}
\label{sec:local-hodge-atoms}

The F-bundle $(\mathcal H,\nabla)/B_X$ we built in Section \ref{sec:A-model} equips \cite[Section 3.4]{katzarkovpantevyu} the tangent bundle $TB_X$ with the structure of a Frobenius manifold. At any point $b \in B_X$, the tangent space $T_b B_X$ is identified with a commutative superalgebra whose product is the quantum product $\star_b$. This structure can be encoded in a map
    \[\mu_b: T_b B_X \rightarrow \mathrm{End}(\mathcal H_{(b,0)})\]
defined on the basis vectors by
    \[\mu_b(\partial_{t_i}) := T_i \star_b (\cdot).\]
Moreover, for any geometric point $b \in B_X$ there exists an {even cyclic vector} $h\in \mathcal H_{(b,0)}$ such that the map $\mathrm{ev}_h\circ \mu_b:T_bB_X\rightarrow \mathcal H_{(b,0)}$ is an isomorphism. The Euler vector field (Equation \eqref{eq:Euler}) is the unique even vector field $\mathsf{Eu}$ on $B_X$ which under the map $\mu$ maps to the residual endomorphism $\boldsymbol\kappa:=\nabla_{u^2\tfrac{\partial}{\partial u}}\Big|_{\mathcal H|_{u=0}}$, i.e., $\mu(\mathsf{Eu})=\boldsymbol\kappa$.

Let \( B^{\mathsf{Hod}}_X \subset B_X \) denote the fixed-point locus of the Mumford-Tate group \( \mathsf{Hod} \) acting on the non-archimedean super \( \mathbbm{k} \)-analytic variety \( B_X \). Consider the ramified cover \( \widetilde{B} \to B^{\mathsf{Hod}}_X \) parametrizing the eigenvalues of the operator \( \boldsymbol{\kappa} \). Let \( U_X \subset B^{\mathsf{Hod}}_X \) be the open subset over which the number of eigenvalues of \( \boldsymbol{\kappa} \) is maximal, and define
\[
\widetilde{U}_X := \widetilde{B}_{\mathrm{red}} \times_{B_X} U_X.
\]
\begin{definition}[Local Hodge Atoms]
    The connected components of \( \widetilde{U}_X \), denoted \( \pi_0(\widetilde{U}_X) \), are called the \emph{local Hodge atoms} associated to \( X \). For each \( \alpha \in \pi_0(\widetilde{U}_X) \), the \emph{multiplicity} of \( \alpha \) is defined as the degree of the connected component \( \widetilde{U}_{X, \alpha} \subset \widetilde{U}_X \) considered as a covering of \( U_X \).
\end{definition}

\ 

\subsection{The Global Set of Hodge Atoms}
\label{sec:global-hodge-atoms}

The behavior of quantum cohomology under birational transformations, such as blowups, motivates the definition of a global, birationally invariant set of atoms.

\subsubsection{Motivation from Birational Geometry}
\label{sec:atoms-and-birational-geometry}

Recall that any birational map between two smooth projective varieties (in characteristic zero) can be decomposed as the composition of blowups at smooth centers and inverses to such blowups (see, e.g., \cites{AbramovichKaruMatsukiWlodarczyk2002, Wlodarczyk2000}):
\begin{theorem}[Weak Factorization]\label{thm:weak}
    Let $\phi: X\dashrightarrow Y$ be a birational map between smooth complete varieties over an algebraically closed field of characteristic 0. Let $U\subset X$ be an open set where $\phi$ is an isomorphism. Then $\phi$ can be factorized as
    \[X=X_0\stackrel{f_0}{\dashrightarrow}X_1\stackrel{f_1}{\dashrightarrow}\ldots\stackrel{f_{n-1}}{\dashrightarrow}X_n=Y.\]
    Each $X_i$ is a smooth variety, and $f_i$ is a blowup or blowdown at a smooth center disjoint from $U$. Moreover, if $X$ and $Y$ are projective, each $X_i$ is projective.
\end{theorem}

Let \( Z \subset X \) be a closed smooth subvariety of codimension \( r \geq 2 \). For simplicity, we will assume that \( Z \) is irreducible. The following is a classical result (see, e.g. \cite[Theorem 7.31]{Voisin_2002}):

\begin{proposition}\label{thm:voisin}
    Let \( X \) be a complex smooth projective variety, and let \( Z \) be an irreducible closed subvariety of codimension \( r \geq 2 \). Denote by \( \mathrm{Bl}_Z X \) the blowup of \( X \) in the smooth center \( Z \). Then, for each non-negative integer \( k \), there is a canonical isomorphism of Hodge structures
    \[
    \mathrm{H}^k (X, \mathbb{Q}) \oplus \bigoplus_{i=0}^{r-2} \mathrm{H}^{k-2i-2}(Z, \mathbb{Q}) \cong \mathrm{H}^k (\mathrm{Bl}_Z X, \mathbb{Q}).
    \]
    The Hodge structure on the components \( \mathrm{H}^{k-2i-2}(Z) \) corresponds to that of \( Z \) but is shifted by \( (i+1, i+1) \) in bidegree, resulting in a Hodge structure of weight \( k \).
\end{proposition}

Herein, we denote \(\mathrm{Bl}_Z X\) as \(\widetilde{X}\). Proposition \ref{thm:voisin} teaches us that if we consider the non-archimedean $\mathbbm{k}$-analytic supermanifold \(B_{\widetilde{X}}\), constructed mutatis mutandis to that of \(B_X\) (as in Section \ref{sec:A-model}), it should be possible to relate the A-model triples \((B_{\widetilde{X}}, \widetilde{\star}, \widetilde{\mathsf{Eu}})\) and \((B_X, \star, \mathsf{Eu})\).

Let \(\omega \in \mathrm{H}^2_B(X, \mathbb{Z}) \subset \mathrm{H}\) be an ample class on \(X\), and pick a point \(\widetilde{b} \in B_{\widetilde X}^{\text{even}}\) corresponding to an ample class on \(\widetilde{X}\) that is sufficiently close in \(\mathrm{H}^{\ast}(\widetilde{X}, \mathbb{Q}) \otimes \mathbbm{k}\) to the semi-ample class \((\widetilde{X} \rightarrow X)^{\ast} \omega\). 

Consider the spectrum of \(\widetilde{\mathsf{Eu}} \widetilde{\star} (\cdot) |_{T_{\widetilde{b}} B_{\widetilde{X}}}\). A straightforward computation shows that the corresponding eigenvalues group together in \(r\)-clusters, where each cluster is contained within a small analytic disk in \(\mathbbm{k}\) as follows:
\begin{itemize}
\item The disks are disjoint, one of which is centered at $0\in \mathbbm{k}$, while the rest are centered at the $(r-1)$-roots of $1$, rescaled by $r-1$
\item The sum of the generalized eigenspaces corresponding to the eigenvalues that are close to $0$ gives a super-vector space isomorphic to $\mathrm{H}^{\ast}(X,\mathbbm{k})$
\item The sum of the generalized eigenspaces corresponding to the eigenvalues close to a rescaled $(r-1)$-st root of $1$ gives a super vector space isomorphic to $\mathrm{H}^{\ast}(Z,\mathbbm{k})$
\end{itemize}
In a picture, if $r=\mathrm{codim}(Z\subset X)=4$, we might see an illustration similar to Figure \ref{pic:satellite}.
\begin{center}
\begin{figure}[H]
\begin{tikzpicture}[scale=1, every node/.style={font=\tiny}]
    \definecolor{satellite}{RGB}{211,211,211} 
    \definecolor{main}{RGB}{173,216,230}      

    \fill[main] (0,0) circle (0.8cm); 
    \foreach \x/\y in {-0.4/0.4, 0.4/0.4, -0.4/-0.4, 0.4/-0.4, 0/0.5, 0/-0.5, -0.5/0, 0.5/0} {
        \node[black] at (\x,\y) {\textbullet}; 
    }
    \node[anchor=north] at (0,-0.9) {$\mathrm{H}^{\ast}(X,\mathbbm{k})$}; 

    \begin{scope}[shift={(-2,1.5)}, rotate=143.13]
        \fill[satellite] (0,0) circle (0.6cm); 
        \foreach \x/\y in {-0.2/0.2, 0.2/0.2, -0.2/-0.2, 0.2/-0.2} {
            \node[black] at (\x,\y) {\textbullet};
        }
    \end{scope}
    \node[anchor=south] at (-2,2.1) {$\mathrm{H}^{\ast}(Z,\mathbbm{k})$}; 

    \begin{scope}[shift={(2,1.5)}, rotate=36.87]
        \fill[satellite] (0,0) circle (0.6cm);
        \foreach \x/\y in {-0.2/0.2, 0.2/0.2, -0.2/-0.2, 0.2/-0.2} {
            \node[black] at (\x,\y) {\textbullet};
        }
    \end{scope}
    \node[anchor=south] at (2,2.1) {$\mathrm{H}^{\ast}(Z,\mathbbm{k})$}; 

    \begin{scope}[shift={(0,-2)}, rotate=270]
        \fill[satellite] (0,0) circle (0.6cm);
        \foreach \x/\y in {-0.2/0.2, 0.2/0.2, -0.2/-0.2, 0.2/-0.2} {
            \node[black] at (\x,\y) {\textbullet};
        }
    \end{scope}
    \node[anchor=north] at (0,-2.6) {$\mathrm{H}^{\ast}(Z,\mathbbm{k})$}; 

\end{tikzpicture}
\caption{Clusters}
\label{pic:satellite}
\end{figure}
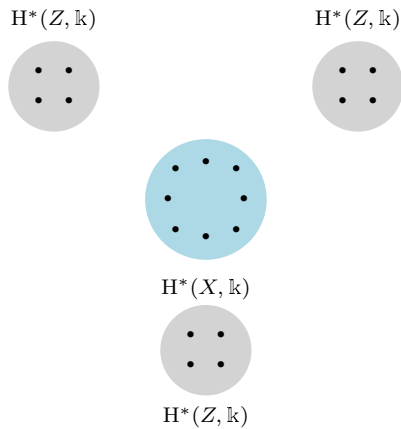
\end{center}

Theorem 4.1 in \cite{katzarkovpantevyu} implies that near the point $\widetilde b$, the F-manifold $B_{\widetilde X}$ is locally isomorphic to the product of $r$ distinct F-manifolds, each having an Euler field. Among these, one manifold has the same dimension as $B_X$, while the remaining $(r-1)$ factors have the same dimension as $B_Z$. Furthermore, an adaptation of the proof of Theorem 5.14 in \cite{iritani2023quantumcohomologyblowups} (Theorem 4.5 in \cite{katzarkovpantevyu}) shows that the factors in the decomposition surrounding $\widetilde b \in (B_{\widetilde X}, \widetilde \star, \widetilde{\mathsf{Eu}})$ are isomorphic to an open domain in $(B_X, \star, \mathsf{Eu})$, along with $(r-1)$ copies of an open domain in $(B_Z, \star_Z, \mathsf{Eu}_Z)$. Under this isomorphism, the Euler vector field $\widetilde{\mathsf{Eu}}$ is a direct sum of the pullbacks of the Euler vector fields $\mathsf{Eu}$ on $B_X$ and $\mathsf{Eu}_Z$ on various copies of $B_Z$. Similarly, the quantum product $\widetilde\star$ on the F-manifold $B_{\widetilde X}$ is block-diagonal. This provides a block-diagonal decomposition of the operator $\widetilde{\mathsf{Eu}}\widetilde{\star}$ as a direct sum of pullbacks of the operators $\mathsf{Eu}\star,~\mathsf{Eu}_Z\star_Z$ from the factors. 

\subsubsection{Definition of Global Hodge Atoms}
\label{subsec:global-hodge-atoms-def}

One is naturally led to impose equivalence relations on the set
\[
\bigsqcup_{\substack{
      [X]\text{\, iso. classes of smooth}\\
      \text{complex projective varieties}
    }}
    \pi_{0}(\widetilde U_{X}) \Big/ \mathrm{Aut}(X),
\]
yielding a quotient set we denote by \(\mathsf{HAtoms}\). As detailed in \cite{katzarkovpantevyu}:
\[
\mathsf{HAtoms}
  = \Big(\bigsqcup_{\substack{
      [X]\text{\, iso. classes of smooth}\\
      \text{complex projective varieties}
    }}
    \pi_{0}(\widetilde U_{X}) \Big/ \mathrm{Aut}(X) \Big)\Big/ \sim,
\]
where \(\sim\) is the equivalence relation generated by:
\begin{enumerate}[label=(\roman*)]
  \item \textbf{disjoint unions} (see \cite[Section 5.2.3]{katzarkovpantevyu}):  If $X = X_1 \sqcup X_2$, the local atoms of $X$ are naturally identified with the disjoint union of the local atoms of $X_1$ and $X_2$. A local atom of $X_1$ is declared equivalent to its image in the set of local atoms of $X$.
  \item \textbf{blowups in smooth centers} (see \cite[Section 5.2.4]{katzarkovpantevyu}). Let $\widetilde X= \mathrm{Bl}_Z(X)$ be the blowup of $X$ in a smooth subvariety $Z$ of codimension $c$. Let $X' = X \sqcup Z \sqcup \ldots \sqcup Z$ (with $c-1$ copies of $Z$). Iritani's theorem provides a canonical isomorphism of F-bundles over suitable analytic domains, which in turn induces a canonical bijection between the local atoms of $\widetilde X$ and the local atoms of $X'$. A local atom of $\widetilde X$ is declared equivalent to its corresponding atom in $X'$ under this bijection.
  \item  \textbf{projectivizations of vector bundles} (``quantum Leray-Hirsch decomposition'', see \cite[Section 5.2.5]{katzarkovpantevyu}). If $\mathbb P(E) \to X$ is a projectivization of a vector bundle of rank $r\geq 2$ over $X$, a similar (though less subtle) decomposition theorem relates the F-bundle of $\mathbb P(E)$ to the direct sum of $r$ copies of the F-bundle of $X$. This induces an equivalence between the local atoms of $\mathbb P(E)$ and those of $X \sqcup \ldots \sqcup X$ ($r$ copies).
\end{enumerate}
It is worth mentioning that the identifications implicit in relations $\mathrm{(ii),(iii)}$ are defined using the fundamental results of H.~Iritani \cite{iritani2023quantumcohomologyblowups}, and H.~Iritani and Y. Koto \cite{iritani2024quantumcohomologyprojectivebundles}.

For each $\boldsymbol{\alpha} \in \mathsf{HAtoms}$ there exists a smooth projective variety \(X\) such that the equivalence class \(\boldsymbol{\alpha}\) is represented by a local atom of $X$, i.e., by a choice of a connected component \(U_{\alpha}\) of the unramified part \({\widetilde U}_X\) of the reduced spectral cover of the operator
\[
\boldsymbol\kappa = \mathsf{Eu} \star (-) : \mathcal{H}_0|_{{B}^{\mathsf{Hod}}_X} \to \mathcal{H}_0|_{{B}^{\mathsf{Hod}}_X}
\]
for the \(\mathbbm k\)-analytic A-model maximal F-bundle \((\mathcal{H}, \nabla) / {B}_X\). The subbundle 
\[
\mathcal{E}^{\alpha} \subset \mathcal{H}_0|_{U_X}
\]
of generalized eigenspaces of \(\mathsf{Eu} \star (-)\) corresponding to the eigenvalues parametrized by \({\widetilde U}_{X,\alpha}\) is invariant by the (extended scalar) action of $\mathsf{Hod}$. By Lemma~5.25 in \cite{katzarkovpantevyu}, for any rigid point \( b \in U_X \), the representation of \( \mathsf{Hod} \) on the fiber \( \mathcal{E}^{\alpha}_b \) is induced by extension of scalars from some finite-dimensional \( \overline{\mathbb{Q}} \)-linear representation \( E^{\boldsymbol{\alpha}} \) of \( \mathsf{Hod} \), whose isomorphism class is independent of \( b \). As indicated by the notation, the isomorphism class of the \( \mathsf{Hod} \)-module \( E^{\boldsymbol{\alpha}} \) depends only on the atom \( \boldsymbol\alpha \), and not on the particular local atom representing it. This follows immediately from the fact that the elementary equivalences used to define Hodge atoms all respect the action of \( \mathsf{Hod} \). Thus, to every Hodge atom \( \boldsymbol\alpha \), we have assigned a finite-dimensional \( \overline{\mathbb{Q}} \)-linear representation \( E^{\boldsymbol{\alpha}} \) of \( \mathsf{Hod} \), well defined up to isomorphism. With this notation, we now have the following invariants of Hodge atoms:
\begin{equation}
\text{
\fbox{\parbox{0.85\textwidth}{
\begin{itemize}
    \item The dimension \(\rho_{\boldsymbol\alpha}\) of the generalized \(\boldsymbol\alpha\)-eigenspace of \(\boldsymbol{\kappa}\) acting on the space of Hodge classes \(\mathrm{H}(X)^{\mathsf{Hod}} \otimes_{\mathbb{Q}} \mathbbm{k}\).
    \item The Hodge polynomial \(P_{\boldsymbol \alpha}(t) \in \mathbb{Z}[t, t^{-1}]\), where the coefficient at \(t^k\) corresponds to the dimension of the generalized \(\boldsymbol\alpha\)-eigenspace in \(\bigoplus_{p,q:p-q=k} \mathrm{H}^{p,q}(X)\).
\end{itemize}
}
}}\label{invariants-box}
\end{equation}
The above two invariants of Hodge atoms come from a finer one: the isomorphism class of a  $\mathbbm{k}$-linear representation (typically reducible)
of the Mumford-Tate group of  $\Z/2$-weighted  polarizable pure Hodge structures.

Since birational equivalences between smooth projective varieties are generated by blowups with smooth centers of codimension \( \geq 2\), if a \(d\)-dimensional variety \(X\) (with \(d \geq 2\)) contains at least one atom in its atomic content that does not appear in the atomic contents of varieties of dimension \(\leq d - 2\), then \(X\) cannot be rational. 
We conclude this section by summarizing all the discussion in the following:
\begin{theorem}[Katzarkov-Kontsevich-Pantev-Yu, \cite{katzarkovpantevyu}] \label{thm:Katzarkov-Kontsevich-Pantev-Yu} Let $X$ be a smooth projective variety, and $\mathrm{H}$ be the de Betti cohomology of $X$ equipped with the A-model $\mathbbm{k}$-non-archimedean $\mathbbm{k}$-analytic $\mathsf{nc}$-Hodge structure at a general point $s$ in the algebraic locus of the Frobenius manifold $B_X^{\mathsf{Hod}}$.
  \begin{itemize}
  \item For each point $b\in B_X^{\mathsf{Hod}}(\mathbbm{k})$ such that the number of eigenvalues of $\boldsymbol{\kappa}_s$ is maximal,   there is a canonical decomposition of $\mathrm{H}=\mathrm{H}^{\ast}(X,\mathbbm{k})$ into
    $A$-model $\mathsf{nc}$-Hodge substructures $\{\mathrm{H}_{\lambda_i}\}$ labeled by
    these eigenvalues.
  \item The above decomposition is also compatible with the
ordinary (B-model) $\mathbb Z/2$-folded polarizable pure Hodge structure on
$\mathrm{H}$. Hence, each summand $\mathrm{H}_{\lambda_i}$ carries a $\mathbbm{k}$-linear representation of the Mumford-Tate group. The collection of isomorphism classes of these representations up to reordering is independent
of the choice of the generic point $s$.
  \item If projective varieties are birational, they have the same collection of isomorphism classes of representations, modulo those coming
from subvarieties of codimension-two and higher.
 \end{itemize}
\end{theorem}
\begin{remark}
    We have reasons to believe that passing to isomorphism classes of the Mumford-Tate group with coefficients in $\mathbbm{k}$ is not necessary; the isomorphism classes of representations will be defined over $\mathbb Q\subset \mathbbm{k}$, i.e., it will be the usual $\mathbb
Z/2$-folded polarized Hodge structures. See, e.g., the discussion in Section \ref{sec:ideal-legal}.\label{rem:Q-not-K}
\end{remark}

\ 

\subsection{$G$-equivariant atoms}\label{sec:equivariant-atoms}

Let us fix a finite group $G$ and consider a smooth projective variety $X$ with a $G$-action (not necessarily free at the generic point).
Then $G\times \mathsf{Hod}$ acts on $\mathrm{H}^{\ast}(X,\Q)$. It also acts on $B_X$, with the action over the tube domain in $\mathsf{NS}(X)\otimes\mathbbm{k}^\times$ coming from the natural action of $G$ on $\mathsf{NS}(X)$. One can define local $G$-equivariant atoms
by repeating the definition of local Hodge atoms verbatim, restricting to the larger group's fixed locus $B_X^{G\times \mathsf{Hod}}$. It is known that the analog of the Weak Factorization Theorem also holds in the equivariant setting (see, e.g., Proposition 2.6 in \cite{Kresch2022}). Hence, we get a theory of $G$-equivariant atoms. Explicitly, we define 
$$  \mathsf{HAtoms}^G:=\Big(\bigsqcup_{\substack{
      [X]\text{\, iso. classes of smooth}\\
      \text{complex projective $G$-varieties}
    }}
    \text{Local  }G\text{-equivariant atoms of }X \Big/ \mathrm{Aut}(X)\Big) \Big/ \sim,$$
where $\sim$ is the equivalence relation generated by
\begin{enumerate}[label=(\roman*)]
  \item \textbf{disjoint unions}: any local $G$-equivariant atom of $X=X_1\sqcup X_2$ are identified with a local atom of $X_1$ or of $X_2$.
  \item \textbf{blowups in $G$-invariant smooth centers}: for $G$-invariant smooth $Z\subset X$ of codimension $c\ge 2$ any local $G$-equivariant atom of $\widetilde X=\mathrm{Bl}_Z(X)$ is identified with a local $G$-equivariant atom of $X$ or of $Z$.
  \item  \textbf{projectivizations of $G$-equivariant vector bundles} If $E\to X$ is $G$-equivarinat vector bundle of rank $r\ge 2$, then any local $G$-equivariant atom of $\mathbb P(E)$ is identified with a local $G$-equivariant atom of $X$.
\end{enumerate}
The identifications implicit in relations (ii), (iii) come from the fact that the identification of analytic F-bundles in Iritani's \cite{iritani2023quantumcohomologyblowups} (resp. Iritani--Koto's \cite{iritani2024quantumcohomologyprojectivebundles}) theorem is natural, i.e., commutes with automorphisms. It can hence be restricted to the fixed locus of $G$ (and eventually of $G\times\mathsf{Hod}$).

\begin{remark}
In principle, if $G$ acts on $X$ and only \textit{projectively} on the vector bundle $E\to X$, Iritani-Koto's result leads to a certain relation between local $G$-equivariant atoms of $\mathbb P(E)$ and local $G$-equivariant atoms of $X$. We do not impose identification by this relation, as it is not compatible with the categorical picture; see Section \ref{sec:ideal-legal}. The category $\mathrm{Perf}(\mathbb P(E))$ has a $G$-equivariant semi-orthogonal decomposition with terms which are equivalent to $\mathrm{Perf}(X)$ endowed with a \textit{twisted} $G$-action.
\end{remark}

We will describe in Section \ref{sec:examples-atoms} several examples of applications of $G$-equivariant atoms to $G$-equivariant birational geometry, analogous to the applications of the theory of atoms for non-algebraically closed fields in \cite{katzarkovpantevyu}. As invariants of $G$-equivariant atoms, we will use the dimension of Hodge classes and Hodge polynomial as in \eqref{invariants-box} and one extra invariant, the dimension of $G$-invariant Hodge classes:
\begin{equation}
\text{
\fbox{\parbox{0.85\textwidth}{
\begin{itemize}
    \item The dimension \(\rho_{\boldsymbol\alpha}^G\) of $G\times \mathsf{Hod}$-invariants in the generalized \(\boldsymbol\alpha\)-eigenspace of \(\boldsymbol{\kappa}\) acting on $\mathrm{H}^{*}(X,\mathbbm{k})$.
\end{itemize}
}
}}\label{invariants3-box}
\end{equation}

In a sense, the $G$-action is very similar to the Galois symmetry.

\subsection{Filtration by dimension}
\label{sec:filtration}
Expanding upon the discussion in Section~\ref{sec:global-hodge-atoms}, the set of Hodge atoms $\mathsf{HAtoms}$ admits a natural increasing filtration (\cite[Section 5.2.6]{katzarkovpantevyu} for the non-equivariant version):
\[
\mathsf{HAtoms}_{\dim \leq 0}^G
\subset
\mathsf{HAtoms}_{\dim \leq 1}^G
\subset \cdots, 
\qquad
\mathsf{HAtoms}^G =
\bigcup_{d \geq 0} \mathsf{HAtoms}_{\dim \leq d}^G.
\]
An atom $\boldsymbol\alpha$ belongs to $\mathsf{HAtoms}_{\dim \leq d}^G$ if it arises in the atomic decomposition of some smooth projective $G$-variety of complex dimension at most $d$. Then we have the following criterion:

\ 

\noindent\fbox{%
\parbox{0.97\textwidth}{%
If $G$-variety $X$ of dimension $d\ge 2$ has a $G$-equivariant atom which does not belong to $\mathsf{HAtoms}_{\dim \leq {d-2}}^G$, then $X$ is not $G$-birationally equivalent to $\mathbb P^d$ endowed with a linearizable $G$-action.}%
}

\

\  


In what follows, we will be mostly interested in varieties of dimension $3$. Therefore, we should study the truncated set 
\[
\mathsf{HAtoms}^{G}_{\dim \leq 1}
\]
as these are the only atoms that can appear/disappear under blowups or blowdowns in dimension $3$.
\medskip

\begin{lemma}\label{lem:atom-generators} Let us assume that for any subgroup $H<G$ any action of $H$ on $\mathbb P^1$ is linearizable.
The following varieties with $G$-action have each only one local equivariant Hodge atom, and in this way we obtain \textit{all} elements of \(\mathsf{HAtoms}^{G}_{\dim \leq 1}\):
\begin{enumerate}
    \item an orbit $G/H$ of a point where $H<G$ is the stabilizer subgroup,
    \item  the product $G\times_H C$ where $C$ is a smooth connected curve of genus $\mathsf g\ge 1$, endowed with an action of a subgroup $H<G$.
\end{enumerate}
In all cases, we have the invariant $\rho^G$ equal to $1$ or $2$. Moreover, if the Hodge polynomial is constant, then $\rho^G=1$.
\end{lemma}

\begin{proof} By definition, all elements of  \(\mathsf{HAtoms}^{G}_{\dim \leq 1}\) come from atomic decompositions of varieties (possibly disconnected) of dimension $\le 1$ endowed with a $G$-action. Obviously, it suffices to consider only the case when the group acts transitively on the set of connected components. If the dimension is $0$, we get case $(i)$, and here there is only one local equivariant Hodge atom. In the case of dimension $1$, we have an orbit of a connected curve which is invariant under a subgroup $H<G$.
 Suppose the genus of the connected component is $\mathsf g\ge 1$. In that case, we have only one local atom, as follows from the general properties of Gromov-Witten invariants of varieties with nef canonical class (see Example \ref{ex:nef}). Finally, if the Hodge polynomial is constant, then we deal with points or curves of genus zero, and we have $\rho^G=1$ in this case. Otherwise, we deal with curves of positive genus and $\rho^G=2$.
\end{proof}

We revisit Lemma \ref{lem:atom-generators} for the case when $G=\mathbb Z/p$ in Lemmas \ref{lem:atom-generators-2} and \ref{lem:invariant-for-generators}.



\subsection{Why atoms? Hypothetical atomic semi-orthogonal decomposition}
\label{sec:ideal-legal}
This section compares the ideal versus the ``legal''\footnote{By ``legal'', we mean the very restricted amount of tools distinguishing atoms that we can use at the moment. The structure is very rich, and we expect new tools to be added to the toolbox in the future, with new non-rationality results following.} perspectives in atomic theory. Specifically, we highlight the conjectural aspects, explaining how to approach atoms ideally before addressing them more formally. The ideal perspective offers a convenient way to gain insight into concrete examples, laying the groundwork for more detailed analysis later. 

The most optimistic conjecture is that atomic decomposition reflects a \emph{semi-orthogonal decomposition} of the bounded derived category $$\mathrm{Perf}(X)=\langle \mathcal C_1,\ldots,\mathcal C_m\rangle $$
where $m$ is the number of atoms, and the decomposition is defined up to the usual action of the braid group $B_m$ (see \cite{Bondal1990}).
Moreover, for each category $\mathcal C_i$, we should have a connected component of the space of Bridgeland stability conditions (in particular, the Hochschild homology/periodic cyclic homology space is non-trivial; hence, we can not have phantoms \cite{Bhning2015} in this decomposition).

This picture is partially justified by the Mirror symmetry for Fano varieties, where $\mathrm{Perf}(X)$ is expected to be equivalent to the Fukaya category of the mirror Landau-Ginzburg model. The relation to stability structures was also discussed in \cite{Blum2021}.

The categories ${\mathcal C}_i$ which appear in an SOD of $\mathrm{Perf}(X)$ are smooth and proper in the categorical sense (\cites{Kontsevich2008, Tabuada2015-tu}). For any smooth and proper category, which is \emph{geometric}. i.e., appears in SOD as above, there is a well-defined $\Z/2$-folded pure Hodge structure, i.e., a finite-dimensional super-representation \emph{with coefficients in $\Q$} of $\mathsf{Hod}$, where $\epsilon$ acts as the parity.   This explains our expectation in Remark \ref{rem:Q-not-K} that the equivalence of $\mathbbm{k}$-linear representation of $\mathsf{Hod}$ should be defined over $\Q$.    

Looking from the A-side, one is tempted to conjecture that the geometric categories $\mathcal{C}_i$ should have associated ``Fukaya categories'' depending on non-archimedean parameters, which are essentially local factors in the atomic decomposition of the Frobenius manifold. These categories are, in general, $\Z/2$-graded and Calabi-Yau of parity equal to $\dim X\pmod 2$. The structure of an F-manifold should arise from the variation of noncommutative Hodge structures for these categories, see, e.g., \cite{KatzarkovKontsevichPantev2008}.

When we have the action of a finite group $G$ in $X$, there is an associated categorical action of $G$ in $\mathrm{Perf}(X)$. One can consider general smooth proper categories with $G$ action and $G$-invariant semi-orthogonal decompositions. In this way, we get an equivariant version of the atomic decomposition. After the first public appearance of this paper, we learnt about \cite{ElaginSchneiderShinder}, where authors explore this in the case of surfaces.

\begin{remark}  
      There is a subtle point in the picture above that we do not fully understand, even regarding Fano varieties. A semi-orthogonal decomposition of $\Perf(X)$, identified with the Fukaya-Seidel category of the mirror Landau-Ginzburg model, should exist at \emph{any}  point of the even part of the Frobenius manifold, not necessarily on the Hodge locus. It remains unclear whether Hodge atoms can split further when we move away from the Hodge locus. This seems to contradict the property of any SOD that the group $\mathsf{Hod}$ preserves the induced decomposition of cohomology.  
      
      One can ask the question whether \emph{always} there is no further splitting. The positive answer would imply the invariance of the number of atoms under any deformation of the complex structure on $X$. Indeed, the absence of further splitting means that the number of atoms coincides with the number of distinct eigenvalues of $\mathsf{Eu}\star (\cdot)$ at the generic point of the even part of $B_X$. The latter is an invariant of $X$ considered as a \emph{real symplectic manifold} up to deformation of the real symplectic structure. The whole story generalizes to the  $G$-equivariant case.
  \label{rem:question}     
\end{remark}

\ 

\begin{remark}\label{rem:noncompact}
    The whole theory can be extended to a class of \textit{non-compact} quasi-projective varieties $X$ (possibly endowed with a finite group action) which are admissible at infinity in the  following sense: for any compact set $K\subset X$ and any homology class $\beta\in \rH_2(X,\Z)$ there exists a larger compact $K'\subset X$ such that for any stable map $\phi:C\to X$ of degree $\beta$ and genus $0$, if $\phi(C)\cap K\ne \emptyset$ then $\phi(C)\subset K'$. This property is inherited by smooth closed subvarieties and blowups at smooth closed centers. There are many examples of such varieties, e.g., those which admit proper maps to affine schemes, as is common for GIT quotients. 
 There is no Gromov-Witten potential for admissible at infinity varieties, but still the quantum product and F-bundle exist. Taylor coefficients of the quantum product give a structure of certain \textit{operad} on the cohomology space $\rH^*(X)$.
  The blow-up formula and projective bundle formula should hold in the more general non-compact setting, and given by certain universal formulas which are summations over trees, as it is the only general type of universal expressions for algebras over operads.

  A potential application of the non-compact theory will be \textit{non-compact birational geometry}, i.e., the study of equivalence classes of smooth quasi-projective varieties under blow-ups in smooth closed centers. Non-compact atoms should again correspond to terms in certain SOD decomposition of $\mathrm{Perf}(X)$ well-defined up to the usual braid group action.
\end{remark}

\ 

\subsection{Some examples and first applications}
\label{sec:examples-atoms}
We conclude this section by presenting some explicit atomic computation (without group action) and providing immediate obstructions regarding two smooth projective varieties being $G$-birationally equivalent for some cyclic groups $G=\Z/m$.


\begin{example}
\label{ex:nef}
    Assume that \( X \) is a connected smooth projective variety over \( \mathbb{C} \), and that its canonical class \( K_X \) is numerically effective (nef). Then the Hodge atomic decomposition of \( X \) consists of a single Hodge atom, denoted \( \boldsymbol\alpha(X) \).

To see this, let \( (T_i) \) be a homogeneous basis of \(\mathrm{H}^\ast(X, \mathbb{Q}) \), let \( (t_i) \) be the associated formal coordinates, and let \( (T^r) \) be the dual basis with respect to the Poincaré pairing. The big quantum product is defined via the Gromov-Witten potential $\Phi$:
\[
T_i \star T_j = \sum_r \frac{\partial^3 \Phi}{\partial t_i \partial t_j \partial t_r} \, T^r.
\]
Since \( K_X \) is nef, the virtual dimension formula for the moduli space of curves imposes a strong constraint on the Gromov-Witten invariants. A key consequence is that if the class \( T^r \) appears with a non-zero coefficient in the quantum product \( T_i \star T_j \), its degree is bounded from below:
\[
\deg T^r \geq \deg T_i + \deg T_j.
\]
Thus, the quantum product is filtered and does not decrease the sum of degrees.

The operator whose spectrum defines the atoms is $\boldsymbol{\kappa} = \mathsf{Eu} \star (\cdot)$. Because the Euler field $\mathsf{Eu}$ is a non-trivial cohomology class of degree at least 2, the degree constraint implies that the action of $\boldsymbol{\kappa}$ strictly increases the degree of any class it acts upon (except for the action on $\mathrm{H}^0(X)$).
\end{example}

\ 

 \begin{example}[Complete intersections in projective spaces]
\label{ex:complete-intersection}

Let $X$ be a smooth complete intersection in $\mathbb P^{N-1}$ of hypersurfaces of degress $d_1,\dots,d_k$ (we assume all $d_i\ge 2$). The dimension of $X$ is $N-1-k$.
By Lefschetz theorem, the cohomology of $X$ splits into $(N-k)$-dimensional subspace spanned by $(c_1(\mathcal O(1))^i,\,i=0,\dots N-1-k$ (the image of the restriction map $\rH^{\ast}(\mathbb P^{N-1})\to\rH^{\ast}(X)$), and the primitive  cohomology $\mathrm{H}^{N-1-k}_{\mathrm{prim}}(X)$.

 The canonical class of $X$ is given by
 $$K_X=\mathcal O(\sum_i d_i-N)\,.$$
In the Calabi-Yau case $\sum_i d_i=N$ or general type $\sum_i d_i>N$, the class $K_X$ is nef. Hence, we have only one atom by Example \ref{ex:nef}.

In the Fano case $\sum_i d_i<N$, a classical calculation by Givental \cite{Givental} gives the spectrum of $\boldsymbol{\kappa}_b$ at the point $b$ of $B_X$ corresponding to the ample class  $c_1(\mathcal O(1))\in \rH^2(X)$ (the corresponding product is called \emph{small quantum product}). The result is (up to scaling) roots of $1$ of order $r$ with multiplicity $1$, and $0\in \mathbbm{k}$ with a huge multiplicity\footnote{Strictly speaking, Givental calculated the action of quantum multiplication by $c_1(\mathcal O(1))$ only on the image of $\rH^{\ast}(\mathbb P^{N-1})$ in $\rH^{\ast}(X)$. It follows from the degree reasons  that the action on the remaining primitive part $\rH^{N-1-k}_{\mathrm{prim}}(X)$ has eigenvalue $0$}, where 
$$r=r_X:=N-\sum_i d_i$$ is called  the \textit{index} of Fano variety $X$.

For Fano complete intersection $X$ as before, there is a well-known semi-orthogonal decomposition matching with the spectral decomposition above
$$\mathrm{Perf}(X)=\langle\mathcal O(0),\mathcal O(1),\dots \mathcal O(r-1),\mathcal C\rangle\,.$$
Here  the full 
subcategory $\mathcal C\subset \mathrm{Perf}(X)$ is the right orthogonal to $(\mathcal O(i))_{i=0,\dots ,r-1}$. This category is a fractional Calabi-Yau category (for the hypersurface case, see, e.g., \cite{Kuznetsov2017}), and it is known under the name \emph{Kuznetsov component} of $X$.

 We continue with two further explicit examples.

The first is the threefold complete intersection of two quadrics in $\mathbb P^5$. So that $r_X=2, d_1=d_2=2$. The full Hodge diamond of $X$ is given by
\[
\begin{array}{ccccccc}
 &  &  & 1 &  &  & \\
 &  & 0 &  & 0 &  & \\
 & 0 &  & 1 &  & 0 & \\
0 &  & 2 &  & 2 &  & 0\\
 & 0 &  & 1 &  & 0 & \\
 &  & 0 &  & 0 &  & \\
 &  &  & 1 &  &  & \\
\end{array}
\]
and it is the direct sum of the restiction of $\rH^{\ast}(\mathbb P^5)$ (which coincides with  $\rH_{\mathrm{Hodge}}(X)$), and the purely transcendental part which is $\rH^3_{\mathrm{prim}}(X)$. The eigenvalues of $\boldsymbol{\kappa}_b$ are zero with multiplicity six and (up to scale) $\pm 1$ (roots of $1$ of order $r_X=2$) with multiplicity $1$. One can show using the general result concerning Witt algebra action (see \cite[Remark 3.14, p. 22]{katzarkovpantevyu}) that moving point $b$, we do not get further splitting of the spectrum, i.e., $X$ has exactly $3$ atoms. In this case, the Kuznetsov component is $\mathrm{Perf}(C)$ where $C$ is a genus $2$ curve. The variety $X$ is known to be rational; see, e.g., \cite{reid1972}.

For the second example, take $X$ as a fourfold complete intersection of two quadrics in $\mathbb P^6$. So that $r_X=3, d_1=d_2=2$. The full Hodge Diamond of \(X\) is given as
\[
\begin{array}{ccccccccc}
 &  &  &  & 1 &  &  &  & \\
 &  &  & 0 &  & 0 &  &  & \\
 &  & 0 &  & 1 &  & 0 &  & \\
 & 0 &  & 0 &  & 0 &  & 0 & \\
0 &  & 0 &  & 8 &  & 0 &  & 0\\
 & 0 &  & 0 &  & 0 &  & 0 & \\
 &  & 0 &  & 1 &  & 0 &  & \\
 &  &  & 0 &  & 0 &  &  & \\
 &  &  &  & 1 &  &  &  & \\
\end{array}
\]
There are $4$ eigenvalues of $\boldsymbol{\kappa}_b$ at the point $b$ corresponding to an ample class in $\rH^2(X)$: $3$ roots of order $r_X$ (up to scale) with multiplicity $1$, and  eigenvalue $0$ with multiplicity $9$. The key difference with the previous example is that \emph{all} cohomology classes of $X$ are algebraic, hence $B_X=B_X^{\mathsf{Hod}}(X)$. The point $b$ (small quantum product) is not general enough. We expect that at the generic point of $B_X^{\mathsf{Hod}}(X)$, the spectrum will consist of $12$ points with multiplicity $1$. This corresponds to the fact that $\mathrm{Perf}(X)$ has a full exceptional collection. The variety $X$ is rational because it contains a line \cite{reid1972}.
    \end{example}

\

Moving to $G$-equivariant atoms, we will prove some new non-rationality results in the equivariant setting.

\begin{example}\label{ex:cubic-surface}
    Let $X$ be a smooth cubic surface with a finite group $G$ action by regular automorphisms. Assume that $\mathrm{rank}~(\mathrm{Pic}(X))^{G}=1$, so the only (up to  scalar) $G$-invariant ample class is $-K_X\simeq \mathcal O(1)$. The spectral decomposition at this specific point we know by Givental's theorem \cite{Givental}, and get two eigenvalues: a multiple of $1$ (the only root of $1$ of order $r_X=1$) with multiplicity $1$ (corresponding to the exceptional object $\mathcal O_X$, and $0$ with multiplicity $8$\footnote{The cohomology of $X$ has rank $9$, recall that $X$ can be realized as the blowup of $\mathbb P^2$ in $6$ general points.}.

   We claim that there will be no further splitting as we move along the $G$-fixed locus in $B_X$. 
   Indeed, here is the  picture of the Hodge diamond, where we colored in red the $G$-invariant Hodge classes:

    \[
\begin{array}{ccccc}
 & & \color{red}{1} & & \\
 & 0 & & 0 & \\
 0 & & 6+\color{red}{1} & & 0 \\
 & 0 & & 0 & \\
 & & \color{red}{1} & &
\end{array}
\]
   A Witt Lie algebra argument, as in \cite[Remark 3.14, p. 22]{katzarkovpantevyu}, shows that the spectrum does not split further. I.e., for generic $s$ in $B_X^{G\times \mathsf{Hod}}$ we have only two eigenvalues for $\boldsymbol{\kappa}_s$.

We see that the atom corresponding to the zero eigenvalue  has 
$$\rho=8,\quad \rho^G=2\,.$$
This atom can not be an atom of zero-dimensional variety with $G$-action, as the latter corresponds to a permutation module (an orbit $G/H$ of a point) and always has $\rho^G=1$. Therefore,  $X$ is not  $G$-equivariantly birational to $\mathbb{P}^2$ with a linearizable action.

    Now, consider instead $X\times \PP^1$ with a linearizable  $G$-action (possibly trivial) on the $\PP^1$-factor. The atomic content of $X\times \mathbb{P}^1$ consists of two copies of the atoms of $X$. Thus, we have two atoms with $(\rho,\rho^G)=(8,2)$. An atom with $\rho^G=2$ cannot belong to \(\mathsf{HAtoms}^{G}_{\dim \leq 1}\), see Lemma \ref{lem:atom-generators}.  We conclude that $X\times \mathbb{P}^1$ cannot be $G$-equivariantly birational to $\mathbb{P}^3$ with a linearizable $G$-action.

 \end{example}   

\begin{example}
\label{ex:X(1,1,1,1)}
First, notice that any smooth hypersurface $X(1,1,1,1)$ of multi-degree $(1,1,1,1)$ in
\((\mathbb{P}^1)^4\) is \emph{rational}. Indeed, the projection to the first three factors identifies it with the blowup of  \((\mathbb{P}^1)^3\) in an elliptic curve. 

\begin{theorem}\label{thm:X(1,1,1,1)-withZ_2}
Assume that \( X(1,1,1,1) \) is invariant under the \(\mathbb Z/2\) action induced by swapping the first and second \(\mathbb{P}^1\) factors and the third and fourth \(\mathbb{P}^1\) factors. Then, there is no \(\mathbb Z/2\)-linearizable action on \(\mathbb{P}^3\) making it \(\mathbb Z/2\)-equivariantly birational to \( X(1,1,1,1) \).
\end{theorem}
\begin{proof}
Herein, we set $X:=X(1,1,1,1)$ for ease of notation. As for the first observation, notice that by the Lefschetz hyperplane section theorem, the algebraic classes in $X$ have dimensions $1, 4, 4, 1$. Its full Hodge diamond is given by:
 \[
\begin{array}{ccccccc}
  &&& 1 &&& \\
  && 0 && 0 && \\
  & 0 && 4 && 0 &\\
  0 && 1 && 1 && 0\\
  & 0 && 4 && 0 &\\
  && 0 && 0 && \\
  &&& 1 &&&
\end{array}
\]

The action of $G=\mathbb{Z}/2$ is generated by $\sigma$ swapping the factors $(1\leftrightarrow 2)$ and $(3\leftrightarrow 4)$. This action permutes the generators of $\mathrm{H}^{1,1}((\mathbb P^1)^4)$, $\{h_1,h_2,h_3,h_4\}$. A basis for the $G$-invariant subspace of $\mathrm{H}^{1,1}((\mathbb P^1)^4)$ is $\{h_1+h_2,\,h_3+h_4\}$. By the Lefschetz Theorem, the restriction to $X$ is an isomorphism in $\mathrm{H}^2$. The total number of $G$-invariant Hodge classes is
\[
{\rho^G}^{\mathrm{total}}=h^{0,0,G}+h^{1,1,G}+h^{2,2,G}+h^{3,3,G}=1+2+2+1=6
\]
which are colored in \textcolor{red}{red} in the picture of the full Hodge diamond:
\begin{equation}
    \begin{array}{ccccccc}
  &&& \textcolor{red}{1} &&& \\ 
  && 0 && 0 && \\
  & 0 && 2+\textcolor{red}{2} && 0 &\\
  0 && 1 && 1 && 0\\
  & 0 && 2+\textcolor{red}{2} && 0 &\\
  && 0 && 0 && \\
  &&& \textcolor{red}{1} &&&
\end{array}
\end{equation}

The atomic decomposition of $X$ can be inferred from its Landau–Ginzburg mirror. While a full derivation is beyond our scope, the mirror model for $X$ can be given as a map on the hypersurface in affine space 
\[
z_1 + z_2 + z_3 + z_4 = 1
\] 
to \(\mathbb{A}^1\) given by
\[
\boldsymbol{w} = \frac{1}{z_1} + \frac{1}{z_2} + \frac{1}{z_3} + \frac{1}{z_4},
\]
which is manifestly \(\mathbb Z/2\)-equivariant.  We obtain two atoms corresponding to non-zero critical values (with conjectural categorical descriptions):
\begin{itemize}
    \item one is one-dimensional $(\lambda =16)$ with the trivial representation (sheaf \(\mathcal{O}_X\)), 
    \item another is 4-dimensional $(\lambda=4)$ with the permutation representations (\(\mathcal{O}(1)_i,\mathcal O(1)_j\), \(i=1,2,~j=3,4\)).
\end{itemize}

They together occupy a 5-dimensional subspace in all cohomology, and just a 3-dimensional subspace in invariant algebraic classes. Hence, in the rest we have \(3\)-dimensional invariant algebraic classes, and 5-dimensional classes, as well as a \((p-q)=\pm 1\) piece like \(\mathrm{H}^1(\text{elliptic curve})\). Recalling the invariants provided in Equations \eqref{invariants-box}, \eqref{invariants3-box}, we can build the following table labeling the missing atom from the computation $\lambda=0$:

\begin{table}[H]
    \centering
    \begin{tabular}{|c|c|c|c|}
     \hline $\lambda$  & $\rho$  & $\rho^G$ & $P$\\ \hline \hline
      $16$   & $1$ & $1$ & 1 \\ \hline
      $4$   &  $4$ & $2$ &4\\ \hline
    $0$  & $5$ & $3$ &  $t+3+t^{-1}$\\ \hline
    \end{tabular}
    \label{tab:invariants}
    \end{table}

Now assume, for contradiction, that $X$ is $G$-equivariantly birational to $\mathbb{P}^3$ with some linearizable action of $G$. By the Weak Equivariant Factorization Theorem, the atomic content of $X$ must be a sum of the atoms of $\mathbb{P}^3$ and of atoms of $G$-invariant subvarieties of codimension $\geq 2$ (points and curves). The atoms of $\mathbb{P}^3$ are all one-dimensional and consist only of Hodge classes, with $\rho^G=1$. The atoms of low-dimensional $G$-varieties are listed next (check, for instance, Lemma \ref{lem:atom-generators}):



\begin{table}[h!]
    \centering
    \begin{tabular}{|c|c|c|c|}
     \hline $\Z/2$-atom  & $\rho$  & $\rho^G$ & $P$ \\ \hline \hline
      point   & $1$ & $1$ & $1$ \\ \hline
      two points   &  $2$ & $1$ & $2$  \\ \hline
    elliptic curve  & $2$ & $2$ & $t+2+t^{-1}$ \\ \hline
atom of $X_0$ & $1$ & $1$ & $1$ \\ \hline
    \end{tabular}
    \end{table}

The presence of non-Hodge cohomology in $X$ (the term $h^{2,1}=1$) implies that the atom of an elliptic curve must appear in the decomposition of $X$. This atom contributes the vector $(\rho,\rho_G)=(2,2)$. Subtracting this from the $\lambda=0$ atom of $X$, the remaining atomic content to be explained is:
\[
(\rho,\rho_G)_{\mathrm{remaining}}=(5,3)-(2,2)=(3,1).
\]

This remaining vector must be a nonnegative integral linear combination of the basic point atoms: $(1,1)$ and $(2,1)$. However, it is impossible to write
\[
(3,1)=a(1,1)+b(2,1), \qquad a,b\in \mathbb{Z}_{\geq 0}.
\]
The only integer solution is $a=-1,b=2$, which is not allowed. This contradiction proves the theorem. \qedhere
\end{proof}

The former result provides an interesting dichotomy:
\begin{corollary}\label{thm:X(1,1,1,1)}
For any subgroup $G<\mathcal S_4$-action on $X(1,1,1,1)$ induced from a transitive action in $\{1, 2, 3, 4\}$, either
\begin{itemize}
\item[{\parbox[t]{2.5em}{(a)}}]%
\hspace{0.5em}the $G$-action on the set $\{1,2,3,4\}$ has a fixed point, and in this case $X(1,1,1,1)$ is $G$-birationally equivalent to 
$\mathbb P^3$ with a $G$-action;
\item[{\parbox[t]{2.5em}{(b)}}]%
\hspace{0.5em}the $G$-action on the set $\{1,2,3,4\}$ has no fixed points, and in this case $X(1,1,1,1)$ is not $G$-birationally equivalent to $\mathbb P^3$ with a $G$-action.
\end{itemize}
\end{corollary}
\begin{proof} 
Suppose we are in case $(a)$. Then, by projecting from a fixed factor, we get a $G$-equivariant birational map between $X(1,1,1,1)$ and $(\PP^1)^3$ with $G$-action permuting three factors. The latter space contains Zariski open part $(\mathbb A^1)^3$ where the $G$ action is by permutation of coordinates, and hence it extends to the action on $\mathbb P^3$. In the case $(b)$, the group $G$ contains (up to conjugation) the subgroup $\Z/2$ acting by the permutation $1\leftrightarrow 2, 3\leftrightarrow 4$, and we can refer to the previous result. \qedhere

\end{proof}

\begin{corollary}\label{cor:Kuznetsov}
    Let $X$ be a singular cubic threefold with four ordinary double points. Assume that $\mathcal S_4$ acts on $X$ exchanging these ordinary double points. Then $X$ is not birationally equivalent to $\mathbb P^3$ with \emph{any} $\mathcal S_4$-linearizable action.
\end{corollary}
\begin{proof}
The proof is accomplished when we observe that a slight modification of the proof of Proposition 4.5 in \cite{Kuznetsov_Prokhorov_2024} proves that such a cubic is $\mathcal S_4$-equivariantly birational to \(X(1,1,1,1)\) for the \(\mathcal S_4\)-action induced from the swapping factors \(\mathcal S_4\)-action in \((\PP^1)^4\).
\end{proof}
\end{example}

\ 

\begin{example}
\label{ex:multi-degree-times}
Let $X=X(1,1,1,1)$ and $G=\langle (12)(34)\rangle \cong \mathbb{Z}/2$, as in Theorem~\ref{thm:X(1,1,1,1)-withZ_2}. The action on $X\times \mathbb{P}^1$ is given by the action of $G$ on $X$ and the trivial action on $\mathbb{P}^1$. We claim that this $\mathbb Z/2$-action on $X\times \mathbb P^1$ is not linearizable.

The atomic content of $X\times \mathbb{P}^1$ consists of two copies of the atoms of $X$. Thus, we have two atoms with $(\rho,\rho^G)=(5,3)$. They contain the non-Hodge part of $X$, coming from $\mathrm{H}^{2,1}(X)$. Thus, each such atom must contain the atom of an elliptic curve, which contributes $(2,2)$. The remainder of the atom has invariants $(3,1)$ that can not correspond to a $G$-equivariant atom in $\mathsf{HAtoms}^{\mathbb Z/2}_{\dim \leq 1}$ (there are too many algebraic cycles). It cannot be a $G$-equivariant atom of a surface, as per \cite{katzarkovpantevyu}, if a surface atom has at least three Hodge classes, it must have at least three invariant Hodge classes\footnote{see the computation of Hodge atoms for surfaces, for instance, in the Introduction}. 
\end{example}

\

\subsection{Use of classical monodromy} 
\label{sec:monodrony}

Here, we combine B-side monodromy with Gromov-Witten theory to obtain a refined framework for rationality obstructions via atom theory, exemplified in Theorem \ref{thm:classical-monodromy} below. A concrete application is given in Example \ref{ex:product=with-line}.

\begin{theorem}\label{thm:classical-monodromy} Let $(X_t)_{t\in B}$ be an algebraic family of compact algebraic 3-dimensional varieties endowed with a generically free action of a finite group $G$, over a connected base $B$. Assume that 
\begin{enumerate}
    \item the action of $\mathsf{Hod}$ on $\mathrm H^\ast(X_t,\Q)$ is trivial for all $t\in B$, in other words, all rational cohomology classes of $X_t$ are Hodge classes,
    \item  the local system $\mathrm H^\ast(X_t,\Q)^G$ on $B$ is trivial,
    \item there exists an irreducible representation $\rho$ of the local system $\mathrm H^\ast(X_t,\overline \Q)$ of dimension strictly larger than the cardinality $|G|$ of $G$.
\end{enumerate}
    Then, for any $t\in B$, the $G$-variety $X_t$ is not $G$-equivariantly birationally equivalent to $\mathbb P^3$ endowed with a generically free linearizable $G$-action.
\end{theorem}
\begin{proof} Conditions $(i),(ii)$ imply that the fixed locus $B_{X_t}^{G\times \mathsf{Hod}}$ does not depend on $t$. Therefore, we have just one rigid analytic space, which we can denote $B_{univ}$.

Now we use the basic property of Gromov-Witten invariants: \textit{deformation invariance}, which, e.g., follows from the fact that these invariants can be defined in terms of $C^\infty$ symplectic geometry. This implies that for any $\mathbbm{k}$-point $b$ of $B_{univ}$,  the  operator $\boldsymbol{\kappa}_b$ acts  on  $\mathrm H^\ast(X_t,\mathbbm{k})$ in a covariantly constant way. The condition $(iii)$ implies that we have an eigenspace of  $\boldsymbol{\kappa}_b$ of dimension larger than $|G|$. Therefore, for each $t\in B$ we have a $G$-equivariant atom of dimension $>|G|$ on which $\mathsf{Hod}$ acts trivially. This atom can not come from $G$-varieties of dimension $0$ and $1$, see, e.g., Lemma \ref{lem:atom-generators}. 
\end{proof}

In fact, condition $(iii)$ can be weakened: it suffices to assume that there exists an irreducible representation $\rho$ of $G\times \pi_1(B,t_0)$ in $\mathrm H^\ast(X_{t_0},\overline\Q)$ which when considered as the $\overline\Q[G]$-module is not isomorphic to the direct summand of the permutation module, i.e., representation of $G$ induced from the trivial representation of a subgroup $H< G$.

\

\

\section{Atoms meet symbols: A new $\mathbb Z$-module}
\label{sec:atomes-meet-symbols}

While the theory of \(G\)-equivariant atoms offers powerful birational invariants derived from the interplay of Gromov-Witten theory and Hodge structures, a natural question arises: is this global perspective sufficient to resolve finer questions in \(G\)-equivariant birational geometry? Complementing the ``atomic'' view is the theory of modular symbols, which precisely captures the local geometry of the group action through the characters of \(G\)-representations on normal bundles to the fixed locus \(X^G\). The central philosophy throughout this section is that these two perspectives are not independent. 

Since a fundamental operation like a \(G\)-equivariant blowup alters both the global atomic content and the local fixed-locus geometry, a complete invariant must track both. We therefore construct a unified algebraic object, a combinatorial \(\mathbb{Z}\)-module  \(\mathcal{B}_3(\mathbb Z/p)^{\mathrm{comb}}\), whose generators are formal unions of atoms and symbols and whose relations are defined by their transformations under blowups. This combined module is designed to produce invariants that are fundamentally stronger than those derived from either the atomic or symbolic theory in isolation. From this \(\mathbb{Z}\)-module, we derive two \(\mathbb{Z}\)-valued \(\mathbb Z/p\)-birational invariants (Theorems \ref{thm:invariant}, \ref{thm:invariant-fine}), and provide applications (Examples \ref{ex:product=with-line} and \ref{ex:fixed-higher-genus-curve}).

\ 

Unless stated otherwise, we focus on irreducible smooth complex projective varieties \(X\) of dimension \(3\) that possess \(\mathbb Z/p\)-generically free actions by regular automorphisms. The choice for $\mathbb Z/p$ is merely for ease of notation and easier exposition. The entire story can be generalized to other finite groups, including nonabelian ones, once the necessary changes are made everywhere, for instance, following \cites{Tschinkel2022, Kresch-Tschinke-Structure-and-Operations, Kresch2022}.

\ 

\subsection{A recap of modular symbols}
\label{sec:maximal-stabilizers}

In this section, we briefly recall the construction of a \(\mathbb{Z}\)-module (denoted by $\mathcal B_n(G)$) introduced in \cite{kontsevich2021equivariant}, which assigns to smooth projective varieties $X$ with generically free regular actions by finite abelian groups $G$ certain symbolic equivariant birational invariants $\beta(X)\in \mathcal B_n(G)$.

Let \(G\) be a finite abelian group, and let its dual group be denoted by \(G^{\vee} = \mathrm{Hom}(G, \mathbb{C}^{\times})\). For any \(n \geq 1\), we define the \(\mathbb{Z}\)-module \(\mathcal{B}_n(G)\), generated by symbols \([a_1, \ldots, a_n]\), where \(\{a_i \in G^{\vee}\}_{i=1}^n\) satisfy \(\sum_{i=1}^n \mathbb{Z} a_i = G^{\vee}\), subject to the following relations:
\begin{center}\small
\noindent\fbox{%
    \parbox{0.9\linewidth}{%
        \begin{itemize}
            \item[(O)] \([a_1, \ldots, a_n] = [a_{\sigma(1)}, \ldots, a_{\sigma(n)}]\) for any permutation \(\sigma \in \mathcal{S}_n\).
            \item[(B)] For all \(2 \leq k \leq n\), any \(a_1, \ldots, a_k \in G^{\vee}\), and any \(b_1, \ldots, b_{n-k} \in G^{\vee}\) such that
            \[
            \sum_{i=1}^k \mathbb{Z} a_i + \sum_{j=1}^{n-k} \mathbb{Z} b_j = G^{\vee},
            \]
            we impose the relation:
            \[
           [a_1, \ldots, a_k, b_1, \ldots, b_{n-k}] = \sum_{1 \leq i \leq k, a_i \neq a_{i'}\,\forall i'<i} [a_1 - a_i, \ldots, a_i, \ldots, a_k - a_i, b_1, \ldots, b_{n-k}].
            \]
        \end{itemize}
    }%
}
\end{center}

The simplest identification of $\mathcal B_n(G)$ via an isomorphism with a ``more recognizable'' $\mathbb Z$-module happens for $n=1$. We have:
\[
\mathcal{B}_1(G) = 
\begin{cases}
    \mathbb{Z}^{\phi(m)} & \text{if } G = \mathbb{Z}/m, \\
    0 & \text{otherwise},
\end{cases}
\]
where \(\phi\) denotes Euler’s totient function,
\[\phi(m)=m\prod_{\text{distinct}~p|m}\left(1-\tfrac{1}{p}\right)\]

For $n=2$ the $\mathbb Z$-module $\mathcal B_2(\mathbb Z/m)$ is generated by symbols \([a_1, a_2]\) with \(a_1, a_2 \in \mathbb{Z}/m\) such that \(\gcd(a_1, a_2, m) = 1\), subject to:
\begin{center}\small
\noindent\fbox{%
    \parbox{0.9\linewidth}{%
\begin{enumerate}
    \item[(O)] \([a_1, a_2] = [a_2, a_1]\),
    \item[(B)] \([a_1, a_2] = [a_1 - a_2, a_2] + [a_1, a_2 - a_1]\) if \(a_1 \neq a_2\), and \([a, a] = [a, 0]\) whenever \(\gcd(a, m) = 1\).
\end{enumerate}
 }%
}
\end{center}
Table~\ref{table:B_2(H)} lists examples of \(\mathcal{B}_2(G)\) for various abelian groups \(G\), adapted from \cite{Tschinkel2024}.
\begin{table}[h!]
    \centering
    \small
    \begin{tabular}{|c|c|}
        \hline 
        \(m\) & \(\mathcal{B}_2(\mathbb{Z}/m)\) \\ 
        \hline\hline
        2 & 0 \\
        3 & \(\mathbb{Z}\) \\
        4 & \(\mathbb{Z}\) \\
        5 & \(\mathbb{Z}^2\) \\
        6 & \(\mathbb{Z}^2 \oplus \mathbb{Z}/2\) \\
        7 & \(\mathbb{Z}^3 \oplus \mathbb{Z}/2\) \\
        8 & \(\mathbb{Z}^3 \oplus \mathbb{Z}/4\) \\
        9 & \(\mathbb{Z}^5 \oplus \mathbb{Z}/3\) \\
        10 & \(\mathbb{Z}^4 \oplus (\mathbb{Z}/2)^2 \oplus \mathbb{Z}/6\) \\
        11 & \(\mathbb{Z}^6 \oplus \mathbb{Z}/5\) \\
        12 & \(\mathbb{Z}^7 \oplus \mathbb{Z}/8\) \\
        13 & \(\mathbb{Z}^8 \oplus \mathbb{Z}/7\) \\
        14 & \(\mathbb{Z}^7 \oplus (\mathbb{Z}/2)^4 \oplus \mathbb{Z}/12\) \\
        15 & \(\mathbb{Z}^{13} \oplus \mathbb{Z}/8\) \\
        16 & \(\mathbb{Z}^{10} \oplus (\mathbb{Z}/2)^2 \oplus \mathbb{Z}/16\) \\
        \hline 
    \end{tabular}
    \vspace{5pt}  
    \caption{Structure of \(\mathcal{B}_2(G)\) for selected finite abelian groups \(G\)}
    \label{table:B_2(H)}
\end{table}

\medskip

Let \(X\) be a complex smooth irreducible projective variety of dimension \(n \geq 2\), equipped with a birational generically free action of a finite abelian group \(G\).\footnote{The construction extends to varieties over any field of characteristic zero, but we restrict to the complex case for clarity.} After performing a \(G\)-equivariant resolution of singularities, we may assume that the \(G\)-action is through regular automorphisms.

Given a variety \(X\), we associate an element \(\beta(X) \in \mathcal{B}_n(G)\) as follows. Let
\[
X^G = \bigsqcup_{\alpha \in A} F_\alpha
\]
denote the fixed-point locus of \(G\), written as a disjoint union of smooth irreducible subvarieties \(F_\alpha \subset X\). Choose a point \(x_\alpha \in F_\alpha\). The action of \(G\) on the tangent space \(T_{x_\alpha}X\) yields a homomorphism \(\rho: G \to \mathrm{Aut}(T_{x_\alpha}X)\simeq\mathrm{GL}(n,\C)\). Since \(G\) is abelian, this representation decomposes as:
\[
\rho = \bigoplus_{i=1}^n a_{i,\alpha},
\]
where \(a_{i,\alpha} \in G^{\vee}\) are characters (well-defined up to permutation).
Since the \(G\)-action is generically free, we have:
\[
\sum_{\alpha} \mathbb{Z} a_{i,\alpha} = G^{\vee},
\]
independently of the choice of points \(x_\alpha\). The number of vanishing characters \(a_{i,\alpha} = 0\) corresponds to the dimension \(n_\alpha = \dim F_\alpha\). Thus, to each \(\alpha\), we associate a symbol:
\[
[a_{1,\alpha}, \ldots, a_{n,\alpha}] \in \mathcal{B}_n(G),
\]
and define:
\[
\beta(X) := \sum_\alpha [a_{1,\alpha}, \ldots, a_{n,\alpha}].
\]

\begin{theorem}[\cite{kontsevich2021equivariant}]\label{thm:kontsevich-pestun-tschinkel}
The class \(\beta(X) \in \mathcal{B}_n(G)\) is a \(G\)-birational invariant.
\end{theorem}

Theorem \ref{thm:kontsevich-pestun-tschinkel} follows from the \(G\)-equivariant Weak Factorization Theorem (see, e.g., Proposition 2.6 in \cite{Kresch2022}), which reduces the proof to verifying invariance under \(G\)-equivariant blowups. The relations defining $\mathcal B_n(G)$ are chosen conveniently to produce a birational invariant.

Tschinkel and Kresch, in a series of works \cites{kresch2019birational, Kresch2022, Kresch2022-Representation, kresch2023birational, kresch2022burnside, Kresch-Tschinke-Structure-and-Operations}, have produced finer symbol invariants that take into account the birational types of components of fixed loci. Moreover, they extend the theory to nonabelian groups, quasi-projective varieties, and algebraic orbifolds. One expects that in the same spirit of this paper, a construction merging atoms and their invariants can be obtained. The atom theory in the Deligne-Mumford stack setting is under development by the authors and shall appear elsewhere, see Section \ref{sec:CR} for preliminary progress.


\

\subsection{The definition of $\mathcal{B}_3(\mathbb Z/p)^{\mathrm{comb}}$}
\label{sec:B_3-comb}
Let $p$ be a prime number. We define the combinatorial $\Z$-module 
\[
\mathcal{B}_3(\Z/p)^{\mathrm{comb}}
\]
in an abstract manner, using generators derived from \emph{symbols} (S) and \emph{atoms} (A). The superscript $comb$ indicates that this is a \emph{combined} structure. Throughout, let $X$ be a smooth projective threefold equipped with a generically free action of $G = \Z/p$ by regular automorphisms. The \emph{generators} and \emph{relations} of $\mathcal{B}_3(G)^{\mathrm{comb}}$ are motivated by two sources:
\begin{enumerate}
    \item the geometry of the fixed-point set of the $G$-action, and
    \item $G$-equivariant atoms from \( \mathsf{HAtoms}^{\Z/p}_{\dim \leq 1} \) as these are the only atoms that can appear/disappear under blowups or blowdowns in dimension $3$.
\end{enumerate}

\medskip

Below, we provide a complete description of the truncated set 
\(\mathsf{HAtoms}^{\Z/p}_{\dim \leq 1}\).

\begin{lemma}\label{lem:atom-generators-2}
The following varieties with $\Z/p$-action have each only one local equivariant Hodge atom, and in this way we obtain \textit{all} elements of \(\mathsf{HAtoms}^{\Z/p}_{\dim \leq 1}\):
\begin{enumerate}
    \item a point with trivial $\Z/p$-action,
    \item  a free $\Z/p$-orbit of a point,
    \item  a curve $C$ of genus $\mathsf{g} \geq 1$ with trivial $\Z/p$-action,
    \item a curve $C$ of genus $\mathsf{g} \geq 1$ with a nontrivial $\Z/p$-action,
    \item $C\times$ (a free $\mathbb Z/p$ orbit of a point) with $\mathsf{g}(C)\geq 1$ and $\mathbb Z/p$ acting
trivially on $C$.
\end{enumerate}
\end{lemma}

\begin{proof} Use Lemma \ref{lem:atom-generators} and the fact that any action of $\Z/p$ on $\mathbb P^1$ is linearizable.
\end{proof}

\ 

Potentially, some of the atoms listed above give the same element of  \(\mathsf{HAtoms}^{\Z/p}_{\dim \leq 1}\). In order to distinguish them, we can use our main invariant, the isomorphism class of a $\overline\Q$-linear representation of $G\times\mathsf{Hod}$.

\begin{lemma}\label{lem:invariant-for-generators} For $\mathsf{g}\ge 2$, Hodge atoms $\boldsymbol{\alpha}\in \mathsf{HAtoms}^{\Z/p}_{\dim \leq 1}$ with Hodge polynomial 
$P_{\boldsymbol{\alpha}}(t)=\mathsf{g}\,t^{-1}+2+\mathsf{g}\,t$ and 
  such that $\Z/p$ acts trivially on the corresponding eigenspace can \textit{only} happen in the case $(iii)$ of Lemma \ref{lem:atom-generators-2}, for the curve $C$ of genus $\mathsf g$ endowed with the trivial $\Z/p$-action.
\end{lemma}
\begin{proof}
    The cases \((i),(ii)\) are excluded because they give wrong Hodge polynomials. The case \((v)\) is ruled out by the fact that the $\mathbb Z/p$-action is not trivial on the corresponding eigenspace -- the eigenspace is isomorphic to $2\mathsf{g}+2$-copies of the regular representation of $\mathbb Z/p$. The case \((iv)\) is excluded because any nontrivial $\Z/p$-action on a curve of genus  $\mathsf{g}\ge 2$ induces a nontrivial action on its cohomology.
\end{proof}

\begin{remark}\label{rem:distinct}
    The isomorphism class of a $\overline\Q$-linear representation of $G\times\mathsf{Hod}$ does not distinguish some elements of $\mathsf{HAtoms}^{\Z/p}_{\dim \leq 1}$ which we expect to be different. For example, if $X$ is an elliptic curve, then we can not distinguish by this invariant the trivial $\Z/p$-action, and the one by shifts by $p$-torsion points. Similarly, even in the case of the trivial action, we can not distinguish two isogenous elliptic curves, or more generally, two curves with isogenous Jacobians. However, the proposed enhancement in Section \ref{sec:CR} does indeed differentiate between such atoms.
\end{remark}

Now we are ready to give the definition of \(\mathcal{B}_3(\mathbb{Z}/p)^{\mathrm{comb}}\), which is quite heavy, and it is not used explicitly in the rest of the paper. The role of this definition is to provide a template that can be extended to other situations, e.g., including further geometric information related to the fixed loci, global cohomology of varieties with $\Z/p$-actions, e.g., of the type considered in Section   \ref{sec:geometric}.

The reader can skip the description of  \(\mathcal{B}_3(\mathbb{Z}/p)^{\mathrm{comb}}\) and go directly to Section \ref{sec:Z-invariants}, and see explicitly a self-contained example of atoms non-trivially interacting with symbols.

\begin{definition} \label{def:atoms-meet-symbols}
The $\mathbb{Z}$-module 
\[
\mathcal{B}_3(\mathbb{Z}/p)^{\mathrm{comb}}
\]
is generated by the following elements:

\noindent\textbf{(S) Symbols:}
\begin{defenum}
    \item \label{item:S1} $[a,b,c]$, where $a,b,c \in (\mathbb{Z}/p)^{\vee}$ are nonzero, with relations 
    \[
    [a,b,c] = [b,a,c] = [c,b,a].
    \]
    Geometrically, these correspond to symbols associated with the $\mathbb{Z}/p$-action on the tangent space of $X$ at isolated fixed points: the characters $a,b,c$ represent the weights of the $\mathbb{Z}/p$-normal bundle representation at a fixed point.

    \item $(\mathbb{P}^1,[a,b])$ with $a,b \in (\mathbb{Z}/p)^{\vee}$ nonzero and $[a,b]=[b,a]$.  
    Geometrically, these correspond to irreducible components of $X^{\mathbb{Z}/p}$ isomorphic to $\mathbb{P}^1$, with weights $(a,b)$ coming from the $\mathbb{Z}/p$-representation on the normal bundle of $\mathbb{P}^1$ in $X^{\mathbb{Z}/p}$.

    \item $(C,[a,b])$ with $a,b \in (\mathbb{Z}/p)^{\vee}$ nonzero and $[a,b]=[b,a]$, where $C$ is a curve of genus $\mathsf{g} \geq 1$.  
    Geometrically, these correspond to irreducible components of $X^{\mathbb{Z}/p}$ isomorphic to $C$, with weights $(a,b)$ arising from the $\mathbb{Z}/p$-representation on the normal bundle of $C$ in $X^{\mathbb{Z}/p}$.

    \item $([S],a)$ with $a \in (\mathbb{Z}/p)^{\vee}$ nonzero, where $[S]$ denotes the birational class of a surface $S$.  
    Geometrically, these correspond to birational types of irreducible components of $X^{\mathbb{Z}/p}$ that are surfaces, with weight $a$ coming from the $\mathbb{Z}/p$-representation on the normal bundle of $S$ in $X^{\mathbb{Z}/p}$.
\end{defenum}

\noindent\textbf{(pre-A) pre-atoms:} isomorphism classes of varieties of dimension $\le 1$ with $\Z/p$-action, of the following types:
\begin{enumerate}
    \item $(\mathrm{pt}\stackrel{\mathrm{trivial}}{\curvearrowleft} \Z/p )$
    \item \( (\Z/p \stackrel{\mathrm{free}}{\curvearrowleft} \Z/p )\)
    \item \( (C\stackrel{\mathrm{trivial}}{\curvearrowleft} \Z/p )\)
    \item \( (C\stackrel{\mathrm{nontrivial}}{\curvearrowleft} \Z/p )\)
    \item \( (C\times \Z/p \stackrel{\mathrm{free}}{\curvearrowleft} \Z/p )\)
\end{enumerate}
where $C$ is a curve of genus $\mathsf g\ge 1$.

\medskip

The generators are subject to the following relations:

 \ 
 
\noindent\textbf{(A) Atomic equivalence:}\footnote{This relation is not explicit as we do not know which pre-atoms give the same atom, see Remark \ref{rem:distinct}.} two pre-atoms are equivalent iff the corresponding elements of \(\mathsf{HAtoms}^{\Z/p}_{\dim \leq 1}\) coincide with each other.

\

\noindent\textbf{(Bl0) Blowup in points:}
\begin{defsubenum}
    \item[(Bl0-a)] \emph{Blowup in free $\mathbb{Z}/p$-orbits:}
    \[ 2(\Z/p \stackrel{\mathrm{free}}{\curvearrowleft} \Z/p ).
    \]

    \item[(Bl0-b)] \emph{Blowup in fixed points with normal weights $a,b,c$:}
    \begin{align}
    &2(\mathrm{pt}\stackrel{\mathrm{trivial}}{\curvearrowleft} \Z/p ) \\
    &\quad + \begin{cases}
    [a,b-a,c-a] + [a-b,b,c-b] + [a-c,b-c,c] - [a,b,c] & \text{if } a \neq b \neq c, \\
    (\mathbb{P}^1,[a,c-a]) + [a-c, a-c,c] - [a, a,c] & \text{if } a = b \neq c, \\
    ([\mathbb{P}^2],a) - [a, a, a] & \text{otherwise}.
    \end{cases}
    \end{align}
\end{defsubenum}

\noindent\textbf{(Bl1) Blowup in a fixed curve $C$ which is itself a fixed-point component with normal weights $a,b\,$:}
\begin{defsubenum}
    \item[(Bl1-0)] \emph{$C$ has genus $\mathsf g=0$.}
    \[
    \begin{cases}2(\mathrm{pt}\stackrel{\mathrm{trivial}}{\curvearrowleft} \Z/p )
    + (C,[a,b-a]) + (C,[a-b,b]) - (C,[a,b]), & \text{if } a \neq b, \\2(\mathrm{pt}\stackrel{\mathrm{trivial}}{\curvearrowleft} \Z/p )
    + ([C\times\mathbb{P}^1],a) - (C,[a,a]), & \text{if } a = b.
    \end{cases}
    \]
    \item[(Bl1-1)] \emph{$C$ has genus $\mathsf g\ge 1$.}
    \[
    \begin{cases}
    (C\stackrel{\mathrm{trivial}}{\curvearrowleft}\mathbb{Z}/p)
    + (C,[a,b-a]) + (C,[a-b,b]) - (C,[a,b]), & \text{if } a \neq b, \\
    (C\stackrel{\mathrm{trivial}}{\curvearrowleft}\mathbb{Z}/p)
    + ([C\times\mathbb{P}^1],a) - (C,[a,a]), & \text{if } a = b.
    \end{cases}
    \]
\end{defsubenum}

\noindent\textbf{(Bl2) Blowup in a fixed curve $C$  contained in a surface $S$ which is a fixed-point component with the normal weight $a$:}
\begin{defsubenum}
    \item[(Bl2-0)] \emph{$C$ has genus $\mathsf g=0$.}
    \[2(\mathrm{pt}\stackrel{\mathrm{trivial}}{\curvearrowleft}\mathbb{Z}/p) 
    + (\mathbb{P}^1,[a,-a])\]
\item[(Bl2-1)] \emph{$C$ has genus $\mathsf g\ge 1$.}
    \[   (C\stackrel{\mathrm{trivial}}{\curvearrowleft}\mathbb{Z}/p) + (C,[a,-a])
    \]
\end{defsubenum}

\noindent\textbf{(Bl3) Blowup in a curve that is not fixed:}
\begin{defsubenum}
    \item[(Bl3-a)] \emph{Blowup in $p$-components cyclically permuted.}
    \[
    \begin{cases}2 (\Z/p \stackrel{\mathrm{free}}{\curvearrowleft} \Z/p ), & \text{if } \mathsf{g}(C)=0\\
    (C\times \Z/p \stackrel{\mathrm{free}}{\curvearrowleft} \Z/p ), & \text{if } \mathsf{g}(C)\geq 1.
    \end{cases}
    \]
    
    \item[(Bl2-b)] \emph{Blowup in a connected curve $C$ with some fixed points along it.}
    \begin{align}
    \sum_{\substack{\text{fixed pts} \\ a \neq 0}} 
    \begin{cases}
    (\mathbb{P}^1,[a,b]) - [a,b,b], & \text{if } b=c\neq 0, \\
    [a,b-c,c] + [a,b,c-b] - [a,b,c], & \text{if } 0 \neq b \neq c \neq 0, \\
    [a,c,-c], & \text{if } b=0, c \neq 0, \\
    \text{nothing}, & \text{if } b=c=0.
    \end{cases}\\
    +\begin{cases}2(\mathrm{pt}\stackrel{\mathrm{trivial}}{\curvearrowleft}\mathbb{Z}/p) & \text{if } \mathsf g(C)=0\\
    (C\stackrel{\mathrm{nontrivial}}{\curvearrowleft} \Z/p ) & \text{if } \mathsf g(C)\ge 1.
    \end{cases}
    \end{align}
  Here $(a,b,c)$ are normal weights at the fixed point $p\in C$, and $a\ne 0$ is the weight of the $\Z/p$-action on $T_p C$.  
    
\end{defsubenum}
\end{definition}

\ 

\subsection{$\mathbb Z$-valued combined invariant}
\label{sec:Z-invariants}
In this section, we prove the following:
\begin{theorem}\label{thm:invariant}
Let \(X\) be an irreducible smooth projective threefold equipped with a regular generically free \(\mathbb{Z}/p\)-action, where \(p\) is prime. Fix an in integer \(\mathsf{g} \geq 2\), and denote by \(X^{\mathbb{Z}/p}\) the fixed-point locus of the action. Then the following quantity defines a \(\mathbb{Z}/p\)-birational invariant ($C$ denotes an arbitrary smooth projective curve of genus \(\mathsf g\)):
\begin{equation}
\begin{split}
  &-\,\#\bigl\{\text{components of }X^{\Z/p}\cong C\bigr\}
   \;-\;2\,\#\bigl\{\text{components of $X^{\Z/p}$ birational to }C\times\PP^1\bigr\}\\
  &\quad+\;\#\Bigl\{\Z/p\text{-equivariant atoms }\boldsymbol{\alpha} \;\Bigm|\;
       P_{\boldsymbol{\alpha}}(t)=\mathsf{g}\,t^{-1}+2+\mathsf{g}\,t,\\
  &\qquad\qquad\text{and }\Z/p\text{ acts trivially on the corresponding eigenspace}
    \Bigr\}.
\end{split}
\end{equation}
\end{theorem}

We also have a companion result, providing a finer invariant that depends on the isogeny class of an abelian variety.

\begin{theorem}\label{thm:invariant-fine} Under the assumptions of the previous theorem, let us also fix a genus $\mathsf g\geq 2$ curve $C_0$. Then the following expression is a \(\mathbb{Z}/p\)-birational invariant.
\begin{equation}
\begin{split}
  &-\;\#\bigl\{
    \text{components of } X^{\mathbb{Z}/p} \cong C  
  \bigr\} 
   -\; 2\,\#\bigl\{
    \text{components of $X^{\mathbb Z/p}$ birational to } C \times \mathbb{P}^1  
  \bigr\} \\
  &\quad +\; \#\Bigl\{
    \mathbb{Z}/p\text{-equivariant atoms } \boldsymbol{\alpha} 
    \;\Bigm|\; 
    \mathbb{Z}/p \text{ acts trivially on the corresponding eigenspace,} \\
  &\qquad\quad 
    \text{and the isomorphism class of the } \overline{\mathbb{Q}}\text{-linear } 
    \mathsf{Hod}\text{-representation associated to } \boldsymbol{\alpha} \\
  &\qquad\quad 
    \text{is equivalent to that on } \mathrm{H}^{\ast}(C_0)
  \Bigr\}.
\end{split}
\end{equation}
Here $C$ denotes arbitrary curve of genus $\mathsf g$ such  that $\operatorname{Jac}(C)$ is isogenous to $\operatorname{Jac}(C_0)$.
\end{theorem}

\

We can slightly modify the statements of Theorem \ref{thm:invariant} by specifying that the atoms under consideration belong to \(\mathsf{HAtoms}^{\Z/p}_{\dim \leq 1}\), in order to be consistent with our approach to the definition of $\mathcal{B}_3(\mathbb{Z}/p)^{\mathrm{comb}}$. The statement remains true; the change does not affect the proofs.

\

\begin{proof}[Proof of Theorem \ref{thm:invariant}]
Let
$$I(X) = I_1(X)+2I_2(X)-I_3(X),\quad\text{where}$$
\begin{align}
I_1=&\#\left\{ \text{components of } X^{\mathbb{Z}/p} \cong C \right\} \\
I_2=&\#\left\{ \text{components birational to } C \times \mathbb{P}^1 \right\} \\
I_3=& \#\Bigl\{\Z/p\text{-equivariant atoms }\boldsymbol{\alpha} \;\Bigm|\;
       P_{\boldsymbol{\alpha}}(t)=\mathsf{g}\,t^{-1}+2+\mathsf{g}\,t,\\
  &\qquad\qquad\text{and }\Z/p\text{ acts trivially on the corresponding eigenspace}\Bigr\}.
\end{align}

We prove that \( I(X) \) is preserved under all \( \mathbb{Z}/p \)-equivariant blowups by considering all possible cases for the center of the blowup.

\medskip

\noindent\textbf{Case 1: Blowing up a free orbit \( Z \subset X \setminus X^{\mathbb{Z}/p} \) consisting of $p$ points.} \\
None of summands $I_i(X),\,i=1,2,3$ will change.

\medskip
\noindent\textbf{Case 2: Blowing up in an isolated fixed point with normal weights $a_1,a_2,a_3\in \Z/p^\vee,\,a_1,a_2,a_3\ne 0$.}
There are three subcases, depending on whether all $a_i$ are distinct, two coincide with each other, or all coincide. None of the summands $I_i(X),\,i=1,2,3$ will change in any of the three subcases.

\medskip

\noindent\textbf{Case 3: Blowing up in a nonconnected curve  \( Z \subset X \setminus X^{\mathbb{Z}/p} \) consisting of $p$ components cyclically permuted by $\mathbb Z/p$.} \\
Clearly the summands $I_i(X),\,i=1,2$ will not change.
  \begin{itemize}
      \item if components of $Z$ are of genus $0$, clearly the atomic invariant $I_3$ does not change.
      \item if components of $Z$ are of genus $>0$, we have only one new equivariant atom, and the action of $G$ on cohomology $\mathrm{H}^{\ast}(Z)$ is not trivial, hence $I_3(X)$ does not change.
  \end{itemize}
\medskip
\noindent\textbf{Case 4: Blowing up in  a connected $\Z/p$-invariant curve \( C\) such that $\Z/p$ acts on it non-trivially}. 
The curve $C$ has finitely many intersection points with the fixed locus $X^{\Z/p}$, and $I_i(X),i=1,2$ do not change. 
\begin{itemize}
      \item if the genus of  $C$ is $0$,  the atomic invariant $I_3$ does not change as we add two atoms of $\PP^1$,
      \item if the genus of $C$ is $1$, the added atom does not contribute to $I_3$,
      \item if the genus of $C$ is $>1$, then the action of $\Z/p$ on $\mathrm{H}^1(C)$ is nontrivial, so the added atom does not contribute to $I_3$.
  \end{itemize}
\medskip

\noindent\textbf{Case 5: Blowing up in a component $F_\alpha$ of the fixed locus, which is a curve \( C \subset X^{\mathbb{Z}/p} \) of genus $0$ with normal weights $a_1,a_2,\in \Z/p^\vee,\,a_1,a_2\ne 0$.}
There are two subcases, depending on whether $a_i$ are distinct or coincide with each other. In both subcases, none of the summands $I_i(X),\,i=1,2,3$ will change.

\medskip

\noindent\textbf{Case 6: Blowing up in a component $F_\alpha$ of the fixed locus, which is a curve \( C \subset X^{\mathbb{Z}/p} \) of genus $\ne 0,\mathsf g$ with normal weights $a_1,a_2,\in \Z/p^\vee,\,a_1,a_2\ne 0$.}
There are two subcases, depending on whether $a_i$ are distinct or coincide with each other. In both subcases, none of the summands $I_i(X),\,i=1,2,3$ will change.

\medskip
\noindent\textbf{Case 7: Blowing up in a component $F_\alpha$ of the fixed locus, which is a curve \( C \subset X^{\mathbb{Z}/p} \) of genus $\mathsf g$ with normal weights $a_1,a_2,\in \Z/p^\vee,\,a_1,a_2\ne 0$.}
\begin{itemize}
    \item If \( a_1 \neq a_2 \), the exceptional divisor \( E \cong \mathbb{P}(\mathcal{N}_{X}C) \) is birational to \( C \times \mathbb{P}^1 \), and the fixed locus inside \( E \) consists of two disjoint sections \( C_1 \) and \( C_2 \), each isomorphic to \( C \). In this process, an additional atom corresponding to curve $C$ with the trivial action appears. We have the following changes:
    $$\Delta I_1=2-1=1,\Delta I_2=0,\Delta I_3=1\implies \Delta I=(1)+2\cdot(0)-(1)=0$$
    \item If \( a_1 = a_2 \), the exceptional divisor is again birational to \( C \times \mathbb{P}^1 \), but the original curve \( C \) is replaced by a component birational to \( C \times \mathbb{P}^1 \). An additional atom corresponding to curve $C$ with the trivial action is created. We have:
    \[
    \Delta I_1=-1,\Delta I_2=+1,\Delta I_3=1\implies \Delta I=(-1)+2\cdot (1)-(-1)=0
    \]
\end{itemize}

\medskip
\noindent\textbf{Case 8: Blowing up in a point lying on a component \(F_\alpha\subset X^{\mathbb{Z}/p} \)  which is a curve.} There are two subcases depending in whether normal weights  $a_1,a_2$ of $F_\alpha$ coincide or not. None of the summands $I_i(X),\,i=1,2,3$ will change in both cases.

\medskip

\noindent\textbf{Case 9: Blowing up in a point lying on a component \( F_\alpha=S \subset X^{\mathbb{Z}/p} \) which is a surface.} Again, none of summands $I_i(X),\,i=1,2,3$ will change.

\medskip
\noindent\textbf{Case 10: Blowing up in  a  curve \( C\) lying inside a fixed surface \( S \subset X^{\mathbb{Z}/p}\).} \\
\begin{itemize}
    \item the genus of $C$ is not equal to $\mathsf g$. Clearly $\Delta I_i(X)=0, i=1,2,3$.
    \item the genus of $C$ is equal to $\mathsf g$. The normal bundle of \( C \) inside \( X \) splits as \(\mathcal{N}_{S}C \oplus \mathcal{O}_C\), with weights $a_1=0,a_2\ne 0$. A new component of fixed locus appears, isomorphic to $C$; also, an atom corresponding to curve $C$ with the trivial action is created. The change is:
    $$\Delta I_1=1,\Delta I_2=0,\Delta I_3=1\implies \Delta I=(1)+2\cdot(0)-(1)=0$$
\end{itemize}

Since every \( \mathbb{Z}/p \)-equivariant birational map can be decomposed into such blowups by the Weak Factorization Theorem, and \( I(X) \) remains unchanged under each step, it follows that \( I(X) \) is a \(\mathbb{Z}/p\)-birational invariant.
\end{proof}

\begin{proof}[Proof of Theorem \ref{thm:invariant-fine}]
One can repeat the analysis of various cases as in the previous proof. Only in Cases 7 and 10 does the summand $I_3$ change; we have arranged the definition so that the total sum remains constant.
\end{proof}
\begin{remark}
    The use of Jacobians \textit{up to isogeny} is an artifact of the modern state of affairs, as we have at the moment the isomorphism class of the $\bar{\mathbb{Q}}$-linear representation of $\mathsf{Hod}$ as a rigorously defined invariant of an atom.
\end{remark}

\

We conclude this section with examples that highlight the strength of the developed theory. Both present families of smooth projective threefolds with $G$-actions whose non-linearizability can not be captured by the means of symbols. It is worth noting, however, that the monodromy-type result related to atom theory (Theorem \ref{thm:classical-monodromy}) indeed succeeds in verifying the non-linearizability. However, the ``atoms meet symbols'' invariants given in Theorems \ref{thm:invariant},~\ref{thm:invariant-fine} are finer in the sense they distinguish $ G$-equivariant birationality for distinct members in the presented families.


\begin{example}\label{ex:product=with-line}
Let $S$ be a smooth rational surface,  endowed with a nontrivial $\Z/p$-action and such that the fixed locus $S^{\Z/p}$ contains at least one curve $C$ of genus $\mathsf g\ge 2$. Also, let us fix an action of $\Z/p$ on $\mathbb P^1$, possibly trivial. We define
 $$X:=S\times \mathbb P^1,\text{ with the diagonal }\Z/p\text{-action}$$
 Then the invariant of $X$ defined in Theorem \ref{thm:invariant} is nonzero, hence $X$ is not $\Z/p$-birational to $\mathbb P^3$ with a linearizable action. 

 Indeed, the condition on cohomology of $S$ implies that the Hodge polynomial of $X$ is constant, equal to $2\chi(S)$, and it immediately implies that there are no $\Z/p$-equivariant atoms of $X=S\times\mathbb P^1$ with Hodge polynomial $\mathsf g t^{-1}+2+\mathsf g t$. The fixed locus $X^{\Z/p}$ contains either product surfaces $C\times\mathbb P^1$, or two copies of $C$, depending on whether the action of $\Z/p$ on the factor $\mathbb P^1$ is trivial or not.

 There are many examples of surfaces $S$ satisfying our condition, at least in the case $p=2$. One can take, e.g., the double cover of $\mathbb P^2$ ramified at a smooth quartic curve, or the surface   whose open affine part is given by the equation 
 $$x_1 x_2=Q(x_3),\quad \text{with the involution }x_1\leftrightarrow x_2$$
 where $Q$ is any polynomial of degree $2k\ge 4$ with simple roots.

 Also, in the case $p=2$, dealing \textit{solely} with symbols (and keeping track of birational types of connected components $F_\alpha$ of the fixed locus $X^{\Z/2}$), we can not detect $\Z/2$-equivariant non-rationality. Indeed, if $X$ is a threefold with a $\Z/2$ action and if $F_\alpha$ is an
irreducible component of the fixed point locus, then this component
contributes to $\beta(X)$ the summand $[1,1,0]$ when $\dim F_\alpha = 1$,
and the summand $[1,0,0]$ when $\dim F_\alpha = 2$. However, both $[1,1,0]$
and $[1,0,0]$ are equal to zero in ${\mathcal B}_3(\Z/2)$. Indeed, let $Y$ be any smooth
threefold with a $\Z/2$ action, such that the fixed point locus contains
a surface $S \subset Y$. Let $C \subset S$ be a smooth curve. Then $\Z/2$
acts with weights 0 and 1 on $\mathcal{N}_YC$. Consider the blowup $\mathrm{Bl}_C(Y)$ with
exceptional divisor $\mathbb{P}(\mathcal{N}_YC)$. The components of $(\mathrm{Bl}_C(Y))^{\Z/2}$ are
the components of $(Y - S)^{\Z/2}$, plus the strict transform of $S$, plus
one extra curve component, which is the section of $\mathbb{P}(\mathcal{N}_YC)$
corresponding to the character 1. The additional component has normal
bundle on which $\Z/2$ acts with characters $(1,1)$, and hence
$\beta(\mathrm{Bl}_C(Y)) = \beta(Y) + [1,1,0]$. Since $\beta$ is a $\Z/2$-birational
invariant, this shows that $[1,1,0] = 0$ in $B_3(\Z/2)$. Similarly, let $Y$
be a smooth threefold with a $\Z/2$ action, such that the fixed point
locus contains a curve component $C \subset Y$. Then, in $\mathrm{Bl}_C(Y)^{\Z/2}$
we have all components of $(Y-C)^{\Z/2}$ plus one extra surface
component, namely the exceptional divisor $E = \mathbb{P}(\mathcal{N}_YC)$. Thus
$\beta(\mathrm{Bl}_C(Y)) = \beta(Y) - [1,1,0] + [1,0,0]$. Since $[1,1,0] = 0$ in
${\mathcal B}_3(\Z/2)$ and $\beta$ is a birational invariant we conclude that $[1,0,0] = 0$
in ${\mathcal B}_3(\Z/2)$ as well.

 Returning to atoms, let us see how one can use Theorem \ref{thm:classical-monodromy} to verify the non-linearizability of the examples of the form $S\times \mathbb P^1$ with $\mathbb Z/2$-actions discussed above. We provide a detailed analysis of the case when $S$ has an open affine part given by the equation $x_1x_2=Q(x_3)$. The other examples can be treated analogously. 

It is immediate to check that $\chi(S)=4+2k$ and $\chi(S/\mathbb Z/2)=4$, and that we have the following Hodge decompositions for cohomologies of $S$ and $S\times \mathbb P^1$, where we colored in \textcolor{red}{red} the invariant classes:
\[
\begin{array}{c}
    \textcolor{red}{1} \\
    \textcolor{red}{2} + 2k \\
    \textcolor{red}{1}
\end{array} \quad \quad 
\begin{array}{c}
    \textcolor{red}{1} \\
    \textcolor{red}{3} + 2k \\
    \textcolor{red}{3} + 2k \\
    \textcolor{red}{1}
\end{array}
\]

A natural family for deformation is that given by the moduli space $B$ of the polynomials $Q$ of degree $2k$ with no multiple roots up to scaling and translation of the variable. Over $B$, we have a universal family of surfaces $\pi_S: S\to B$. Consider the family $\pi: \mathcal X \to B$ by setting $\mathcal{X} = S \times_B (B \times \mathbb P^1)$.

It is straightforward from the above Hodge diamonds that all rational cohomology classes on the fiber $X_t$ are Hodge classes. Also, the non-invariant part of the cohomology of the fibers concentrates the monodromy action: the fundamental group $\pi_1(B)$-action defines a monodromy representation of the symmetric group $\mathcal S_{2k}$ by permutting the roots of $Q$. Thus, the local system on the $G$-invariant cohomology is trivial, verifying the first and second items in Theorem \ref{thm:classical-monodromy}.

 The standard representation of the symmetric group $\mathcal S_m$ on a vector space of dimension $m-1$ is known to be irreducible for $m \ge 2$. In our case, $m=2k$. Since we assume $2k \ge 4$, we have an irreducible representation of $\mathcal S_{2k}$ of dimension $2k-1>2=|\mathbb Z/2|$ for $k\geq 2$.


\end{example}

\

\begin{example}\label{ex:fixed-higher-genus-curve}
Fix $k\ge 2$. Let $P\in\mathbb{C}[x]$ be a monic polynomial of degree $3k$ with only simple roots. Consider the affine threefold
\[
X:\qquad x_1x_2x_3 = P(x_4)\subset \mathbb{C}^4,
\]
equipped with the cyclic action of $\mathbb{Z}/3$ that permutes the coordinates $x_1,x_2,x_3$ and fixes $x_4$. 

Embed $X$ into the product $((\mathbb{P}^1)_{[X_1:Y_1]}\times(\mathbb{P}^1)_{[X_2:Y_2]}\times(\mathbb{P}^1)_{[X_3:Y_3]})\times(\mathbb{P}^1)_{[X_4:Z_4]}$ by the standard affine charts
\[
x_i=\frac{X_i}{Y_i}\;(i=1,2,3),\qquad x_4=\frac{X_4}{Z_4}.
\]
Clearing denominators in the equation $\prod_{i=1}^3(X_i/Y_i)=P(X_4/Z_4)$ yields a bihomogeneous equation
\[
X_1X_2X_3\,Z_4^{3k} \;=\; Q(X_4, Z_4)\,Y_1Y_2Y_3,
\]
where $Q(X_4, Z_4) = Z_4^{3k}P(X_4/Z_4)$ is homogeneous of degree $3k$. This defines a hypersurface
\[
\overline X\;\subset\; (\mathbb{P}^1)^3\times\mathbb{P}^1
\]
defined by a single bihomogeneous polynomial of bidegree $((1,1,1),3k)$. The $\mathbb{Z}/3$-action extends regularly to $\overline X$ by cyclically permuting the factors $(X_i:Y_i)$ for $i=1,2,3$ and fixing $(X_4:Z_4)$.

In the affine chart $X_4\neq0$ (set $X_4=1$) we introduce
\[
y_4:=\frac{Z_4}{X_4}=\frac{1}{x_4},\qquad y_i:=\frac{X_i}{Y_i\,x_4^{\,k}}\;(i=1,2,3),
\]
so that $x_i=y_i x_4^{k}$. Substituting into $x_1x_2x_3=P(x_4)$ and dividing by $x_4^{3k}$ gives 
\begin{equation}\label{eq:chart-infty}
 y_1y_2y_3 \;=\; \frac{P(x_4)}{x_4^{3k}} \;=\; P\bigl(1/y_4\bigr)\,y_4^{3k}.
\end{equation}
Since $P$ is monic of degree $3k$, the right-hand side tends to $1$ as $y_4\to0$. Hence, to leading order near $x_4=\infty$, the hypersurface is approximated by the affine surface
\[
y_1y_2y_3=1,
\]
which is independent of the parameter $x_4$ to leading order.

The $\mathbb{Z}/3$-action permutes the coordinates $(y_1,y_2,y_3)$ in this chart. The fixed locus of the permutation action inside the locus $y_1y_2y_3=1$ is cut out by $y_1=y_2=y_3$, so one must solve
\[
y_1^3 = 1.
\]
Thus, over the point $x_4=\infty$, there are three fixed points given by the three cube roots of unity. Tracking these fixed points as $x_4$ varies gives a compact curve
\[
C=\{(x,t)\in\mathbb{C}^2:\; x^3 = P(t)\}\subset\overline X,
\]
where we use $x$ for the common value of $x_1=x_2=x_3$ on the fixed locus and $t$ for the coordinate $x_4$.

Consider the natural projection
\[
\pi:\; C\longrightarrow \mathbb{P}^1_t,\qquad (x,t)\mapsto t,
\]
where $C$ is the (projective) normalization of the affine curve $x^3=P(t)$. This is a degree $3$ cover. For each simple root $t_0$ of $P$ the local equation is $x^3=a(t-t_0)$ with $a\neq0$, hence there is a single point on $C$ above $t_0$ with ramification index $e=3$ (contribution $e-1=2$). There are $3k$ such simple branch points. At $t=\infty$, the polynomial $P$ is monic of degree $3k$, so the cover is unramified at infinity (the three sheets behave like $x\sim t^k$ and are distinct). The Riemann-Hurwitz formula, therefore, gives
\[
2\mathsf{g}(C)-2 \;=\; 3\cdot(-2) + \sum_{p\in C}(e_p-1) \;=\; -6 + 3k\cdot 2,
\]
whence
\[
\mathsf{g}(C)=3k-2\ge 4.
\]

At a point of the fixed curve $C$ the tangent space of the ambient threefold splits as the tangent to $C$ (the diagonal direction in the $x_1,x_2,x_3$-coordinates together with the $t$-direction along $C$) plus the normal plane which is the two-dimensional subspace of variations of $(x_1,x_2,x_3)$ orthogonal to the diagonal. The permutation action of $\mathbb{Z}/3$ on this two-dimensional subspace has no trivial summand. It splits over $\mathbb{C}$ as the sum of the two nontrivial characters of $\mathbb{Z}/3$ (which we record as weights $1$ and $2$ modulo $3$). Concretely, the representation of a generator $\zeta=e^{2\pi i/3}$ on the normal lines has eigenvalues $\zeta$ and $\zeta^2$. Thus, the normal weights are $1$ and $2$.

The hypersurface $\overline X\subset(\mathbb{P}^1)^3\times\mathbb{P}^1$ may be singular along coordinate strata where several of the homogeneous coordinates $Y_i$ or $Z_4$ vanish simultaneously. Observe, however, that the fixed curve $C$ constructed above is contained in the affine chart where $Y_1Y_2Y_3\neq0$ and $X_4\neq0$ (equivalently, the affine chart in Equation \eqref{eq:chart-infty}). Indeed, the fixed equation forces the $y_i$ coordinates solving $y_1^3=1$ to be nonzero, so $Y_i\neq0$ at points of $C$. Consequently, the curve $C$ lies in the smooth locus of the ambient product and is disjoint from those boundary strata where singularities can occur.

It follows that one can choose an explicit $\mathbb{Z}/3$-equivariant resolution of singularities obtained by blowing up toric centers supported in the complement of the chart containing $C$. Such blowups do not meet $C$ and therefore introduce no new fixed points along $C$. Performing these blowups (and normalizations) produces a smooth projective $\mathbb{Z}/3$-variety $\widetilde X$ whose fixed locus contains the curve $C$ (with the same genus and the same normal weights) and no additional curves.

Under the explicit compactification and equivariant resolution given above, the fixed locus of the $\mathbb{Z}/3$-action on $\widetilde X$ consists of the smooth curve $C$ of genus $3k-2$, with normal weights $1$ and $2$. To compute the invariant in Theorem \ref{thm:invariant}, one notices the following: 
\begin{enumerate}
    \item The term related to symbolic generator is \((C, [1, 2])\) contributes $-1$
    \item There are no terms related to atoms. Indeed, without the group action, the variety is rational and its cohomology contains only Hodge classes (\(\mathrm{H}^{p,q}(X) = 0\) whenever \(p-q \neq 0\)). 
\end{enumerate}
Thus, we get \(-1\) as value for this \(\mathbb{Z}\)-invariant. 

Finally, notice that regardless of the chosen \(\mathbb{Z}/3\)-linearizable action, the invariant for \(\mathbb{P}^3\) will always be \(0\). Consequently, the compactification \(\widetilde{X}\) is not \(\mathbb{Z}/3\)-birationally equivalent to \(\mathbb{P}^3\) with any linear \(\mathbb{Z}/3\)-action.

\

Notice that working with symbols only, we will not be able to deduce the non-\( G \)-linearizability. Indeed, the fixed locus is the curve $C$ with normal weights $1,2$. 
After the blowup with center at $C$, we obtain \textit{two} copies of $C$, again with normal weights $1,2$. Hence, any symbolic invariant is equal to twice itself, therefore vanishes\footnote{Alternatively, similar to what we discussed at the end of Example \ref{ex:product=with-line}, we can blow up some $\Z/3$-variety $X$ in a curve $C$ contained in a 2-dimensional component of the fixed locus $X^{\Z/3}$. The resulting variety has one extra connected component of the fixed locus, a copy of $C$ with the normal weights 1,2. This argument also shows the vanishing of the purely symbolic contribution of $C$ with the normal weights 1,2.}. However, arguing similarly to what is done at the end of Example \ref{ex:product=with-line}, one can apply Theorem \ref{thm:classical-monodromy} to conclude the non-linearizability of the $\mathbb Z/3$-action in the presented example using only atom theory. One of the important aspects of the invariant computed here lies in the fact that for different degrees $k$ (and different polynomials) our varieties could be $\mathbb Z/3$-birationally equivalent to each other. However, the invariant in Theorems \ref{thm:invariant}, \ref{thm:invariant-fine} distinguishes different varieties, at least for different $k$, or for the same $k$ and non-isogenous Jacobians.
\end{example}

\ 

\ 

\section{Atomic decomposition for cohomology of fixed loci, via Chen-Ruan theory}
\label{sec:CR}

Here, we will present another potential extension of the atom theory, directly relating atoms to the cohomology of fixed loci. This approach is conditional on the not-yet-proven blow-up formula for orbifolds (smooth projective  Deligne-Mumford stacks).

\ 

\subsection{Reminder on Chen-Ruan cohomology for global quotients}

Let $G$ be a finite group acting on a smooth projective variety $X$. The Chen-Ruan, or stringy cohomology of the global quotient stack $[X/G]$, is defined as the cohomology of the \text{inertia stack} $\mathcal I_{[X/G]}$, that is, the quotient of 
$$\{(g,x)\in G\times X| \,gx=x\}    $$
by the action of $G$:
$$h\cdot (g,x):=(hgh^{-1},hx).$$
Categorically, in de Rham realization, it is the periodic cyclic homology of the category of perfect complexes on $[X/G]$, or equivalently, the category of $G$-equivariant perfect complexes on $X$ \cite{baranovsky2003orbifold, lupercio2004inertia}.

Explicitly, the Chen-Ruan cohomology is
\begin{equation}\mathrm{H}^\ast_{CR}([X/G]):=\left(\bigoplus_{g\in G}\rH^\ast(X^g)\right)^G\simeq \bigoplus_{g_i}\rH^\ast(X^{g_i})^{C(g_i)}\label{formulaCR}\end{equation}
where $\{g_i\}\subset G$ is a set of representatives for the conjugacy classes in $G$, and $C(g_i)\subset G$ is the centralizer of $g_i$. The summand corresponding to the identity element $g=1$ is $\rH^\ast(X)^G$, known as the \textit{untwisted sector}. Gromov-Witten invariants in this context are constructed from moduli spaces of G-equivariant stable maps.

The domain of such a map is an \emph{admissible G-cover} \cite[Def 2.1]{Jarvis2005}, which is a finite morphism $\pi: E \to C$. Here, $C$ is an $n$-pointed, genus-$\mathsf{g}$, semistable curve with marked points, and the cover must satisfy several conditions:
\begin{itemize}
    \item $E$ is a (potentially disconnected) nodal curve with a $G$-action, and its nodes map to the nodes of $C$.
    \item The map $\pi$ is a principal $G$-bundle over the smooth, unmarked part of $C$. Ramification is permitted over the marked points and nodes.
    \item The action of a stabilizer subgroup $G_q \subseteq G$ at a node $q$ of $E$ must not exchange two branches at $q$, and be \emph{balanced}, meaning the eigenvalues of the group action on the two tangent branches at the node are multiplicative inverses of each other.
\end{itemize}

The map itself is a $G$-equivariant morphism $f: E \to X$. The final stability condition requires that the induced map on the quotient stacks, $\bar{f}: [E/G] \to [X/G]$, is an \emph{orbifold stable map}. 

The moduli space of such objects is denoted by $\overline{\mathcal{M}}_{\mathsf g,n}^{G}(X)$. For a given homology class $\beta \in \rH_2(X/G, \mathbb{Z})$, the substack of maps whose image corresponds to $\beta$ is denoted by $\overline{\mathcal{M}}_{\mathsf g,n}^{G}(X, \beta)$. This is a proper Deligne-Mumford stack, endowed with a natural virtual fundamental class.

\ 

Each marked point $p_i,\,i=1,\dots,n$ of $C$ gives rise to a \textit{canonical} point $(\tilde p_i,g_i)\in \mathcal I_{[E/G]}$ of the inertia stack of $E$. Namely, pick a point $\tilde p_i$ in the preimage of $p_i$. The stabilizer of $\tilde p_i$ is a cyclic group with a \textit{canonical} generator $g_i$, which acts on the tangent space $T_{\tilde p_i}E$ by multiplication by $\exp(2\pi\mathsf i/n_i)$ where $n_i$ is the order of the stabilizer. Therefore, the evaluation maps give a morphism
$$\overline{\mathcal{M}}_{\mathsf g,n}^{G}(X, \beta)\to (\mathcal I_{[X/G]})^{n}\,.$$
Taking the evaluations of the virtual fundamental class of $\overline{\mathcal{M}}_{\mathsf g,n}^{G}(X, \beta)$ at the marked points, we obtain equivariant Gromov-Witten invariants of $[X/G]$.

In order to define quantum product we use the non-degenerate pairing on $\mathrm{H}^\ast_{CR}([X/G])$ coming from Poincar\'{e} duality on the disjoint union $X^g$ of smooth projective varieties (possibly of different dimension), and pairing in \eqref{formulaCR} summands $\rH^\ast(X^g)$ with $\rH^\ast(X^{g^{-1}})$. In the definition of the  Euler vector field one modifies the grading operator by taking into account the so called \textit{age}: the locally constant $\mathbb Q_{\ge 0}$-valued function on $\mathcal I_{[X/G]}$ whose value at $(g,x)\in G\times X,\,gx=g$ is defined as  $\sum_j \theta_j/2\pi$ where $0\le \theta_j<2\pi$ are the principal arguments of the eigenvalues of the action of $g$ on $T_x X$, see e.g. \cite{Jarvis2005}.

In this way, we obtain a version of the A-model maximal F-bundle for global quotient stacks\footnote{This theory generalizes to arbitrary smooth projective Deligne-Mumford stacks which are not necessarily global quotients.}.
Finally, we have a natural action of $\mathsf{Hod}$ preserving the whole structure. The reason is that all virtual fundamental classes of stable orbifold maps are algebraic, and hence Hodge.

\subsection{Untwisted sector}

Define the \textit{untwisted sector} to be the summand $\rH^\ast(X)^G$ of $\rH_{CR}^\ast(X)$ corresponding to $g=1$ in Equation \eqref{formulaCR}.

It follows from the definition that genus $0$ correlators for $[X/G]$ restricted to the untwisted sector $\rH^\ast(X)^G\subset \rH_{CR}^\ast([X/G])$ \textit{coincide} with the correlators for $X$ restricted to $\rH^\ast(X)^G\subset \rH^\ast(X)$. The corresponding moduli spaces of stable maps to $X$ and to $[X/G]$ are the same in this situation. Indeed, one can check that the a priori ramified covering $\pi:E\to C$ in the definition of orbifold stable map, is in our situation automatically non-ramified, i.e. $E=G\times C$.

This implies that the $G$-invariant part $B_X^G$ of the base $B_X$ of the F-bundle for A-model on $X$, is naturally embedded in the base $B_{[X/G]}$ of the F-bundle for the A-model on $[X/G]$:
\begin{equation}
    i_X:B_X^G\hookrightarrow B_{[X/G]}
\end{equation}
Moreover, the inclusion $i_X$ commutes with the action of $\mathsf{Hod}$.

\ 

We will also use the following lemma:
\begin{lemma}
    For any $g_1,\dots,g_n\in G$
    and any elements
    $$\gamma_i\in \rH^\ast(X^{g_i})^{C(g_i)}\subset \rH^\ast_{CR}([X/G]),\quad i=1,\dots, n$$
   and $\beta\in \rH_2(X/G,\Z)$, the correlator 
    $$\langle \gamma_1\dots\gamma_n\rangle_\beta$$
    is zero unless there exist elements
     $g_1',\dots, g_n'$ conjugate to $g_1,\dots, g_n$ such that
      $$g_1'\dots g_n'=1\,.$$
    \label{lmm:product_constraint}
\end{lemma}
\begin{proof} In the case when the stable map is from a smooth curve $C$ of genus $0$ with $n$ marked points $p_1,\dots,p_n$, we have a $G$-covering $E$ of $C-\{p_1,\dots,p_n\}$. This covering gives a homomorphism from the fundamental group of the sphere minus $n$ points to $G$, i.e., $n$ elements in $G$ whose product is $1$, and conjugated to $g_i$. In the more general case of a nodal curve, we get a tree of rational curves. In this case, we use the balancing condition at the nodes.
\end{proof}

\begin{corollary} For any $g\in G$,  the quantum multiplication for $[X/G]$ at any point $b$ of submanifold $B_X^G\subset B_{[X/G]}$ equips the summand $\mathrm{H}^\ast(X^g)^{C(g)}$ with the structure of a module over the algebra $({\mathrm H^\ast(X)}^{G},\star_b)$, equivariant with respect to the action of $\mathsf{Hod}$. \label{cor:module}
    \end{corollary}

\begin{proof}
    Indeed, for the multiplication, we use correlators in the following special case:  
 all observables except 2 belong to the untwisted sector. Then the conjugacy classes associated with the remaining two points are automatically opposite, as follows from the  Lemma \ref{lmm:product_constraint}. The equivariance with respect to $\mathsf{Hod}$-action is obvious.
\end{proof}

\ 

\subsection{Conjectures about blowup formulas for quotient stacks}

Here, we will formulate three conjectures on which all further considerations in this section are based. We expect that these conjectures are not only true but will become essentially trivial if one understands sufficiently well the mechanism of Iritani's proof \cite{iritani2023quantumcohomologyblowups}.

Let $X$ be a smooth projective variety endowed with $G$-action, and $Z\subset X$ be a $G$-invariant submanifold of codimension $c\ge 2$.
 Denote by $\widetilde X= \mathrm{Bl}_Z(X)$  the blowup of $X$ in  $Z$, and  let $X' = X \sqcup Z \sqcup \ldots \sqcup Z$ (with $c-1$ copies of $Z$). There is an obvious $G$-action on $\widetilde X$ and on $X'$. First, we expect an analog of Iritani's theorem for quotients:

 \begin{conjecture}\label{conj:1} There is a canonical (partial) isomorphism $iso_{[\widetilde X/G],[X'/G]}$ of F-bundles restricted to nonempty connected analytic domains in $B_{[\widetilde X/G]}$ and $B_{[X'/G]}$ commuting with $\mathsf{Hod}$-action.
 \end{conjecture}

 Next, we expect that this isomorphism behaves well with respect to the loci corresponding to the untwisted sector:

 \begin{conjecture}\label{conj:2}
     The isomorphism $iso_{[\widetilde X/G],[X'/G]}$ maps the untwisted sector of $B_{[\widetilde X/G]}$ to the  untwisted sector of $B_{[ X'/G]}$. 
 \end{conjecture}

Finally, we want to compare the restricted isomorphism with those coming from the plain A-model, forgetting group action:

\begin{conjecture}\label{conj:3}
    The partial isomorphism from the previous conjecture coincides with the restriction of Iritani's isomorphism $iso_{\widetilde X, X'}$
     to the loci of $G$-fixed points.
\end{conjecture}

A heuristic argument in favor of Conjectures \ref{conj:2} and \ref{conj:3} is that it looks highly plausible that the Iratini's non-linear isomorphism in all contexts should be given by certain sum over trees (see Remark \ref{rem:noncompact}), and  by Lemma \ref{lmm:product_constraint} we should get the same formulas for isomorphisms for the untwisted sector considered as a submanifold of $B_{\widetilde X}$ (resp. $B_{X'}$) or $B_{[\widetilde X/G]}$ (resp. $B_{[X'/G]}$).

One can also formulate analogous conjectures for projectivizations of vector bundles, akin to \cite{iritani2024quantumcohomologyprojectivebundles}.

\ 

\subsection{Atomic decomposition of cohomology of fixed loci}

Let us assume all conjectures from the previous subsection. Recall by Corollary \ref{cor:module}
that for any given element $g\in G$, restricting to a point $b$ in the sublocus $B_X^{G\times\mathsf{Hod}}$ of the big base $B_{[X/G]}$, we obtain an action of the algebra $({\mathrm H^\ast(X)}^{G\times\mathsf{Hod}},\star_b)$ on $\mathrm{H}^\ast(X^g)^{C(g)}$, commuting with $\mathsf{Hod}$. In fact, one can do more. Namely, for a given $g\in G$ consider a \textit{new} group $g^\Z$, the cyclic group generated by $g$, and apply the general theory to the action of $g^\Z$. We obtain the action of the larger algebra $(\mathrm H^\ast(X)^{g^\Z\times \mathsf{Hod}},\star_b)$ on the larger space  $\mathrm{H}(X^g)$ which commutes by naturality with $C(g) \times\mathsf{Hod}$-action. Restricting to the smaller subalgebra $(\mathrm H^\ast(X)^{G\times \mathsf{Hod}},\star_b)$ we get action of it on $\mathrm H^\ast(X^g)$ commuting with $C(g)\times\mathsf{Hod}$-action.

 Therefore, for each $G$-equivariant atom $\boldsymbol\alpha$ of $X$, we obtain as an invariant an isomorphism class of a $C(g)\times \mathsf{Hod}$-module, which we denote by $\mathsf H_{\boldsymbol\alpha,g}$. In particular, we have the corresponding Hodge polynomial which will be denoted by $P_{\boldsymbol{\alpha},[g]}\in\Z_{\ge 0}[t,t^{-1}]$, here $[g]$ is the conjugacy class of $g$.

 We have the following isomorphism of $C(g)\times \mathsf{Hod}$-modules:
 $$\rH^\ast(X^g)\simeq \oplus_{\boldsymbol\alpha}  \mathsf H_{\boldsymbol\alpha,g}\,$$

  In the case $g=1$, we get the usual atoms and the original invariant $P_{\boldsymbol{\alpha}}$. Non-trivial conjugacy classes give new invariants of atoms, which we will explore in the next subsections.
  
 Notice that it is quite possible that for two elements $g_1,g_2\in G$ generate the same cyclic subgroup, i.e., $g_2=g_1^k$ for some $k\ge 1$ coprime to the order of $g_1$, the atomic decompositions of the same space $\rH^\ast(X^{g_1})=\rH^\ast(X^{g_2})$ are different, when one interprets it as the fixed locus of $g_1$ or of $g_2$. The reason is that the moduli spaces of orbifold stable maps could be different for $g_1$ and $g_2$.

\ 

 \subsection{First applications}

 \begin{lemma}\label{lem:coefficient}
    For any $d\ge 1$, any conjugacy class $[g]$ of $G$, and any atoms $\boldsymbol\alpha$ coming from varieties of dimension $\le d$, we have 
    $$\mathrm{Coeff}_{t^d}P_{\boldsymbol{\alpha},[g]}\le \mathrm{Coeff}_{t^d}P_{\boldsymbol{\alpha}}.$$
\end{lemma}
\begin{proof}
    Let $X$ be $d'$-dimensional variety with $G$-action where $d'\le d$. If $d'<d$ then obviously for all atoms $\boldsymbol{\alpha}$ of $X$ the corresponding Hodge polynomials $P_{\boldsymbol{\alpha},[g]},P_{\boldsymbol{\alpha}}$ do not contain the mononial $t^d$. 
     If $X$ is $d$-dimensional and the action of $g^\Z$ is non-trivial, then 
      the dimension of fixed locus $X^g$ is strictly smaller than $d$, and hence  $\mathrm{Coeff}_{t^d}P_{\boldsymbol{\alpha},[g]}=0$.
       Finally, if $\dim X=d$ and the action of $g^
       \Z$ is trivial, then the atomic decomposition of $\mathrm{H}^{\ast}(X)$ coincides with the atomic decomposition of $\mathrm{H}^{\ast}(X^g)$, and for each atom we have an equality. \end{proof}

    In what follows, for any smooth projective variety $X$ we denote by $P_X\in \mathbb Z_{\ge 0}[t,t^{-1}]$ its Hochschild graded Hodge
polynomial:
    $$ P_X:=\sum_{p,q\ge 0}h^{p,q}(X)t^{p-q}.$$

\begin{corollary}\label{cor:rough}
    For a $d$-dimensional smooth projective variety $X$ endowed with $G$-action, if 
$$\mathrm{Coeff}_{t^{d-2}}P_{X^g}> \mathrm{Coeff}_{t^{d-2}}P_{X}$$ 
 then $X$ is not $G$-equivariantly birational to a projective space with a $G$-linearizable action.
\end{corollary}
\begin{proof}
    Lemma \ref{lem:coefficient} implies that under the hypothesis, not all atoms of $X$ can come from dimension $\le d-2$.
\end{proof}

\ 

We get many possible examples, which we detail below. Interestingly enough, \emph{all} the formerly presented examples in this paper can be fully captured by the present technique. 

\ 

\begin{example}\label{ex:family-higher}
Very much in the nature of Example \ref{ex:fixed-higher-genus-curve},
for any $d\ge 3$, let $X$ be the $d$-dimensional affine variety
    $$x_1 x_2 x_3=P(x_4,\dots,x_{d+1})$$
   where $P(x_3,\dots, ,x_{d+1})$ is a  polynomial of degree $3k$ where $2k\ge d-1$, defining a smooth hypersurface in the projective space $\mathbb P^{d-2}$. The group $G:=\Z/3$ acts by cyclic permutations of $x_1,x_2,x_3$.

 The closure $\overline X$ of the image of variety $X$ in $\mathbb P^1\times \mathbb P^1\times \mathbb P^1\times \mathbb P^{d-2}$ is singular, the singularity locus consists of 3 copies of the hypersurface in $\mathbb P^{d-2}$ given by equation $P=0$. Resolving the singularities torically, we get a smooth variety $\widetilde X$. It is easy to see that
 $$\mathrm{Coeff}_{t^{d-2}}P_{\widetilde X^g}>0= \mathrm{Coeff}_{t^{d-2}}P_{\widetilde X}$$
 where $g\in G$ is a generator.
 Indeed, the transcendental part of the cohomology of $\widetilde X$ comes from the cohomology of the $(d-3)$-dimensional hypersurface given by $P=0$, and it can not contribute to  $\mathrm{Coeff}_{t^{d-2}}P_{\widetilde X}$. Similarly, the fixed locus is $(d-2)$-dimensional, and given by the equation
 $$x^3=P(x_4,\dots,x_{d+1}).$$
 Then the  condition $2k\ge d-1$ guarantees that
$${\mathrm{d}x_2\dots \mathrm{d}x_{d+1}\over x^2}$$
is a non-vanishing section of the canonical bundle of $\widetilde X^g$, and therefore $\mathrm{Coeff}_{t^{d-2}}P_{\widetilde X^g}>0$.
\end{example}

\ 

\begin{example}
Let $S_1,~S_2$ be two rational surfaces with involutions as in Example \ref{ex:product=with-line}, and assume for simplicity that the fixed loci for each of the surfaces is just one curve $C_i$ of a positive genus $\mathsf{g}_i,~i=1,2$. The Hodge polynomial of $X$ is constant and equal to $\chi(S_1)\chi(S_2)$. Let $X:=S_1\times S_2$ with the diagonal action. The fixed locus is $C_1\times C_2$.
Using the K\"unneth formula, we see that $\mathrm{Coeff}_{t^2}P_{X^{\mathbb{Z}/2}}=\mathsf{g}_1\mathsf{g}_2>0=\mathrm{Coeff}_{t^2}P_X$.
\end{example}

\ 
\begin{example}\label{ex:complete-intersection-odd}
  Let $X$ be a three-dimensional smooth complete intersection of two quadrics in $\mathbb{P}^5$ as discussed in Example \ref{ex:complete-intersection}. Assume it is invariant under the $\mathbb Z/2$-action on $\mathbb P^5$ induced from the linear representation $\underbrace{\C^2}_+\oplus \underbrace{\C^4}_{-}$, i.e., $\mathbb P^5=\mathbb P(1^{\oplus 2}\oplus (-1)^{\oplus 4})$ for $\mathbb Z/2=\langle 1,-1\rangle$. We claim that $X$ with this $\mathbb Z/2$-action is not $\mathbb Z/2$-birationally equivalent to $\mathbb P^3$ with any $\mathbb Z/2$-linearizable action.

   On the one hand, as we have seen in Example \ref{ex:complete-intersection}, $X$ has three atoms which are automatically $G$-equivariant atoms because the splitting comes from the small quantum multiplication which is $G$-invariant. These atoms are parameterized by eigenvalues of $\boldsymbol{\kappa}_b$ which are zero, with multiplicity six, and (up to scale) $\pm 1$ (roots of $1$ of order $r_X=2$) with multiplicity $1$. For the atom corresponding to $\lambda=0$ we have $\mathrm{Coeff}_{t}P_{\lambda=0}=1$. Thus, we can not use Corollary \ref{cor:rough} to obtain the claim. However, it is straightforward to check that the fixed locus for this action is an elliptic curve, which per Example \ref{ex:nef} has only one atom. Since its Hodge polynomial is different from $P_{\lambda=0}$, one concludes the desired statement.

   It is interesting to understand this example from the categorical point of view, see Section \ref{sec:ideal-legal}. The derived category $\mathrm{Perf}(X)$ endowed with $\Z/2$-action admits an equivariant semi-orthogonal decomposition
$$\mathrm{Perf}(X)=\langle\mathcal O(0),\mathcal O(1),\mathrm{Perf}(C)\rangle$$
where $C$ is a genus $2$ curve. The third term \textit{does not} come from a curve with geometric involution. The action of $\Z/2$ on $\mathrm{Perf}(C)$ is by the tensor product with a non-trivial line bundle on $C$, which gives a 2-torsion point in $\mathrm{Jac}(C)$.
\end{example}

As we have seen, one can use the technique sketched above to deal with questions on $\mathbb Z/2$-birational equivalence for the three-dimensional complete intersection of two quadrics. We do not immediately see how to apply it to analogous questions for the four-dimensional complete intersection of two quadrics in $\mathbb P^6$. It could also be the case that all the $\mathbb Z/2$-actions on them are linearizable.

\subsection{Generalization: diagram of manifolds with F-bundles} 

One can generalize the two inclusions of maximal F-bundles used earlier in this section, which on the level of tangent spaces corresponds to inclusions
$$\rH^*(X)^G\hookrightarrow \rH^*(X),\quad \rH^*(X)^G\hookrightarrow \rH^*_{CR}(X)\,.$$
Namely, for any pair of subgroups $H_1,H_2<G$ such that $H_1$ is a normal subgroup of $H_2$  (one can write $H_1\triangleleft H_2<G$) we have a maximal $F$-bundle with the base
  $$(B_{[X/H_1]})^{H_2}\subset B_{[X/H_1]}$$
  which is the fixed locus of the natural $H_2$-action on $B_{[X/H_1]}$. Whenever
  $$H_1<H_1'<H_2'<H_2<G,\quad H_1\triangleleft H_2,H_1'\triangleleft H_2'$$
  we have an \textit{inclusion} of bases of maximal F-bundles
  $$(B_{[X/H_1]})^{H_2}\hookrightarrow (B_{[X/H_1']})^{H_2'}\,.$$
  Therefore, we have a partially ordered set of pairs $(H_1\triangleleft H_2)$ as above realized by inclusions of bases of maximal F-bundles, the whole picture is equivariant with respect to $G\times\mathsf{Hod}$. There is also the inclusion of an even smaller F-bundle $(B_X)^{G\times \mathsf{Hod}}$
  in the bases of all these F-bundles.

We can formulate the following generalization of Conjectures \ref{conj:2} and \ref{conj:3}:
  \begin{conjecture}
      Inclusions $(B_{[X/H_1]})^{H_2}\hookrightarrow (B_{[X/H_1']})^{H_2'}$ commute with Iritani's partial isomorphisms, and with action of $\mathsf{Hod}$.
  \end{conjecture}

  Assuming the above conjecture, we get for each $G$-equivariant atom $\boldsymbol\alpha$ of $X$ a $G$-equivariant diagram of germs of bases of maximal F-bundles labeled by the partial ordered set of pairs  $(H_1\triangleleft H_2)$ as above.

This is a very sensitive invariant, from which one can extract easier-to-calculate invariants, as $ G$-equivariant diagrams of $\mathsf{Hod}$-modules. It looks plausible that these invariants will detect, e.g., non-$G$-linearizability of actions of a sufficiently large finite group $G\subset \mathrm O(6)$ on the standard 4-dimensional quadric in $\mathbb P^5$.

 \

\section{Some new $\mathbb Z$-valued geometric invariants for smooth projective varieties in dimensions two and three}
\label{sec:geometric}

In the theory of symbols, one treats connected components $F_\alpha$ of the locus of fixed points just as birational types.
 It is natural to try to exploit further data, like actual isomorphism classes of $F_\alpha$, together with the various weight components of the normal bundle. As a proof of concept, we show that such finer invariants do indeed exist, for $\Z/2$-actions in dimension two and three, and for $\Z/3$-action in dimension three. Surprisingly, they take constant values in all  cases with \textit{rational} $X$ we can think of:
 \begin{itemize}
     \item invariant $I$ for $\Z/2$-actions in $\dim=2$ (see Section \ref{sec:dimension-2}) takes value $4$,
     \item invariant $J$ for $\Z/3$-actions in $\dim=3$ (see Section \ref{sec:dimension-3}) takes value $6$,
     \item invariant $K$ for $\Z/2$-actions in $\dim=3$ (see Section \ref{sec:dimension-3-2}) takes value $8$.
 \end{itemize}

\ 

As a warm-up, we start with the two-dimensional case:

\subsection{A new $\mathbb Z/2$-birational invariant in dimension two}
\label{sec:dimension-2}

\begin{theorem}\label{thm:degrees}
Let \( X \) be a smooth projective surface with a regular generically free \( \mathbb{Z}/2 \)-action. Then the following quantity is a \( \mathbb{Z}/2 \)-birational invariant:
\[
I:=\chi(X^{\mathbb{Z}/2}) + \sum_{\alpha:\dim F_\alpha=1} \deg\big( \mathcal{N}_X F_\alpha \big),
\]
where the sum runs over all 1-dimensional irreducible components \( F_\alpha \subset X^{\mathbb{Z}/2} \) of dimension $1$, and \( \mathcal{N}_X F_\alpha \) denotes the normal bundle of curve \( F_\alpha \) in \( X \).
\end{theorem}

\begin{proof}[Proof of Theorem \ref{thm:degrees}]
   
We prove that \( I \) is preserved under all \( \mathbb{Z}/2 \)-equivariant blowups by considering all possible cases for the center of the blowup.
\medskip

\noindent\textbf{Case 1: Blowing up a free orbit \( Z \subset X \setminus X^{\mathbb{Z}/2} \) consisting of $2$ points.} \\
Obviously, $I$ does not change.

\medskip
\noindent\textbf{Case 2: Blowing up in an isolated fixed point.} The Euler characteristic increases by $1$, and the degree of the normal bundle of the exceptional divisor is $-1$, hence $I$ does not change.

\medskip
\noindent\textbf{Case 3: Blowing up in a point lying on a component \(F_\alpha\subset X^{\mathbb{Z}/p} \)  which is a curve.} The Euler characteristic increases by $1$, and the degree of the normal bundle of the curve \(F_\alpha\) decreases by $1$, hence $I$ does not change.
\end{proof}

\subsection{A new  $\mathbb Z/3$-birational invariant in dimension three}
\label{sec:dimension-3}

\begin{theorem}\label{thm:pure-geometry}
Let $X$ be a smooth projective threefold with a generically free $\mathbb Z/3$‐action.  Then the following quantity is invariant under \(\mathbb{Z}/3\)-equivariant blowups:
\begin{equation}
  J = 
    \sum_{[\lp\lp\lm],\; [\lp\lm\lm]} 1
    + 
    \sum_{(C,[\lp\lp]),\; (C,[\lm\lm])} \left(1 - \mathsf{g} + d\right)
    +
    \sum_{(C,[\lp\lm])} \left(2 - 2\mathsf{g} + d\right)
    +
    \sum_{S} \left(3 - 3\mathsf{g} - K_S\cdot \mathcal N\right)
\end{equation}
where the summation runs over fixed point set components $X^{\mathbb Z/3}$ where:
\begin{itemize}
  \item 
   $C$ is a curve with genus \(\mathsf{g}=\mathsf{g}(C) \geq 0\).
  \item 
    \(d \in \mathbb{Z}\) is the degree of \(\wedge^{2}(\mathcal{N}_XC)\), where \(\mathcal{N}_XC\) is the normal bundle of $C$ in $X$.
  \item $S$ is birational to a ruled surface whose base of the ruling is a curve with genus 
    \(\mathsf{g} \geq 0\) (slightly abusing the notation, we write $\mathsf g=\mathsf g(S)$).
  \item 
    \(K_{S}\) is the canonical bundle of the surface \(S\).
  \item 
    \(\mathcal{N}=\mathcal N_X S\) is the normal bundle of \(S\) in $X$.
\end{itemize}
The signs decorations are the non-zero characters for \(\mathbb Z/3 = (\mathbb Z/3)^{\vee}\), i.e.,
\[
  \lp,\,\lm \;=\; 1,\,2 \pmod{3}.
\]
so that
\begin{itemize}
    \item $[\lp\lp\lm], [\lp\lm\lm]$ are the normal weights for the $\mathbb Z/3$-normal bundle representation at fixed points
    \item for symbols like $(C,[\lp\lp]), (C,[\lp\lm]), (C,[\lm\lm])$ the terms $[\lp\lp], [\lm\lm], [\lp\lm]$ stand to the normal weights for the $\mathbb Z/3$-normal bundle representation at the fixed curve $C$.
\end{itemize}
 \end{theorem}

\begin{proof}[Proof of Theorem \ref{thm:pure-geometry}]
To establish the invariance of the quantity $J$, we must show that for any smooth, invariant, closed center of blowup $Z \subset X$ (of codimension $r \geq 2$), the change $\Delta J:= J(\mathrm{Bl}_Z X) - J(X)$ is zero. We proceed by a case-by-case analysis of all possible centers $Z$. The locus of fixed points on the blowup maps to the locus of fixed points on the original variety $X$. We separate contributions by connected components of $X^{\Z/3}$ of dimensions $0,1,2$, and the local behavior of the center $Z$ of the blowup near these components. 

\noindent \textbf{Case 1: Blowup in an isolated fixed point.}

\begin{itemize}
    \item \textbf{Type $[\lp\lp\lp]$} (\textit{and the opposite case  $[\lm\lm\lm]$}): A fixed point of this type contributes $0$ to $J$. The blowup replaces the point with an exceptional divisor $E \cong \mathbb{P}^2$, which is a new surface component in the $\mathbb Z/3$-fixed locus on the blowup. For this new surface $S=E$, we have genus $\mathsf{g}(S)=0$, canonical bundle $K_S = \mathcal{O}_{\mathbb{P}^2}(-3)$, and normal bundle $\mathcal{N}=\mathcal{O}_{\mathbb{P}^2}(-1)$. The contribution of this new surface to $J$ is:
    \begin{align}
        3 - 3\mathsf{g}(S) - K_S \cdot \mathcal{N} = 3 - 3(0) - (-3)(-1) = 3 - 3 = 0.
    \end{align}
    Thus, $\Delta J = \underbrace{0}_{\text{new surface}} - \underbrace{0}_{\text{old point}} = 0$.

    \item \textbf{Type $[\lp\lp\lm]$} (\textit{and the opposite case}): A fixed point of type $[\lp\lp\lm]$ contributes $1$ to $J$. The blowup replaces the point with an exceptional divisor $E$. The $\mathbb Z/3$-fixed locus on $E$ is a curve $C \cong \mathbb{P}^1$ with normal weights $[\lp\lp]$. For this curve, $\mathsf{g}(C) = 0$ and the degree of the second exterior power of its normal bundle is $d = \deg(\wedge^2 \mathcal{N}_{X'}C) = 0$. Its contribution to $J$ is:
    \begin{align}
        1 - \mathsf{g}(C) + d = 1 - 0 + 0 = 1.
    \end{align}
    Thus, $\Delta J = \underbrace{1}_{\text{new curve}} - \underbrace{1}_{\text{old point}} = 0$.
\end{itemize}

\noindent \textbf{Case 2: Blowup in a curve passing through an isolated fixed point.}

\begin{itemize}
    \item \textbf{Types $[\lp\lp\lp]$} (\textit{and the opposite case}):  The initial fixed point contributes $0$. Blowing up in an invariant curve passing through it replaces the point with a fixed curve $C \cong \mathbb{P}^1$ with normal weights $[\lp\lp]$. The degree of the second exterior power of its normal bundle is $d=-1$. The new contribution is:
    \begin{align}
        1 - \mathsf{g}(C) + d = 1 - 0 + (-1) = 0.
    \end{align}
    Hence, $\Delta J = \underbrace{0}_{\text{new curve}} - \underbrace{0}_{\text{old point}} = 0$.

    \item \textbf{Type $[\lp\lp\lm]$ and blowup in the $\lm$ direction}(\textit{and the opposite case}): The initial point contributes $1$. The blowup replaces it with a new fixed curve $C \cong \mathbb{P}^1$ with normal weights $[\lp\lm]$ and total degree $d=d_\lp+d_\lm = 0+(-1)=-1$. The new contribution is:
    \begin{align}
        2 - 2\mathsf{g}(C) + d = 2 - 2(0) + (-1) = 1.
    \end{align}
    Hence, $\Delta J =\underbrace{1} _{\text{new curve}} - \underbrace{1}_{\text{old point}} = 0$.

    \item \textbf{Type $[\lp\lp\lm]$ and blow up in the $\lp$ direction}(\textit{and the opposite case}):
     The initial point contributes $1$. The blowup replaces it by two isolated fixed points of types $[\lp\lp\lp]$ and $[\lp\lm\lm]$. The new contribution is $1$, hence $\Delta J=\underbrace{1}_{\text{new point of type }[\lp\lm\lm]}-\underbrace{1}_{\text{old point}}=0$.
\end{itemize}

\noindent \textbf{Case 3: Blowup in a point on a fixed curve.}

\begin{itemize}
    \item \textbf{The curve of type $[\lp\lp]$}(\textit{and the opposite case}):  Let the original curve be $C$ with contribution $2 - 2\mathsf{g} + d$. Blowing up at a point on $C$ produces in the fixed locus a new component $\mathbb P^1$ with normal weights [\lp\lm] and the degree of the normal bundle equal to $-1+1=0$. The curve $C$ is replaced by its proper transform $C'$, for which the degree of the second exterior power of the normal bundle becomes $d' = d-2$. The total change is:
    \begin{align}
        \Delta J &= \underbrace{2-2\cdot 0+0}_{\text{new curve }\mathbb P^1} + \underbrace{(2 - 2\mathsf{g} + (d-2))}_{\text{new curve}} - \underbrace{(2 - 2\mathsf{g} + d)}_{\text{old curve}} \\
        &= 2 + (2 - 2\mathsf{g} + (d - 2)) - (2 - 2\mathsf{g} + d) = 0.
    \end{align}
    \item \textbf{The curve of type $[\lp\lm]$:} Let the original curve be $C$ with contribution $2 - 2\mathsf{g} + d$. Blowing up a fixed point on $C$ replaces that point with two new isolated fixed points of types $[\lp\lp\lm]$ and $[\lp\lm\lm]$, each contributing $1$. The curve $C$ is replaced by its proper transform $C'$, for which the degree of the second exterior power of the normal bundle becomes $d' = d-2$. The total change is:
    \begin{align}
        \Delta J &= \underbrace{(1 + 1)}_{\text{new points}} + \underbrace{(2 - 2\mathsf{g} + (d-2))}_{\text{new curve}} - \underbrace{(2 - 2\mathsf{g} + d)}_{\text{old curve}} \\
        &= 2 + (2 - 2\mathsf{g} + (d - 2))
        - (2 - 2\mathsf{g} + d) = 0.
    \end{align}
\end{itemize}

\noindent \textbf{Case 4: Blowup in a fixed curve.}
\begin{itemize}
    \item \textbf{The curve of type $[\lp\lp]$} (\textit{and the opposite case}): The curve $C$ contributes $1 - \mathsf{g} + d$. It is replaced by a ruled surface $S = \mathbb{P}(\mathcal{N}_XC)$ fibered over $C$, with $\mathsf{g}(S) = \mathsf{g}(C) = \mathsf{g}$. The new contribution is $3 - 3\mathsf{g} - K_S \cdot \mathcal{N}$ where $\mathcal N=\mathcal N_S \widetilde X$ is the normal line bundle to $S$ in the blowup $\widetilde X=\mathrm{Bl}_C X$. In order to simplify the calculation, we assume that the rank two bundle $\mathcal N_X C$ splits into the sum $\mathcal L_1\oplus \mathcal L_2$ of line bundles of certain degrees $d_1,d_2$ such that $d=d_1+d_2$. The general case of a non-split bundle follows by continuity. To calculate $K_S \cdot \mathcal{N}$, we use the basis of $\mathsf{NS}(S)$ given by the curve $C_1 = \mathbb{P}(\mathcal{L}_1)\subset S$ and the fiber $C_{\text{vert}} \cong \mathbb{P}^1$. The intersection numbers with $K_S$ and $\mathcal{N}$ are known:
    \begin{align}
        K_S \cdot C_1 &= 2\mathsf{g}-2 + d_1-d_2,  & K_S \cdot C_{\text{vert}} &= -2 \\
        \mathcal{N} \cdot C_1 &= d_1, & \mathcal{N} \cdot C_{\text{vert}} &= -1
    \end{align}
    The intersection form on $\mathsf{NS}(S)$ in the basis $(C_1,C_{\mathrm{vert}})$ is given by
    $$\begin{pmatrix}d_2-d_1 & 1\\1& 0\end{pmatrix}$$
    Therefore, the intersection $K_S\cdot \mathcal N$ is given by
    $$ \begin{pmatrix}d_1  &-1\end{pmatrix} \cdot \begin{pmatrix}d_2-d_1 & 1\\1& 0\end{pmatrix}^{-1}\cdot  \begin{pmatrix}2\mathsf{g}-2+d_1-d_2\\-2\end{pmatrix} =2-2\mathsf{g}-d_1-d_2.  $$
    The change in $J$ is therefore:
    \begin{align}
       \Delta J &=\underbrace{ (3 - 3\mathsf{g} - (2 - 2\mathsf{g} - d))}_{\text{new surface}} - \underbrace{(1 - \mathsf{g} + d)}_{\text{old curve}} \\
        &= (1 - \mathsf{g} + d) - (1 - \mathsf{g} + d) = 0.
    \end{align}

    \item \textbf{The curve of type $[\lp\lm]$:} The curve $C$ contributes $2 - 2\mathsf{g} + d$. It is replaced by two curves, $C_1$ (type $[\lp\lp]$) and $C_2$ (type $[\lm\lm]$). Let denote the degrees of $\lp$ and $\lm$ components of $\mathcal N_XC$ by $d_\lp$ and $d_\lm$ respectively, so $d = d_\lp + d_\lm$. The total change is:
    \begin{align}
        \Delta J &= \underbrace{(1 - \mathsf{g} + d_\lp)}_{\text{new curve } C_1} + \underbrace{(1 - \mathsf{g} + d_\lm)}_{\text{new curve } C_2} - \underbrace{(2 - 2\mathsf{g} + d)}_{\text{old curve } C} \\
        &= (2 - 2\mathsf{g} + d_\lp + d_\lm) - (2 - 2\mathsf{g} + d_\lp + d_\lm) = 0.
    \end{align}
\end{itemize}

\noindent \textbf{Case 5: Blowup in a curve transversal to a fixed curve.}
\begin{itemize}
    \item \textbf{The fixed curve of type $[\lp\lp]$} (\textit{and the opposite case}):  The original curve $C$ contributes $1-\mathsf{g}+d$. For each intersection point the blowup adds a new isolated fixed point of type $[\lp\lp\lm]$. The proper transform of $C$, denoted $C'$, is still a component of the fixed locus. The degree $d'$ of the normal bundle of $C'$ is equal to $d'=d-k$, where $k$ is the number of intersection points of $C$ and the center of the blowup. The total change is:
 \begin{align}
    \Delta J &= k (\text{contribution of new points}) \\
             &\quad + (1 - \mathsf g + d - k) (\text{contribution of } C') \\
             &\quad - (1 - \mathsf g + d) (\text{contribution of } C) = 0
\end{align}
\end{itemize}

\noindent \textbf{Case 6: Blowup in a point on a fixed surface.}

\begin{itemize}
    \item \textbf{The fixed surface of type $\lp$} (\textit{and the opposite case}):  The surface $S$ contributes $3 - 3\mathsf{g} - K_S \cdot \mathcal{N}$. Blowing up a point on $S$ creates a new isolated fixed point of type $[\lp\lm\lm]$ and replaces $S$ with its blowup $\widetilde{S}$. For $\widetilde{S}$, we have $K_{\widetilde{S}} = \pi^{\ast}K_S + E$ and $\widetilde{\mathcal{N}} = \pi^{\ast}\mathcal{N} - E$, where $E$ is the exceptional divisor. The new intersection product is $K_{\widetilde{S}} \cdot \widetilde{\mathcal{N}} = (\pi^{\ast}K_S + E) \cdot (\pi^{\ast}\mathcal{N} - E) = K_S \cdot \mathcal{N} - E^2 = K_S \cdot \mathcal{N} + 1$. The change in $J$ is:
    \begin{align}
        \Delta J &= \underbrace{1}_{\text{new point}} + \underbrace{(3 - 3\mathsf{g} - (K_S\cdot\mathcal{N} + 1))}_{\text{new surface}} - \underbrace{(3 - 3\mathsf{g} - K_S\cdot\mathcal{N})}_{\text{old surface}} \\
        &= 1 - 1 = 0.
    \end{align}
\end{itemize}

\noindent \textbf{Case 7: Blowup in a transverse curve to a fixed surface.}
\begin{itemize}
    \item  \textbf{The fixed surface of type $\lp$} (\textit{and the opposite case}): Let $S$ be a fixed surface and $C$ be a $\Z/3$-invariant curve that intersects $S$ transversely at a finite set of points.  The blowup of $X$ in $C$ replaces component $S$ of the fixed locus with its proper transform $\widetilde{S}$, which is isomorphic to the blowup of $S$ at the points $C \cap S$. The canonical and normal bundles transform as $K_{\widetilde{S}} = \pi^{\ast}K_S + \sum E_i$ and $\widetilde{\mathcal{N}} = \pi^{\ast}\mathcal{N}$, where $E_i$ are the exceptional divisors on $\widetilde{S}$. Since $\pi^{\ast}\mathcal{N} \cdot E_i = 0$, the intersection product $K_{\widetilde{S}}\cdot\widetilde{\mathcal{N}} = K_S\cdot\mathcal{N}$ is unchanged. 
    
    The change in $J$ is therefore $$\Delta J = \underbrace{(3-3\mathsf g +K_{\widetilde S}\cdot \widetilde{\mathcal N} )}_{\text{new surface}} -   (\underbrace{3-3\mathsf g +K_{S}\cdot \mathcal N)}_{\text{old surface}}= 0.$$
\end{itemize}

\noindent \textbf{Case 8: Blowup in a curve on a fixed surface.}

\begin{itemize}
    \item  \textbf{The fixed surface of type $\lp$} (\textit{and the opposite case}): Let $C \subset S$ be a closed curve. In the blowup $\mathrm{Bl}_CX$, the component $S$ of the fixed locus is replaced by its proper transform $\widetilde{S}\cong S$, and a new curve $\widetilde C\cong C$, with normal weights $[\lp\lm]$. The change in the surface term's contribution is $K_S \cdot C$.
    \begin{align}
        \Delta(3-3\mathsf{g}-K_S\cdot\mathcal{N}) &= -(K_{\widetilde{S}}\cdot\widetilde{\mathcal{N}}) - (-K_S\cdot\mathcal{N})\\
        &= -K_S\cdot(\mathcal{N}-C) + K_S\cdot\mathcal{N}\\
        &= K_S \cdot C= 2\mathsf{g}(C) - 2 - \mathcal N_S C\cdot C.
    \end{align}
    The last equality is the adjunction formula. Now turn to the new curve component  $\widetilde C$ of the fixed locus. Its normal bundle canonically splits into $\lp$ and $\lm$ parts. The degrees these line bundles are $d_\lp = \mathcal{N}_XS \cdot C$ and $d_\lm = \mathcal N_S C\cdot C - \mathcal{N}_XS \cdot C$ respectively. The contribution of this new curve to $J$ is therefore:
    \begin{align}
        J_{\text{new curve}} &= 2 - 2\mathsf{g}(C') + d_\lp + d_\lm \\
        &= 2 - 2\mathsf{g}(C) + (\mathcal{N}_XS \cdot C) + (\mathcal N_S C\cdot C - \mathcal{N}_XS \cdot C) \\
        &= 2 - 2\mathsf{g}(C) + \mathcal N_S C\cdot C.
    \end{align}
    The total change in $J$ is:
    \begin{align}
        \Delta J &= J_{\widetilde S}-J_S+J_{\text{new curve}}\\
        &= (2\mathsf{g}(C) - 2 -  \mathcal{N}_XS \cdot C) + (2 - 2\mathsf{g}(C) +  \mathcal{N}_XS \cdot C) =0.
    \end{align}
\end{itemize}

Since $\Delta J = 0$ in all cases, the quantity $J$ is invariant under $\mathbb{Z}/3$-equivariant blowups.
\end{proof}

\begin{remark} Notice that the contribution of any curve component $C$ of the fixed locus with normal weights $\lp\lm$ can be succinctly written as
\[-K_X\cdot C\,.\]\label{rem:+-contribution}
\end{remark}

Next, we revisit Example \ref{ex:fixed-higher-genus-curve}, and use the notation introduced there. We will calculate our ``elementary'' invariant $J$ in this case. It turns out that it is equal to $6$ independently of the degree $3k$ of the polynomial $P(x_4)$. The same value $6$ happens for all projective spaces with $\mathbb Z/3$-lineariazable actions, hence it fails to show that $X$ is not equivariantly birational to $\mathbb P^3$ for any generically free $\mathbb Z/3$-linearizable action. Thus, it is much weaker than the combined invariant from Theorem \ref{thm:invariant}.
    \begin{example}[Example \ref{ex:fixed-higher-genus-curve} revisited]\label{ex:fixed-higher-genus-curve-revisited}
Recall that the variety $X$ is a resolution of singularities of a singular hypersurface $Y$ in $(\mathbb P^1)^4$ of degree $(1,1,1,3k)$ given by the following equation in the affine part $\mathbb A^4$:
\[
  x_{1}\,x_{2}\,x_{3} \;=\; P(x_{4}), 
  \qquad
  P(x_{4}) \;=\; x_{4}^{3k} + \cdots,
\]
where $P$ is monic of degree $3k$ with distinct roots.   The fixed locus of $\Z/3$-action  is  the compactification of the affine curve
\(
  x^{3} = P(x_{4}),
\)
and it is a smooth compact curve $C$ with the normal weights $\lp\lm$ and disjoint from the singularities that we have to resolve. The only contribution to $J$ is by Remark \ref{rem:+-contribution} equal to 
\[2-2\mathsf g+d=-K_X\cdot C\,.\]

By the adjunction formula
\begin{align}-K_X\cdot C&=-K_{(\mathbb P^1)^4}\cdot C-Y\cdot C\\&= \mathcal O(2,2,2,2)\cdot C-\mathcal O(1,1,1,3k)\cdot C\\
&=\mathcal O(1,1,1,2-3k)\cdot C\end{align}
The degree of curve $C$ in the ambient space $(\mathbb P^1)^4$ is
   \[(-3k,-3k,-3k,3)\]
   which implies that
   \[J=-K_X\cdot C =-3k-3k-3k+(2-3k)3=6.\]
   One can easily calculate invariant $J$ for $\mathbb P^3$ endowed with any non-trivial linearizable $\Z/3$-action, and get the same value $6$.
\end{example}

\subsection{A new  $\mathbb Z/2$-birational invariant in dimension three}
\label{sec:dimension-3-2}

As our last contribution, we state Theorem \ref{thm:pure-geometry-2d} below, whose proof follows analogously to that of Theorem \ref{thm:pure-geometry}. 

\begin{theorem}\label{thm:pure-geometry-2d}
    Let $X$ be a smooth projective threefold with a generically free $\mathbb Z/2$‐action.  Then the following quantity is invariant under \(\mathbb{Z}/2\)-equivariant blowups:
\begin{equation}
  K = 
    \#\text{fixed points}
    + 
    \sum_{C} \left(2 - 2\mathsf{g} + d\right)
    +
    \sum_{S} \left(4 - 4\mathsf{g} - K_S\cdot \mathcal N\right)
\end{equation}
where the summation runs over fixed point set components $X^{\mathbb Z/2}$. Here:
\begin{itemize}
  \item 
   $C$ is a curve with genus \(\mathsf{g}=\mathsf{g}(C) \geq 0\).
  \item 
    \(d \in \mathbb{Z}\) is the degree of \(\wedge^{2}(\mathcal{N}_XC)\), where \(\mathcal{N}_XC\) is the normal bundle of $C$ in $X$.
  \item $S$ is birational to a ruled surface whose base of the ruling is a curve with genus 
    \(\mathsf{g} \geq 0\) 
  \item 
    \(K_{S}\) is the canonical bundle of the surface \(S\).
  \item 
    \(\mathcal{N}\) is the normal bundle of \(S\) in $X$.
\end{itemize}
\end{theorem}

\ 

It remains to be checked if these geometric invariants can lead to finer invariants when combined with $G$-equivariant atoms enhanced with Chen-Ruan cohomology.

\ 

\ 

\appendix

\section{$G$-equivariant birational geometry versus birational geometry over non-algebraically closed fields}
\label{ap:specialization}

The geometry of finite group actions is closely related to non-algebraically closed fields of positive transcendence degrees. If a finite group $G$ acts birationally and generically free on a variety $X/\C$, then $\C(X)$ is a Galois extension of $\C(X/G)$ with Galois group $G$. In other words, generically free birational $G$-actions on complex varieties are the same as the epimorphisms of the Galois groups of functional fields to $G$.

Let $X_1,X_2,Y$ be complex varieties endowed with generically free $G$-action, and let's assume $\dim X_1=\dim X_2=d\ge 1$. 
Then we have two varieties $\mathcal X_1, \mathcal X_2$ of dimension $d$ over the non-algberaically closed field $\mathcal K=\C(Y)^G=\C(Y/G)$ corresponding to the fibrations 
$$(X_i\times Y)/G\to Y/G.$$
A birational equivalence $\mathcal X_1\dashrightarrow\mathcal X_2$ over $\mathcal K$ is the same as $G$-equivariant birational equivalence between $X_1\times Y$ and $X_2 \times Y$ commuting with the projection to $Y$. Therefore, we have an obvious implication:

\ 

\begin{center}
\noindent\fbox{%
\parbox{0.67\textwidth}{%
$X_1$ is $G$-birational to $X_2$ $\implies$ $\mathcal X_1$ is birational to $\mathcal X_2$}%
}
\end{center}

\ 

We claim that in some cases the \textit{opposite} implication holds as well. This fact must be known to the experts, but since we could not find it in the literature, we provide a
proof for the sake of completeness.

\begin{theorem}
    Let $G$ be a cyclic group $\Z/N$ for $N\ge 2$, the variety $Y$ is $\PP^1$ with the action of generator of $G$ given by $z\mapsto \exp(2\pi \mathsf i/N)z$. Then in the notations above, if the varieties $\mathcal X_i, i=1,2$ are birational over $\C(Y)^G$, then $X_1$ is $G$-birational to $X_2$ over $\C$.
\end{theorem}
\begin{proof}
    We will adapt the technique of Burnside rings and the specialization result from \cite{Kontsevich2019} to the $G$-equivariant setting. 

    The corresponding version of the $d$-graded component of the Burnside ring is defined as follows:
     we consider isomorphism classes of smooth $d$-dimensional schemes endowed with $G$-action (here $G$ is arbitrary finite group, not necessarily cyclic one), up to removing  closed subsets of dimension $\le (d-1)$. Such classes form a commutative monoid under disjoint union, and we define $\mathrm{Burn}_{d,G}$ to be the group associate with this monoid. Obviously, it is freely generated by the $G'$-birational equivalence classes in dimension $d$ where $G'\subset G$ runs through the set of conjugacy classes all subgroups of $G$. 
     
    Let us consider a flat $G$-equivariant proper morphism
    $\pi:Z\to Y$ with $N$-dimensional fibers, and such that the total space $Z$ is smooth, and the fiber over the fixed point $z=0\subset Y$ is a divisor $\cup_{i\in I} D_i$ with simple normal crossings. The group $G$ acts on the set $I$ of divisors. We associate with such a map a formal $\Z$-linear combination of birational types of $N$-dimensional varieties endowed with $G$-action. Namely for each non-empty subset $J\subset I$ such that $D_J:=\cap_{i\in I} D_i$ is nonempty, consider $N$-dimensional variety  $\widetilde D_J$ which is $D_i$ if $J=\{i\}$ for some $i\in I$, and the exceptional locus in the blow-up of $Z$ at $D_J$ if $\# J\ge 2$. 
    
    Clearly for each $k=1,\dots,d$ the disjoint union of varieties 
    $$D_{(k)}:=\bigsqcup_{\substack{J\subset I:\\\cap_{i\in J}D_i\ne\emptyset\\ \# J=k}}\widetilde D_J$$
    is endowed with a natural $G$-action.
    Then we define the specialization invariant as 
    $$\sum_{k=1}^d (-1)^{k-1} [D_{(k)}]\in \mathrm{Burn}_{d,G}\,.$$
The arguments from \cite{Kontsevich2019} work equally well in the equivariant case, and by the equivariant Weak Factorization we conclude that the above expression does not change under $G$-equivariant blow-ups of $Z$ at smooth centers intersecting cleanly the divisor $\sum_i D_i$. The invariant of the $G$-equivariant family $X_i\times Y\to Y,\,i=1,2$ for the diagonal $G$-action on $X_i\times Y$ is by definition the $G$-birational class of $X_i$. This implies the conclusion of the theorem.
\end{proof}
One can replace the field of rational functions  $\C(z)=\C(Y)$ by the non-archimedean field of Laurent series $\C(\!(z)\!)$ in the above result.

In general, we expect higher-dimensional versions of our result, and a parallel theory of \textit{specializations of atoms}. As the reader can see, the fixed loci of $G$-action on $Y$ play here a crucial role.

\bibliographystyle{alphabetic} 
\bibliography{main} 

@misc{cheltsov2025equivariantgeometrysingularcubic,
      title={{Equivariant geometry of singular cubic threefolds, II}}, 
      author={I. Cheltsov and L. Marquand and Y. Tschinkel and Z. Zhang},
      year={2025},
      eprint={2405.02744},
      archivePrefix={arXiv},
      primaryClass={math.AG},
      url={https://arxiv.org/abs/2405.02744}, 
}

@article{Kresch-Tschinke-Structure-and-Operations,
author = {A. Kresch and Y. Tschinkel},
title = {{Equivariant Burnside Groups: Structure and Operations}},
journal = {Taiwanese Journal of Mathematics},
publisher = {Mathematical Society of the Republic of China},
pages = {1 -- 23},
year = {2025},
doi = {10.11650/tjm/250304},
URL = {https://doi.org/10.11650/tjm/250304}
}

@misc{benedetti2026quantumcohomologyirrationalitygushelmukai,
      title={Quantum cohomology and irrationality of Gushel-Mukai fourfolds}, 
      author={Vladimiro Benedetti and Laurent Manivel and Nicolas Perrin},
      year={2026},
      eprint={2603.17487},
      archivePrefix={arXiv},
      primaryClass={math.AG},
      url={https://arxiv.org/abs/2603.17487}, 
}

@misc{guere2026irrationalitycubicfourfolds,
      title={On the irrationality of cubic fourfolds}, 
      author={J\'er\'emy Gu\'er\'e},
      year={2026},
      eprint={2603.04518},
      archivePrefix={arXiv},
      primaryClass={math.AG},
      url={https://arxiv.org/abs/2603.04518}, 
}

@misc{kresch2026invariantsequivariantbirationalgeometry,
      title={Invariants in equivariant birational geometry}, 
      author={Andrew Kresch and Yuri Tschinkel},
      year={2026},
      eprint={2602.23998},
      archivePrefix={arXiv},
      primaryClass={math.AG},
      url={https://arxiv.org/abs/2602.23998}, 
}

@InProceedings{Tschinkel-Yang,
author="Tschinkel, Y.
and Yang, K.",
editor="Nathanson, Melvyn B.",
title="{Potentially Stably Rational Del Pezzo Surfaces over Nonclosed Fields}",
booktitle="Combinatorial and Additive Number Theory III",
year="2020",
publisher="Springer International Publishing",
address="Cham",
pages="227--233",
isbn="978-3-030-31106-3"
}

@article{Kresch2022-Representation,
  title = {{Equivariant Burnside groups and representation theory}},
  volume = {28},
  ISSN = {1420-9020},
  url = {http://dx.doi.org/10.1007/s00029-022-00797-9},
  DOI = {10.1007/s00029-022-00797-9},
  number = {4},
  journal = {Selecta Mathematica},
  publisher = {Springer Science and Business Media LLC},
  author = {Kresch,  A. and Tschinkel,  Y.},
  year = {2022},
  month = sep 
}

@unpublished{CKK:atom_theory,
  author = {Cavenaghi, L. F. and Katzarkov, L. and Kontsevich, M.},
  title  = {{5 lectures in atom theory}},
  note   = {In preparation},
  year   = {2025}
}

@inproceedings{Batyrev:1997hj,
    author = "Batyrev, V. V.",
    title = "{Stringy hodge numbers of varieties with Gorenstein canonical singularities}",
    booktitle = "{Taniguchi Symposium on Integrable Systems and Algebraic Geometry}",
    pages = "1--32",
    year = "1997"
}

@article{Jarvis2005,
  title = {Pointed admissible G-covers and G-equivariant cohomological field theories},
  volume = {141},
  ISSN = {1570-5846},
  url = {http://dx.doi.org/10.1112/S0010437X05001284},
  DOI = {10.1112/s0010437x05001284},
  number = {04},
  journal = {Compositio Mathematica},
  publisher = {Wiley},
  author = {Jarvis,  T. J. and Kaufmann,  R. and Kimura,  T.},
  year = {2005},
  month = jun,
  pages = {926–978}
}

@article{iritani2024quantumcohomologyprojectivebundles,
      title={Quantum cohomology of projective bundles}, 
      author={H. Iritani and Y. Koto},
      year={2024},
      eprint={2307.03696},
      archivePrefix={arXiv},
      primaryClass={math.AG},
      url={https://arxiv.org/abs/2307.03696}, 
}

@article{baranovsky2003orbifold,
  title={Orbifold cohomology as periodic cyclic homology},
  author={Baranovsky, V.},
  journal={International Journal of Mathematics},
  volume={14},
  number={08},
  pages={791--812},
  year={2003},
  publisher={World Scientific}
}

@article{lupercio2004inertia,
  title={Inertia orbifolds, configuration spaces and the ghost loop space},
  author={Lupercio, E. and Uribe, B.},
  journal={The Quarterly Journal of Mathematics},
  volume={55},
  number={2},
  pages={185--201},
  year={2004},
  publisher={Oxford University Press}
}

@article{Kresch2022,
AUTHOR = {Kresch, A. and Tschinkel, Y.},
     TITLE = {Equivariant birational types and {B}urnside volume},
   JOURNAL = {Ann. Sc. Norm. Super. Pisa Cl. Sci. (5)},
  FJOURNAL = {Annali della Scuola Normale Superiore di Pisa. Classe di
              Scienze. Serie V},
    VOLUME = {23},
      YEAR = {2022},
    NUMBER = {2},
     PAGES = {1013--1052},
      ISSN = {0391-173X,2036-2145},
   MRCLASS = {14L30 (14E07 14M20)},
  MRNUMBER = {4453971},
MRREVIEWER = {Zinovy\ Reichstein},
       DOI = {10.2422/2036-2145.202011\_024},
       URL = {https://doi.org/10.2422/2036-2145.202011_024},
}

@Article{Chen2004,
author={W. Chen
and Y. Ruan},
title={A New Cohomology Theory of Orbifold},
journal={Communications in Mathematical Physics},
year={2004},
month={Jun},
day={01},
volume={248},
number={1},
pages={1-31},
issn={1432-0916},
doi={10.1007/s00220-004-1089-4},
url={https://doi.org/10.1007/s00220-004-1089-4}
}

@article{Kuznetsov2017,
  title = {{Calabi–Yau and fractional Calabi–Yau categories}},
  volume = {2019},
  ISSN = {0075-4102},
  url = {http://dx.doi.org/10.1515/crelle-2017-0004},
  DOI = {10.1515/crelle-2017-0004},
  number = {753},
  journal = {Journal f\"{u}r die reine und angewandte Mathematik (Crelles Journal)},
  publisher = {Walter de Gruyter GmbH},
  author = {Kuznetsov,  A.},
  year = {2017},
  month = mar,
  pages = {239–267}
}

@phdthesis{reid1972,
  author  = {Reid, M.},
  title   = {The complete intersection of two or more quadrics},
  school  = {Cambridge University},
  year    = {1972},
  type    = {Ph.D. dissertation}
}

@article{Kuznetsov_Prokhorov_2024, title={Rationality over nonclosed fields of Fano threefolds with higher geometric Picard rank}, volume={23}, DOI={10.1017/S1474748022000378}, number={1}, journal={Journal of the Institute of Mathematics of Jussieu}, author={Kuznetsov, A. and Prokhorov, Y.}, year={2024}, pages={207–247}}

@article{CSabbah,
     author = {Sabbah, C.},
     title = {Non-commutative {Hodge} structures},
     journal = {Annales de l'Institut Fourier},
     pages = {2681--2717},
     publisher = {Association des Annales de l{\textquoteright}institut Fourier},
     volume = {61},
     number = {7},
     year = {2011},
     doi = {10.5802/aif.2790},
     mrnumber = {3112504},
     language = {en},
     url = {http://www.numdam.org/articles/10.5802/aif.2790/}
}

@incollection{KatzarkovKontsevichPantev2008,
  author    = {L. Katzarkov and M. Kontsevich and T. Pantev},
  title     = {Hodge theoretic aspects of mirror symmetry},
  booktitle = {{From Hodge theory to integrability and TQFT tt*-geometry}},
  series    = {Proc. Sympos. Pure Math.},
  volume    = {78},
  pages     = {87--174},
  publisher = {Amer. Math. Soc.},
  address   = {Providence, RI},
  year      = {2008}
}

@InProceedings{Kontsevich-HMS,
author="Kontsevich, M.",
editor="Chatterji, S. D.",
title="Homological Algebra of Mirror Symmetry",
booktitle="Proceedings of the International Congress of Mathematicians",
year="1995",
publisher="Birkh{\"a}user Basel",
address="Basel",
pages="120--139",
isbn="978-3-0348-9078-6"
}

@article{kontsevich2021equivariant,
  author = {Kontsevich, M. and Pestun, V. and Tschinkel, Y.},
  title = {Equivariant birational geometry and modular symbols},
  journal = {Journal of the European Mathematical Society},
  volume = {25},
  number = {1},
  pages = {153--202},
  year = {2021},
  publisher = {European Mathematical Society - EMS - Publishing House GmbH},
  doi = {10.4171/jems/1170},
}

@article{Kontsevich2019,
author={Kontsevich, M.
and Tschinkel, Y.},
title={Specialization of birational types},
journal={Inventiones mathematicae},
year={2019},
month={Aug},
day={01},
volume={217},
number={2},
pages={415-432},
issn={1432-1297},
doi={10.1007/s00222-019-00870-9},
url={https://doi.org/10.1007/s00222-019-00870-9}
}

@article{Tschinkel2024,
  title = {Equivariant birational geometry of linear actions},
  volume = {11},
  ISSN = {2308-216X},
  url = {http://dx.doi.org/10.4171/EMSS/82},
  DOI = {10.4171/emss/82},
  number = {2},
  journal = {EMS Surveys in Mathematical Sciences},
  publisher = {European Mathematical Society - EMS - Publishing House GmbH},
  author = {Tschinkel,Y. and Yang, K. and Zhang,Z.},
  year = {2024},
  month = sep,
  pages = {235–276}
}

@article{Kontsevich1994,
  title = {{Gromov-Witten classes,  quantum cohomology,  and enumerative geometry}},
  volume = {164},
  ISSN = {1432-0916},
  url = {http://dx.doi.org/10.1007/BF02101490},
  DOI = {10.1007/bf02101490},
  number = {3},
  journal = {Communications in Mathematical Physics},
  publisher = {Springer Science and Business Media LLC},
  author = {Kontsevich,  M. and Manin,  Y.},
  year = {1994},
  month = aug,
  pages = {525–562}
}

@article{Tschinkel2022,
  title = {{Combinatorial Burnside groups}},
  volume = {8},
  ISSN = {2363-9555},
  url = {http://dx.doi.org/10.1007/s40993-022-00333-w},
  DOI = {10.1007/s40993-022-00333-w},
  number = {2},
  journal = {Research in Number Theory},
  publisher = {Springer Science and Business Media LLC},
  author = {Y. Tschinkel and K. Yang and Z. Zhang},
  year = {2022},
  month = may 
}

@article{kresch2019birational,
  title = {Birational types of algebraic orbifolds},
  volume = {212},
  ISSN = {1064-5616},
  url = {http://dx.doi.org/10.1070/sm9386},
  DOI = {10.1070/sm9386},
  number = {3},
  journal = {Sbornik: Mathematics},
  publisher = {Steklov Mathematical Institute},
  author = {Kresch,  A. and Tschinkel,  Y.},
  year = {2021},
  month = mar,
  pages = {319–331}
}

@article{kresch2023birational,
      title={{Birational geometry of Deligne-Mumford stacks}}, 
      author={A. Kresch and Y. Tschinkel},
      year={2023},
      eprint={2312.14061},
journal={arXiv eprint 2312.14061},
      archivePrefix={arXiv},
      primaryClass={math.AG}
}

@article{Bondal1990,
  title = {{Representable Functors, Serre Functors, and Mutations}},
  volume = {35},
  ISSN = {0025-5726},
  url = {http://dx.doi.org/10.1070/IM1990v035n03ABEH000716},
  DOI = {10.1070/im1990v035n03abeh000716},
  number = {3},
  journal = {Mathematics of the USSR-Izvestiya},
  publisher = {Steklov Mathematical Institute},
  author = {Bondal,  A. I. and Kapranov,  M. M.},
  year = {1990},
  month = jun,
  pages = {519–541}
}

@article{Bhning2015,
  title = {{Determinantal Barlow surfaces and phantom categories}},
  volume = {17},
  ISSN = {1435-9863},
  url = {http://dx.doi.org/10.4171/JEMS/539},
  DOI = {10.4171/jems/539},
  number = {7},
  journal = {Journal of the European Mathematical Society},
  publisher = {European Mathematical Society - EMS - Publishing House GmbH},
  author = {B\"{o}hning,  C. and G. von Bothmer,  H-C and Katzarkov,  L. and Sosna,  P.},
  year = {2015},
  month = jun,
  pages = {1569–1592}
}

@article{AbramovichKaruMatsukiWlodarczyk2002,
    author    = {Abramovich, D. and Karu, K. and Matsuki, K. and Włodarczyk, J.},
    title     = {Torification and factorization of birational maps},
    journal   = {Journal of the American Mathematical Society},
    volume    = {15},
    number    = {3},
    pages     = {531--572},
    year      = {2002},
    doi       = {10.1090/S0894-0347-02-00396-X}
}

@article{Wlodarczyk2000,
    author    = {W\l odarczyk, J.},
    title     = {Birational cobordisms and factorization of birational maps},
    journal   = {Journal of Algebraic Geometry},
    volume    = {9},
    number    = {3},
    pages     = {425--449},
    year      = {2000}
}

@BOOK{Tabuada2015-tu,
  title     = "Noncommutative Motives",
  author    = "Tabuada, G.",
  publisher = "American Mathematical Society",
  year      =  2015
}

@Article{Blum2021,
author={Blum, H.
and Halpern-Leistner, D.
and Liu, Y.
and Xu, C.},
title={{On properness of K-moduli spaces and optimal degenerations of Fano varieties}},
journal={Selecta Mathematica},
year={2021},
month={Jul},
day={28},
volume={27},
number={4},
pages={73},
issn={1420-9020},
doi={10.1007/s00029-021-00694-7},
url={https://doi.org/10.1007/s00029-021-00694-7}
}

@incollection{Kontsevich2008,
  title = {{Notes on $A_{\infty}$-Algebras,  $_{\infty}$-Categories and Non-Commutative Geometry}},
  ISBN = {9783540680307},
  ISSN = {1616-6361},
  url = {http://dx.doi.org/10.1007/978-3-540-68030-7_6},
  DOI = {10.1007/978-3-540-68030-7_6},
  booktitle = {Homological Mirror Symmetry},
  publisher = {Springer Berlin Heidelberg},
  author = {Kontsevich,  M. and Soibelman,  Y.},
  year = {2008},
  pages = {1–67}
}

@book{Voisin_2002, place={Cambridge}, series={Cambridge Studies in Advanced Mathematics}, title={{Hodge Theory and Complex Algebraic Geometry I}}, publisher={Cambridge University Press}, author={Voisin, C.}, editor={Schneps, LeilaTranslator}, year={2002}, collection={Cambridge Studies in Advanced Mathematics}}

@article{Givental,
    author = {Givental, A. B.},
    title = {Equivariant Gromov-Witten invariants},
    journal = {International Mathematics Research Notices},
    volume = {1996},
    number = {13},
    pages = {613-663},
    year = {1996},
    issn = {1073-7928},
    doi = {10.1155/S1073792896000414},
    url = {https://doi.org/10.1155/S1073792896000414},
    eprint = {https://academic.oup.com/imrn/article-pdf/1996/13/613/6768535/1996-13-613.pdf},
}

@incollection{kresch2022burnside,
author={Kresch, A.
and Tschinkel, Y.},
editor={Albano, Alberto
and Aluffi, Paolo
and Bolognesi, Michele
and Casagrande, Cinzia
and Colombo, Elisabetta
and Conte, Alberto
and Grassi, Antonella
and Pedrini, Claudio
and Pirola, Gian Pietro
and Verra, Alessandro},
title={{Burnside Groups and Orbifold Invariants of Birational Maps}},
bookTitle={{Perspectives on Four Decades of Algebraic Geometry, Volume 2: In Memory of Alberto Collino}},
year={2025},
publisher={Springer Nature Switzerland},
address={Cham},
pages={433--446},
isbn={978-3-031-66234-8},
doi={10.1007/978-3-031-66234-8_11},
url={https://doi.org/10.1007/978-3-031-66234-8_11}
}

@book{bira-book-2023,
  title = {Birational Geometry,  Rational Curves,  and Arithmetic},
SERIES = {Simons Symposia},
    EDITOR = {Bogomolov, F. and Hassett, B. and Tschinkel, Y.},
 PUBLISHER = {Springer, Cham},
  ISBN = {9781461464822},
  url = {http://dx.doi.org/10.1007/978-1-4614-6482-2},
  DOI = {10.1007/978-1-4614-6482-2},
  year = {2013},
}

@book{bira-book-2020,
 title = {Birational Geometry and Moduli Spaces},
  ISBN = {9783030371142},
  ISSN = {2281-5198},
SERIES = {Springer INdAM Series},
    VOLUME = {39},
    EDITOR = {Colombo, E. and Fantechi, B. and Frediani, P.
              and Iacono, D. and Pardini, R.},
      NOTE = {Papers from the INdAM Workshop held in Rome, June 11-15, 2018},
  url = {http://dx.doi.org/10.1007/978-3-030-37114-2},
  DOI = {10.1007/978-3-030-37114-2},
  journal = {Springer INdAM Series},
  publisher = {Springer International Publishing},
  year = {2020},
}

@article{behrend-fantechi1997,
  author    = {K. Behrend and B. Fantechi},
  title     = {The intrinsic normal cone},
  journal   = {Inventiones mathematicae},
  year      = {1997},
  volume    = {128},
  number    = {1},
  pages     = {45--88},
  doi       = {10.1007/s002220050138}
}

@article{li-tian1998,
  author    = {J. Li and G. Tian},
  title     = {Virtual moduli cycles and Gromov–Witten invariants of algebraic varieties},
  journal   = {Journal of the American Mathematical Society},
  year      = {1998},
  volume    = {11},
  number    = {1},
  pages     = {119--174},
  doi       = {10.1090/S0894-0347-98-00256-4}
}

@article{birkar2017birationalgeometryalgebraicvarieties,
      title={Birational geometry of algebraic varieties}, 
      author={Birkar,C.},
      year={2017},
      eprint={1801.00013},
      archivePrefix={arXiv},
      primaryClass={math.AG},
      url={https://arxiv.org/abs/1801.00013}, 
}

@article{birkar2012lecturesbirationalgeometry,
      title={Lectures on birational geometry}, 
      author={Birkar,C.},
      year={2012},
      eprint={1210.2670},
      archivePrefix={arXiv},
      primaryClass={math.AG},
      url={https://arxiv.org/abs/1210.2670}, 
}

@article{Prokhorov,
author = {Prokhorov, Y.},
title = {{On stable conjugacy of finite subgroups of the plane Cremona group, II}},
volume = {64},
journal = {Michigan Mathematical Journal},
number = {2},
publisher = {University of Michigan, Department of Mathematics},
pages = {293 -- 318},
year = {2015},
doi = {10.1307/mmj/1434731925},
URL = {https://doi.org/10.1307/mmj/1434731925}
}

@article{Bogomolov2013,
  title = {On stable conjugacy of finite subgroups of the plane Cremona group,  {I}},
  volume = {11},
  ISSN = {2391-5455},
  url = {http://dx.doi.org/10.2478/s11533-013-0314-9},
  DOI = {10.2478/s11533-013-0314-9},
  number = {12},
  journal = {Open Mathematics},
  publisher = {Walter de Gruyter GmbH},
  author = {Bogomolov, F. and Prokhorov, Y.},
  year = {2013},
  month = jan 
}

@book{Voisin2019,
author="Voisin, C.",
editor="Hochenegger, Andreas
and Lehn, Manfred
and Stellari, Paolo",
title="Birational Invariants and Decomposition of the Diagonal",
bookTitle="Birational Geometry of Hypersurfaces: Gargnano del Garda, Italy, 2018",
year="2019",
publisher="Springer International Publishing",
address="Cham",
pages="3--71",
isbn="978-3-030-18638-8",
doi="10.1007/978-3-030-18638-8_1",
url="https://doi.org/10.1007/978-3-030-18638-8_1"
}

@article{katzarkovpantevyu,
      title={Birational Invariants from {Hodge Structures and Quantum Multiplication}}, 
      author={L. Katzarkov and M. Kontsevich and T. Pantev and T. Yue YU},
      year={2025},
      eprint={2508.05105},
      archivePrefix={arXiv},
      primaryClass={math.AG},
      url={https://arxiv.org/abs/2508.05105}, 
}

@article{Voisin2014,
  title = {{Unirational threefolds with no universal codimension 2 cycle}},
  volume = {201},
  ISSN = {1432-1297},
  url = {http://dx.doi.org/10.1007/s00222-014-0551-y},
  DOI = {10.1007/s00222-014-0551-y},
  number = {1},
  journal = {Inventiones mathematicae},
  publisher = {Springer Science and Business Media LLC},
  author = {Voisin, C.},
  year = {2014},
  month = oct,
  pages = {207–237}
}

@article{Voisin2016StableBI,
  title={{Stable birational invariants and the L{\"u}roth problem}},
  author={Voisin, C.},
  journal={Surveys in differential geometry},
  year={2016},
  volume={21},
  pages={313-342},
  url={https://api.semanticscholar.org/CorpusID:15445320}
}

@misc{ElaginSchneiderShinder,
      title={Atomic decompositions for derived categories of {G}-surfaces}, 
      author={A. Elagin and J. Schneider and E. Shinder},
      year={2025},
      eprint={2512.05064},
      archivePrefix={arXiv},
      primaryClass={math.AG},
      url={https://arxiv.org/abs/2512.05064}, 
}

@article{cheltsov2024equivariant,
      title={Equivariant geometry of singular cubic threefolds}, 
      author={Cheltsov, I. and Tschinkel, Y. and Zhang, Z.},
      year={2024},
      eprint={2401.10974},
journal={arXiv eprint 2401.10974},
      archivePrefix={arXiv},
      primaryClass={id='math.AG' full_name='Algebraic Geometry' is_active=True alt_name='alg-geom' in_archive='math' is_general=False description='Algebraic varieties, stacks, sheaves, schemes, moduli spaces, complex geometry, quantum cohomology'}
}

@article{Hassett2022,
  title = {Equivariant geometry of odd-dimensional complete intersections of two quadrics},
  volume = {18},
  ISSN = {1558-8602},
  url = {http://dx.doi.org/10.4310/PAMQ.2022.v18.n4.a8},
  DOI = {10.4310/pamq.2022.v18.n4.a8},
  number = {4},
  journal = {Pure and Applied Mathematics Quarterly},
  publisher = {International Press of Boston},
  author = {Hassett,  B. and Tschinkel,  Y.},
  year = {2022},
  pages = {1555–1597}
}

@misc{kresch2025intermediatejacobiansburnsideinvariants,
      title={{Intermediate Jacobians and Burnside invariants}}, 
      author={A. Kresch and S. Tanimoto and Y. Tschinkel},
      year={2025},
      eprint={2511.07101},
      archivePrefix={arXiv},
      primaryClass={math.AG},
      url={https://arxiv.org/abs/2511.07101}, 
}

@article{iritani2023quantumcohomologyblowups,
      title={Quantum cohomology of blowups}, 
      author={H. Iritani},
      year={2023},
journal={arXiv eprint 2307.13555},
      eprint={2307.13555},
      archivePrefix={arXiv},
      primaryClass={math.AG},
      url={https://arxiv.org/abs/2307.13555}, 
}

@article{bryan2007crepantresolutionconjecture,
      title={The Crepant Resolution Conjecture}, 
      author={J. Bryan and T. Graber},
      year={2007},
journal={arXiv eprint math/0610129},
      eprint={math/0610129},
      archivePrefix={arXiv},
      primaryClass={math.AG},
      url={https://arxiv.org/abs/math/0610129}, 
}

@article{Reichstein2002,
  title = {A birational invariant for algebraic group actions},
  volume = {204},
  ISSN = {0030-8730},
  url = {http://dx.doi.org/10.2140/pjm.2002.204.223},
  DOI = {10.2140/pjm.2002.204.223},
  number = {1},
  journal = {Pacific Journal of Mathematics},
  publisher = {Mathematical Sciences Publishers},
  author = {Reichstein,  Z. and Youssin,  B.},
  year = {2002},
  month = may,
  pages = {223–246}
}

\end{document}